\documentclass[reqno, 10pt]{amsart}

\usepackage[top=4cm, bottom=2.7cm, left=3cm, right=3cm]{geometry}

\usepackage[centertags]{amsmath}
\usepackage{amsmath}
\usepackage{amsfonts}
\usepackage{amssymb}
\usepackage{amsthm}
\usepackage{authblk}
\usepackage[dutch,british]{babel}
\usepackage{bbm}
\usepackage{cleveref}
\usepackage{color}
\usepackage{fancybox}
\usepackage{fancyhdr}
\usepackage{graphicx}
\usepackage{mathrsfs}
\usepackage{newlfont}
\usepackage{setspace}
\usepackage{stmaryrd}

\theoremstyle{plain} 
\newtheorem{theorem}{Theorem}[section]
\newtheorem*{theorem*}{Theorem}
\newtheorem{lemma}[theorem]{Lemma}

\newtheorem*{corollary*}{Corollary}
\newtheorem{proposition}[theorem]{Proposition}
\newtheorem*{proposition*}{Proposition}
\newtheorem{definition}[theorem]{Definition}
\newtheorem*{definition*}{Definition}
\newtheorem{assumption}[theorem]{Assumption}

\theoremstyle{definition} 
\newtheorem{example}[theorem]{Example}
\newtheorem*{example*}{Example}
\newtheorem{remark}[theorem]{Remark}

\newtheorem*{remark*}{Remark}
\newtheorem*{remarks*}{Remarks}

\newcommand{\deq}{\mathrel{\mathop:}=}
\newcommand{\e}[1]{\mathrm{e}^{#1}}
\newcommand{\R} {\mathbb{R}}
\newcommand{\C} {\mathbb{C}}
\newcommand{\N} {\mathbb{N}}
\newcommand{\Z} {\mathbb{Z}}

\newcommand{\E} {\mathbb{E}}
\newcommand{\T} {\mathbb{T}}
\newcommand{\p} {\mathbb{P}}

\DeclareMathOperator{\diag}{diag}

\DeclareMathOperator{\Tr}{Tr}

\DeclareMathOperator{\supp}{supp}

\DeclareMathOperator{\re}{\mathrm{Re}}
\DeclareMathOperator{\im}{\mathrm{Im}}

\newcommand{\caD}{{\mathcal D}}
\newcommand{\caE}{{\mathcal E}}

\newcommand{\caN}{{\mathcal N}}
\newcommand{\caO}{{\mathcal O}}

\newcommand{\caX}{{\mathcal X}}

\newcommand{\opunit}{\text{1}\kern-0.22em\text{l}}

\newcommand{\frn}{{\mathfrak n}}

\newcommand{\frX}{{\mathfrak X}}

\newcommand{\wt}{\widetilde}

\newcommand{\beq}{ \begin{equation} }
\newcommand{\eeq}{ \end{equation} }

\newcommand{\baq}{ \begin{eqnarray} }
\newcommand{\eaq}{ \end{eqnarray} }
\newcommand{\bet}{ \begin{theorem} }
\newcommand{\eet}{ \end{theorem} }

\newcommand{\GOE}{\textrm{GOE}}

\newcommand{\lone}{\mathbbm{1}}

\newcommand{\ben}{\begin{arabicenumerate}}
\newcommand{\een}{\end{arabicenumerate}}

\newcommand{\dd}{\mathrm{d}}
\newcommand{\ii}{\mathrm{i}}

\newcommand{\w}{\mathrm{W}}

\renewcommand{\b}{\mathrm{b}}
\renewcommand{\P}{\mathbb{P}}

\numberwithin{equation}{section} 
\numberwithin{theorem}{section}

\begin{document}
 \begin{minipage}{0.85\textwidth}
 \vspace{3cm}
 \end{minipage}
\begin{center}
\LARGE\bf
Edge Universality for Deformed Wigner Matrices
\end{center}

\renewcommand{\thefootnote}{\fnsymbol{footnote}}	
\vspace{1.5cm}
\begin{center}
 \begin{minipage}{0.3\textwidth}
\begin{center}
Ji Oon Lee\footnotemark[1]  \\
\footnotesize { KAIST }\\
{\it jioon.lee@kaist.edu}
\end{center}
\end{minipage}
\begin{minipage}{0.3\textwidth}
 \begin{center}
Kevin Schnelli\footnotemark[2]\\
\footnotesize 
{IST Austria}\\
{\it kevin.schnelli@ist.ac.at}
\end{center}
\end{minipage}
 \footnotetext[1]{Partially supported by National Research Foundation of Korea Grant 2011-0013474 and TJ Park Junior Faculty Fellowship.}
\footnotetext[2]{Supported by ERC Advanced Grant RANMAT, No. 338804. }
\end{center}

\vspace{1.5cm}

\begin{center}
 \begin{minipage}{0.8\textwidth}
\indent
\small

We consider $N\times N$ random matrices of the form $H = W + V$ where $W$ is a real symmetric Wigner matrix and $V$ a random or deterministic, real, diagonal matrix whose entries are independent of~$W$. We assume subexponential decay for the matrix entries of~$W$ and we choose~$V$ so that the eigenvalues of $W$ and $V$ are typically of the same order. For a large class of diagonal matrices~$V$ we show that the rescaled distribution of the extremal eigenvalues is given by the Tracy-Widom distribution $F_1$ in the limit of large $N$. Our proofs also apply to the complex Hermitian setting, i.e., when $W$ is a complex Hermitian Wigner matrix.
 \end{minipage}
\end{center}
 
 \vspace{4mm}
 
{\small
 \vspace{1mm}
 
 }
  \vspace{1.5cm}

\section{Introduction}
It is widely believed that the behavior of the extremal eigenvalues of many random matrix ensembles is universal. This edge universality has been established for a large class of Wigner matrices: Let $\mu_1$ denote the largest eigenvalue of a Wigner matrix of size~$N$. The limiting distribution of $\mu_1$ was identified for the Gaussian ensembles by Tracy and Widom~\cite{TW1, TW2}. They proved that
\begin{align}\label{prototype}
\lim_{N\to\infty} \mathbb{P}(N^{2/3}(\mu_1 -2) \le s) = F_\beta (s)\,,\qquad\qquad(\beta\in\{1,2\})\,,
\end{align}
$s\in\R$, where the Tracy-Widom distribution functions $F_{\beta}$ are described by Painlev\'{e} equations. The choice of $\beta=1,2$ corresponds to the Gaussian Orthogonal/Unitary ensemble (GUE/GOE).  The edge universality can also be extended to the $k$ largest eigenvalues, where the joint distribution of the $k$ largest eigenvalues can be written in terms of the Airy kernel, as first shown for the GUE/GOE in~\cite{F}. These results also hold for the $k$ smallest eigenvalues.

 Edge universality for Wigner matrices was first proved in~\cite{So1} (see also~\cite{SiSo1}) for real symmetric and complex Hermitian ensembles with symmetric distributions. The symmetry assumption on the entries' distribution was partially removed in~\cite{PeSo1,PeSo2}. Edge universality without any symmetry assumption was proved in~\cite{TV2} under the condition that the distribution of the matrix elements has subexponential decay and its first three 
moments match those of the Gaussian distribution, i.e., the third moment of the entries vanish. The vanishing third moment condition was removed in~\cite{EYY}. Recently, a necessary and sufficient condition on the entries' distribution for the edge universality of Wigner matrices was given in~\cite{LY}.

In the present paper, we establish edge universality for deformed Wigner matrices: A deformed Wigner matrix, $H$, is an $N\times N$ random matrix of the form 
\begin{align}\label{la matrice}
H =\lambda V + W \,,\qquad\qquad (\lambda\in\R)\,,
\end{align}
where $V$ is a real, diagonal, random or deterministic matrix and $W$ is a real symmetric or complex Hermitian Wigner matrix independent of $V$. The matrices are normalized so that the eigenvalues
of~$V$ and~$W$ are order one. The ``coupling'' constant $\lambda\in\R$ may depend on~$N$, yet we will always assume that $\lambda$ remains finite in the limit of large~$N$. If the entries of $V$ are random we may think of $V$ as a ``random potential''; if the entries of $V$ are deterministic, matrices of the form~\eqref{la matrice} are sometimes referred to as ``Wigner matrices with external source''. For $W$ belonging to the GUE/GOE, the model~\eqref{la matrice} is often called the deformed GUE/GOE.

Assuming that the empirical eigenvalue distribution of $V=\diag(v_1,\ldots,v_N)$, 
\begin{align}\label{empirical V}
\widehat\nu\deq\frac{1}{N}\sum_{i=1}^N\delta_{v_i}\,, 
\end{align}
converges weakly, respectively weakly in probability, to a non-random measure, $\nu$, it was shown in~\cite{P} that the empirical distribution of the eigenvalues of $H$ converges weakly in probability to a deterministic measure which we refer to as the deformed semicircle law, $\rho_{fc}$. The deformed semicircle law $\rho_{fc}$ depends on $\nu$ and $\lambda$, and is thus in general distinct from Wigner's semicircle law. For many choices of $\nu$, however, the deformed semicircle law $\rho_{fc}$ has compact support and, similar to the standard semicircle law, exhibits a square-root type behavior at the endpoints of its support (see Lemma~\ref{square_root} for the precise statement). This suggests that the typical eigenvalue spacing at the spectral edge is of order~$N^{-2/3}$ as in the Wigner case and that the edge universality holds in the following sense. We assume that $V$ is such that all eigenvalues of $H$ stick to the support of the measure $\rho_{fc}$, i.e., that there are no ``outliers'' in the limit 
of large~$N$. We further assume for simplicity that $\rho_{fc}$ is supported on a single interval. Then the edge 
universality for deformed Wigner matrices states that there are $\gamma_0\equiv\gamma_0(N)$ and $\widehat E_+\equiv\widehat E_+(N)$, such that the limiting distribution of the largest eigenvalue~$\mu_1$ of~$H$ satisfies
\begin{align}\label{prototype2}
\lim_{N\to\infty} \mathbb{P}(\gamma_0N^{2/3}(\mu_1 -\widehat E_+) \le s) = F_\beta (s)\,,
\end{align}
where $\gamma_0>0$ and $\widehat E_+\in\R$ solely depend on $\widehat\nu\equiv\widehat\nu(N)$ defined in~\eqref{empirical V} and the coupling constant $\lambda$. Further, $\widehat E_+$ converges in probability to the upper endpoint, $E_+$, of the deformed semicircle law $\rho_{fc}$. The scaling factor $\gamma_0$ is order one and guarantees that the typical eigenvalue spacings at the edge of the rescaled matrix $\gamma_0 H$ match those of the GOE/GUE up to negligible errors.

The deformed GUE for the special case when~$V$ has two eigenvalues~$\pm a$, each with equal multiplicity, has been treated in a series of papers~\cite{BK1,BK2,BK3}. In this setting the local eigenvalue statistics at the edge can be obtained via the solution to a Riemann-Hilbert problem; see also~\cite{CW} for the case when~$V$ has equispaced eigenvalues. For general $V$, the joint distribution of the eigenvalues of the deformed GUE can be expressed explicitly by the Brezin-Hikami/Johansson formula that admits an asymptotic analysis of the distribution of the extremal eigenvalues for various choices of $V$ and ranges of~$\lambda$; see~\cite{J2,S1,CP}. Once the edge universality has been established for the deformed GUE, it may be extended to complex Hermitian deformed Wigner matrices by appropriate modifications of the comparison methods introduced in~\cite{TV2} and in~\cite{EYY}. However, if the matrix $W$ is real symmetric there is no explicit formula for the joint distribution of the eigenvalues available 
and the methods referred to above cannot be used to identify the Tracy-Widom distribution~$F_{1}$ in the real symmetric setting.

In the present paper, we establish the edge universality for real symmetric deformed Wigner matrices for a large class of $V$ and wide ranges of $\lambda$; see Theorem~\ref{thm main}. In particular, we identify the Tracy-Widom distributions~$F_1$ as the limiting distributions of the extremal eigenvalues. Our proof also applies with minor modifications to the complex Hermitian setting, i.e., when $W$ is a complex Hermitian matrix.

For the special case when the entries of $V$ are independent and identically distributed (i.i.d.) random variables, for simplicity assumed to be bounded, we find that the limiting distribution of the  largest rescaled eigenvalue of $H$ is given by the convolution of the Tracy-Widom distribution and a centered Gaussian distribution with appropriately chosen variance depending on $\lambda$: The relative size of the Tracy-Widom part and the Gaussian part depends on the coupling constant $\lambda$; the Gaussian part is negligible when $\lambda \ll N^{-1/6}$, whereas the Tracy-Widom component is dominated by the Gaussian if $\lambda \gg N^{-1/6}$. The transition from the Tracy-Widom to the Gaussian occurs at $\lambda\sim N^{-1/6}$ as was pointed out first in~\cite{J2} for the deformed GUE.  Yet, the law of the eigenvalue spacing at the spectral edge is solely determined by the Tracy-Widom distribution for all finite~$\lambda$. (See Theorem \ref{thm2} for more detail.)

The main difficulty of the proof of our main results~Theorem~\ref{thm main} and Theorem~\ref{thm2} lies in the proof of the Green function comparison theorem, Proposition \ref{prop green}. The Green function comparison method has shown to be very successful in proving the edge universality of Wigner matrices. However, the direct application of the conventional Green function comparison uses Lindeberg's replacement strategy, which does not work for deformed Wigner matrices mainly due to the large diagonal elements. Simply put, as in the framework of the four-moment theorem in \cite{TV1,TV2}, the usual method requires that the change of the averaged Green function from each replacement is $o(N^{-2})$, which is negligible since the number of such replacement is $O(N^2)$. On the other hand, for the deformed Wigner matrices with $\lambda \sim 1$, the replacement in the diagonal element causes an $O(1)$ change in the averaged Green function, which is too large a change if the number of replacement steps is $O(N)$.

The main novelty of the present paper is a new approach to the Green function comparison theorem via Dyson Brownian motion (DBM). We estimate the change of the Green function along the flow of the DBM, which interpolates between the deformed Wigner matrix and the corresponding Gaussian ensemble. In other words, instead of converting the given random matrix entry by entry, we change all entries simultaneously, continuously. (See Section~\ref{sec:dbm} for more detail.) The continuity of the DBM was used in~\cite{BY} to compare the local eigenvalue statistics along the flow of the DBM for very short times. In our proof we follow the flow of the Green function over a time interval of order $\log N$ during which it undergoes a change of order one. The continuous changes in the Green function can then be compensated by rescaling or ``renormalizing'' the matrix and the spectral parameter of the Green function. Such a proof of the Green function comparison requires, for $\lambda\sim 1$, some non-trivial 
estimates on functions of Green functions as is explained in Section~\ref{sec:fluctuation}. (See, e.g.,~\eqref{3 ward} for such an estimate referred to as an ``optical theorem'' below.) For $\lambda=0$, the presented method also yields, based on estimates in~\cite{EYY}, a streamlined proof of the edge universality for Wigner matrices. (See Section~\ref{sec:dbm}.) For brevity we carry out the proof for real symmetric deformed Wigner matrices only, but the proof also applies with minor modifications to complex Hermitian deformed Wigner matrices. 

Edge universality for deformed Wigner matrices may alternatively be studied via the local ergodicity of the DBM~\cite{ESY4,ESYY}. This approach has been followed in~\cite{BEY} to prove the edge universality for generalized Wigner matrices.  A basic ingredient of that proof is a global entropy estimate whose analogue version has been established for deformed Wigner matrices in Proposition~5.3 of~\cite{LSSY} for some choices of $V$. Relying on this estimate, one can prove edge universality for deformed Winger matrices following the lines of~\cite{BEY} (see Remark~2.9 in~\cite{LSSY}). The advantage of the method presented in the present paper is that it is purely local: the only technical input is the local deformed semicircle law, i.e., estimates on the Green function on scale $N^{-2/3}$. (See Theorem~\ref{local law} below.) Local laws for the deformed ensemble have been established in~\cite{LS1,LSSY}. However, in the proof presented in this paper these estimates are only needed at the edge of the spectrum and no further a priori control of 
the eigenvalues or Green function away from the edge is required. In particular, the method can also be used to study the extremal eigenvalues in a multi cut regime where the eigenvalues' limiting distribution is supported on several disjoint intervals. In such a setting the corresponding global entropy estimates in Proposition 5.3 of~\cite{LSSY} were rather difficult to obtain. Another advantage of the method of the present paper is that it does not require that the eigenvalues evolve autonomously under DBM, i.e., that the stochastic differential equations for the eigenvalue and the eigenvectors decouple under DBM. The method can therefore also be applied to matrix ensembles for which the eigenvalues do not evolve autonomously under DBM.

This paper is organized as follows: In Section~\ref{sec:main}, we define the model precisely and introduce the main results of the paper. In Sections~\ref{sec:prelim} and~\ref{sec:renormalize}, we collect the tools and known results we need in the proof of the main results. In Section~\ref{sec:proof of main results}, we prove the main theorems using the Green function comparison theorem. In Sections~\ref{sec:dbm}-\ref{sec:proof of prop}, we explain the proof of the Green function comparison theorem. While the main ideas of the proof are rather nice and pleasant, the details of the proof of the Green function theorem include long explicit, but elementary, computations that can be found in the Appendices.

{\it Acknowledgements:} We thank Horng-Tzer Yau for numerous helpful discussions and remarks. We are also
grateful to Paul Bourgade and L\'aszl\'o Erd\H{o}s for helpful comments.  We are grateful to Thomas Spencer for hospitality at the IAS where a major part of
this research was conducted.

\section{Definition and Main Result} \label{sec:main}

\subsection{Definition of the model}

\begin{definition} \label{assumption wigner}
Let $W$ be an $N\times N$ random matrix, whose entries $(w_{ij})$ are independent, up to the symmetry constraint $w_{ij}= {w_{ji}}$, centered real random variables. We assume that the random variables $(w_{ij})$ have variance $1/N$ and have finite moments, uniformly in~$N$, $i$ and $j$. More precisely, we assume that for each $p\in\N$ there is a constant $c_p$ such that
\beq\label{eq.C0}
 \E w_{ij}=0\,,\qquad\quad \E w_{ij}^2 = \frac{1+ c_2\delta_{ij}}{N}\,, \qquad\quad \E|w_{ij}|^p \leq \frac{c_p}{N^{p/2}}\,,\qquad (p\ge 3)\,.
\eeq
In case $(w_{ij})$ are Gaussian random variables with $c_2=1$, $W$ belongs to the Gaussian orthogonal ensemble (GOE).
\end{definition}

Let $V=\mathrm{diag}(v_i)$ be an $N\times N$ diagonal, random or deterministic matrix, whose entries $(v_i)$ are real-valued. We denote by $\widehat\nu$ the empirical eigenvalue distribution of the diagonal matrix~$V=\mathrm{diag}(v_i)$,
\begin{align}\label{widehatmuV}
 \widehat\nu\deq\frac{1}{N}\sum_{i=1}^N\delta_{v_i}\,.
\end{align}

\begin{assumption}\label{assumption mu_V convergence}
There is a (non-random) compactly supported probability measure $\nu$ and strictly positive constants~$\alpha_0$ and~$\beta_0$ such that the following holds. For any compact set $\caD\subset\C^+$ with $\mathrm{dist}(\caD,\supp\nu)>0$, there is~$C$ such that 
\begin{align}\label{equation assumption mu_V convergence}
 \P\left(\max_{z\in \caD}\left|\int\frac{\dd\widehat\nu(v)}{v-z}-\int\frac{\dd\nu(v)}{v-z} \right|\le CN^{-\alpha_0}\right)\ge 1-N^{-\beta_0} \,,
\end{align}
for $N$ sufficiently large.
\end{assumption}
Note that~\eqref{equation assumption mu_V convergence} implies that $\widehat\nu$ converges to $\nu$ in the weak sense as $N\to\infty$. Also note that the condition~\eqref{equation assumption mu_V convergence} holds for any $0<\alpha_0<1/2$ and any $\beta_0>0$ if $(v_i)$ are i.i.d.\ random variables.

We define the deformed Wigner matrix ensemble as follows:
\begin{definition} \label{assumption deformed}
A deformed Wigner matrix of size $N$ is an $N\times N$ symmetric random matrix $H$ that can be decomposed into
\beq \label{thematrix}
H = (h_{ij}) \deq \lambda_0 V + W\,,
\eeq
where $W$ is a real symmetric Wigner matrix of size $N$ and $V=\diag(v_i)$ is an $N\times N$ real diagonal matrix. The entries of $V$ can be random or deterministic. In case $V$ is random, we assume that $(v_i)$ are independent of $(w_{ij})$, yet $(v_i)$ need not be independent among themselves. Finally, $\lambda_0\ge 0$ is a finite coupling constant. 
\end{definition}

Our second assumption on $\widehat\nu$, $\nu$ and $\lambda_0$ guarantees that the limiting eigenvalue distribution of $H$ is supported on a single interval and has a square root behavior at the two endpoints of the support. Sufficient conditions for this behavior have been presented in~\cite{S2}. The assumption below also rules out the possibility that the matrix $H$ has ``outliers'' in the limit of large $N$. 
\begin{assumption}\label{assumption mu_V}
Let $I_{\nu}$ be the smallest closed interval such that $\supp{\nu}\subseteq I_{\nu}$. Then, there exists $\varpi>0$, independent of $N$, such that
\begin{align}\label{eq assumption mu_V}
\inf_{x\in I_{\nu}}\int\frac{\dd\nu(v)}{(v-x)^2}\ge (1+\varpi)\lambda_0^2\,.
\end{align}
Moreover, let $I_{\widehat\nu}$ be the smallest closed interval such that $\supp{\widehat\nu}\subseteq I_{\widehat\nu}$. Then,
we assume that there is a constant $\beta_1>0$, such that
\begin{align}\label{eq assumption mu_V 2}
\mathbb{P}\left(\inf_{x\in I_{\widehat\nu}}\int\frac{\dd\widehat\nu(v)}{(v-x)^2}\ge (1+{\varpi})\lambda_0^2 \right)\ge 1-N^{-\beta_1}\,,
\end{align}
for $N$ sufficiently large.
\end{assumption}
\begin{remark}
The left side of~\eqref{eq assumption mu_V} may be infinite. In this case~\eqref{eq assumption mu_V} should be understood in the sense that~$\lambda_0$ can be chosen as any finite positive number (independent of $N$).  Note that if~\eqref{eq assumption mu_V} is satisfied for some~$\lambda\equiv\lambda_0$ and~$\nu$, then it is also satisfied for all~$\lambda_0\in[0,\lambda]$ for this~$\nu$. 
\end{remark}
\begin{remark}
The coupling constant $\lambda_0$ can be chosen to depend on $N$, as long as it stays bounded and converges sufficiently fast in the limit of large $N$. To simplify the exposition we only consider the case $\lambda_0= \sigma_0 N^{-\delta}$, for some constants $\delta\ge 0$ and $\sigma_0\ge 0$, below. 
\end{remark}

\noindent We give two examples for which Assumption~\ref{assumption mu_V} is satisfied: We choose $(v_i)$ to be i.i.d.\ random variables with law $\nu$.
\begin{itemize}
\item [$(1)$] Choosing $\nu=\frac{1}{2}(\mathrm{\delta}_{-1}+\mathrm{\delta}_{1})$, $\lambda\ge 0$, we have $I_{\nu}=[-1,1]$. For $\lambda<1$, one checks that there exist~$\varpi$ and $\beta_1>0$ such that~\eqref{eq assumption mu_V} and~\eqref{eq assumption mu_V 2} are satisfied and that the deformed semicircle law is supported on a single interval with a square root type behavior at the edges. However, in case $\lambda>1$, the deformed semicircle law is supported on two disjoint intervals. For more details see~\cite{BK1,BK2,BK3}.

\item[$(2)$] Let $\nu$ to be a centered Jacobi measure of the form
\begin{align}\label{Jacobi measure}
\nu(v)=Z^{-1}(1+v)^{\mathrm{a}}(1-v)^{\mathrm{b}} d(v)\lone_{[-1,1]}(v)\,,\end{align}
where $d\in C^{1}([-1,1])$, with $d(v)>0$, $-1<\mathrm{a},\mathrm{b}<\infty$, and $Z$ a normalization constant. Then for $\mathrm{a},\mathrm{b}<1$, there is, for any $\lambda\ge0$,  $\varpi\equiv\varpi(\lambda)>0$ and $\beta_1>0$ such that~\eqref{eq assumption mu_V} and~\eqref{eq assumption mu_V 2} are satisfied. However, if $\mathrm{a}>1$ or $\mathrm{b}>1$ then~\eqref{assumption mu_V} may not be satisfied for $\lambda_0$ sufficiently large. In this setting the deformed semicircle law is still supported on a single interval, but the square root behavior at the edge may fail. We refer to~\cite{LS1,LS2} for a detailed discussion.
\end{itemize}

\subsection{Deformed semicircle law}\label{subsection deformed semicirle law}

The deformed semicircle law can be described in terms of the Stieltjes transform: For a (probability) measure $\omega$ on the real line we define its Stieltjes transform, $m_{\omega}$, by
\begin{align}
 m_{\omega}(z)\deq\int\frac{\dd\omega(v)}{v-z}\,,\qquad\qquad (z\in\C^+)\,.
\end{align}
Note that $m_{\omega}$ is an analytic function in the upper half plane and that $\im m_\omega(z)\ge 0$, $\im z> 0$. Assuming that $\omega$ is absolutely continuous with respect to Lebesgue measure, we can recover the density of $\omega$ from $m_{\omega}$ by the inversion formula
\begin{align}\label{stieltjes inversion formula}
 \omega(E)=\lim_{\eta\searrow 0}\frac{1}{\pi}\im m_{\omega}(E+\ii \eta)\,,\qquad\qquad (E\in\R)\,.
\end{align}
We use the same symbols to denote measures and their densities.

Choosing $\omega$ to be the standard semicircular law $\rho_{sc}$, the Stieltjes transform $m_{sc}\equiv m_{\rho_{sc}}$ can be computed explicitly and one checks that $m_{sc}$ satisfies the relation
\begin{align}
 m_{sc}(z)=\frac{-1}{m_{sc}(z)+z}\,,\qquad \quad\im m_{sc}(z)\ge0\,,\qquad\qquad (z\in\C^+)\,.
\end{align}

The deformed semicircle law is conveniently defined through its Stieltjes transform. Let $\nu$ be the limiting probability measure of Assumption~\ref{assumption mu_V convergence}. Then it is well-known~\cite{P} that the functional equation
\begin{align}\label{mfc equation}
 m_{fc}(z)= \int\frac{\dd\nu(v)}{\lambda_0 v-z-m_{fc}(z)}\,,\qquad\quad \im m_{fc}(z)\ge 0\,,\qquad\qquad (z\in\C^+)\,,
\end{align}
 has a unique solution that satisfies $\limsup_{\eta\searrow 0}\im m(E+\ii\eta)<\infty$, for all $E\in\R$.
The deformed semicircle law, denoted by $\rho_{fc}$, is then defined through its density
\begin{align}
 \rho_{fc}(E)\deq\lim_{\eta\searrow 0}\frac{1}{\pi}\im m_{fc}(E+\ii\eta)\,,\qquad \quad(E\in\R)\,,
\end{align}
where $m_{fc}$ is the solution to~\eqref{mfc equation}. The measure $\rho_{fc}$ has been studied in detail in~\cite{B}. For example, it was shown there that the density $\rho_{fc}$ is an analytic function inside the support of the measure. For our analysis the following result from~\cite{S2,LS1} is relevant.

\begin{lemma}\label{lemma vorbereitung}
 Let $\nu$ and $\lambda_0$ satisfy~\eqref{eq assumption mu_V} for some $\varpi>0$. Then there are $E_-,E_+\in\R$, such that $\supp\,\rho_{fc}=[E_-,E_+]$. Moreover, $\rho_{fc}$ has a strictly positive density on $(E_-,E_+)$.
\end{lemma}

The measure $\rho_{fc}$ is also called the additive free convolution of the semicircular law and, up to the rescaling by $\lambda_0$, the measure~$\nu$. More generally, the additive free convolution of two (probability) measure $\omega_1$ and $\omega_2$, usually denoted by $\omega_1\boxplus\omega_2$, is defined as the distribution of the sum of two freely independent non-commutative random variables, having distributions $\omega_1$,~$\omega_2$ respectively; we refer to, e.g.,~\cite{VDN,AGZ} for reviews. Similar to~\eqref{mfc equation}, the free convolution measure $\omega_1\boxplus\omega_2$ can be described in terms of a set of functional equations for the Stieltjes transforms; see~\cite{CG, BeB07}.

\subsection{Main result}

Let $\mu_1^{\w}$ be the largest eigenvalue of the Wigner matrix $W$. The edge universality for Wigner matrices asserts that
\beq\label{le TW}
\lim_{N \to \infty} \p \left( N^{2/3} (\mu_1^{\w} -2) \leq s \right) = F_1 (s)\,,
\eeq
where $F_1$ is the Tracy-Widom distribution function for the Gaussian orthogonal ensembles. We remark that the Tracy-Widom distributions $F_2$ and $F_4$ arise as the limiting laws of the largest eigenvalues for the Gaussian unitary and Gaussian symplectic ensembles. Statement~\eqref{le TW} holds true for the smallest eigenvalue $\mu_N^{\w}$ as well. We henceforth focus on the largest eigenvalues, the smallest eigenvalues can be dealt with in exactly the same way.

The edge universality for deformed real symmetric Wigner matrices, the main result of this paper, is as follows.
\begin{theorem} \label{thm main}
Let $H = \lambda_0 V + W$ be a deformed Wigner matrix of the form \eqref{thematrix}, where $W$ is a real symmetric Wigner matrix satisfying the assumptions in Definition~\ref{assumption wigner}, $V$ is a real diagonal random or deterministic matrix satisfying Assumption~\ref{assumption mu_V convergence} that is independent of~$W$. Further, assume that $V$ and $\lambda_0\ge 0$ satisfy Assumption~\ref{assumption mu_V}. Let~$\mu_1$ be the largest eigenvalue of $H$.

Then, there exists $\gamma_0\equiv \gamma_0(N)$ and $\widehat E_+\equiv\widehat E_+(N)$ depending only on $\lambda_0$ and $\widehat\nu$ such that the distribution of the rescaled largest eigenvalue converges to the Tracy-Widom distribution $F_1$, i.e.,
\beq \label{eq:main}
\lim_{N \to \infty} \p \left( \gamma_0 N^{2/3} \big( \mu_1 - \widehat E_+ \big) \leq s \right) = F_1 (s)\,.
\eeq
Moreover, $\widehat E_+(N)$ converges in probability in the limit $N\to\infty$ to $E_+$, the upper endpoint of the measure $\rho_{fc}$.
\end{theorem}

\begin{remark}
A precise definition of $\widehat E_+\equiv\widehat E_+(N)$ is given in~\eqref{supp mu_fc} and \eqref{def E hat}. As mentioned above $\widehat E_+$ converges in probability to $E_+$, and we may replace $\widehat E_+$ by $E_+$ in~\eqref{eq:main} if the convergence is sufficiently fast. Note that the speed of convergence depends on the exponent $\alpha_0$ in~\eqref{equation assumption mu_V convergence}.

The normalization factor $\gamma_0\equiv\gamma_0(N)$ is given by
\begin{align}\label{le gamma}
\gamma_0=\left(-\int\frac{\dd\widehat\nu(v)}{(\lambda_0 v-\zeta)^3} \right)^{-1/3}\,,
\end{align}
where $\zeta=\widehat E_++m_{fc}(\widehat E_+)$; see~\eqref{gamma}. For any $\lambda_0$ such that Assumption~\ref{assumption mu_V} is satisfied, we have for some constant~$c$, independent of $N$, that $\zeta - \lambda_0 v > c > 0$, for all $v\in\supp\widehat\nu$. In particular, $\gamma_0=O(1)$. It hence follows from Assumption~\ref{assumption mu_V convergence} that we may replace~$\gamma_0$ by the $N$-independent quantity
\begin{align}
 \left(-\int\frac{\dd\nu(v)}{(\lambda_0v-E_++m_{fc}(E_+))^3} \right)^{-1/3}
\end{align}
in~\eqref{eq:main}.
\end{remark}

\begin{remark} \label{k main}
Theorem \ref{thm main} can be extended to correlation functions of extreme eigenvalues as follows: For any fixed~$k$, the joint distribution function of the first $k$ rescaled eigenvalues converges to that of the GOE, i.e., if we denote by $\mu_1^{\GOE} \geq \mu_2^{\GOE} \geq \ldots \geq \mu_N^{\GOE}$ the eigenvalues of a GOE matrix, then
\begin{align} \label{eq: k main}
&\lim_{N \to \infty} \p \left( \gamma_0 N^{2/3} \big( \mu_1 - \widehat E_+ \big) \leq s_1 \,, \gamma_0 N^{2/3} \big( \mu_2 - \widehat E_+ \big) \leq s_2 \,, \ldots, \gamma_0 N^{2/3} \big( \mu_k - \widehat E_+ \big) \leq s_k \right)\nonumber \\
&\quad= \lim_{N \to \infty} \p \left( N^{2/3} \big( \mu_1^{\GOE} - 2 \big) \leq s_1 \,, N^{2/3} \big( \mu_2^{\GOE} - 2 \big) \leq s_2 \,, \ldots, N^{2/3} \big( \mu_k^{\GOE} - 2 \big) \leq s_k \right)\,.
 \end{align}
\end{remark}

Our second result classifies the fluctuation of the largest eigenvalues of $H = \lambda_0 V + W$ when the entries $(v_i)$ of $V$ are i.i.d.\ random variables that are independent of~$W$. 

\begin{theorem} \label{thm2}

Let $H = \lambda_0 V + W$ be a deformed Wigner matrix of the form \eqref{thematrix}, where $W$ is a real symmetric Wigner matrix satisfying the assumptions in Definition~\ref{assumption wigner} and $V$ is a real diagonal random matrix independent of~$W$. Assume that the entries of~$V$ are i.i.d.\ random variables with distribution~$\nu$. Let $m^{(k)}(\nu)$ denote the $k$-th central moment of $\nu$. (In particular, $m^{(2)}(\nu)$ is the variance of $\nu$.) Let $\mu_1$ be the largest eigenvalue of $H$.

Then, with $E_+$ in Lemma \ref{lemma vorbereitung}, which depends only on $\nu$ and $\lambda_0$, the following holds.
\begin{itemize}
\item[i.] Let $\lambda_0=\sigma_0N^{-\delta}$, for two constants $\delta$ and $\sigma_0$ satisfying  $1/6< \delta $, $\sigma_0\ge 0$. Then the distribution of the rescaled largest eigenvalue converges to the Tracy-Widom distribution, i.e., 
\beq \label{thm2 case1}
\lim_{N \to \infty} \p \left( N^{2/3} \big( \mu_1 - E_+\big) \leq s \right) = F_1 (s)\,.
\eeq
\item[ii.]  Let $\lambda_0 =\sigma_0N^{-1/6}$ for some constant $\sigma_0>0$. Then the distribution of the rescaled largest eigenvalue converges to the convolution of the Tracy-Widom distribution and the Gaussian distribution, i.e.,
\beq \label{thm2 case3}
\lim_{N \to \infty} \p \left( N^{2/3} \big( \mu_1 -E_+ \big) \leq s \right) = \p( X_{F} + X_{\Phi} \leq s )\,,
\eeq
where $X_{F}$ and $X_{\Phi}$ are independent random variables whose cumulative distribution functions are $F_1$ and $\Phi_{\sigma^2}$, respectively. Here, $\Phi_{\sigma^2}$ denotes the cumulative distribution function of a centered Gaussian distribution with variance $\sigma^2=\sigma_0^2 m^{(2)}(\nu)$.

\item[iii.] Let $\lambda_0 =\sigma_0N^{-\delta}$, for two constants $\delta$ and $\sigma_0$ satisfying  $0\le \delta<1/6 $, $\sigma_0> 0$. Then the distribution of the rescaled largest eigenvalue converges to the Gaussian distribution, i.e., there exist constants $E_+, \sigma > 0$, depending only on $\lambda_0$ and $\nu$, such that
\beq \label{thm2 case2}
\lim_{N \to \infty} \p \left( N^{1/2} \lambda_0^{-1} \big( \mu_1 - E_+ \big) \leq s \right) = \Phi_{\sigma^2} (s)\,,
\eeq
where the standard deviation $\sigma$ is of order $O(1)$, satisfying
\begin{align}
\sigma^2= \lim_{N \to \infty} \lambda_0^{-2}(1-m_{fc}(E_+)^2)\,.
\end{align}
In particular, when $\delta>0$, we have $\sigma = m^{(2)}(\nu)$.
\end{itemize}

Finally, let $m_1(\nu)$ denote the mean of $\nu$. Then, the point $E_+$ admits the asymptotic expansion
\begin{align} \label{E asymptotic}
 E_+= 2 + \lambda_0 m_1(\nu) + \lambda_0^2 m^{(2)}(\nu) + \lambda_0^3 m^{(3)}(\nu) + \lambda_0^4 \left( m^{(4)}(\nu) - \frac{9 (m^{(2)}(\nu))^2}{4} \right) + O(\lambda_0^5)\,,
\end{align}
as $\lambda_0\to 0$.
\end{theorem}

\begin{remark}
 For the deformed GUE with deterministic $V$, Theorem~\ref{thm main} with deterministic potential has been obtained in~\cite{S2} for rather general $V$ and $\lambda=O(1)$. For the deformed GUE with random $V$, Theorem~\ref{thm main} has been established for some ranges of $\lambda$ in~\cite{S2}. The extension to all $\lambda=O(1)$ was obtained in~\cite{CP}. For random $V$ with i.i.d.\ entries, statement~$ii$ of Theorem~\ref{thm2} at $\lambda\sim N^{-1/6}$ has been established in~\cite{J2} for the deformed GUE. For the deformed GOE with random $V$, partial results on the linear statistics of the eigenvalues at the edge have been obtained in~\cite{Su} for $\lambda=o(1)$. Since very recently, there is a result~\cite{FG} on bulk and edge universality for perturbations of Gaussian matrices under polynomials of matrices with an assumption on the asymptotic expansion of moments.
\end{remark}

\begin{remark}
As remarked before, Assumption~\ref{assumption mu_V} insures that the deformed semicircle law $\rho_{fc}$ has a square root decay at the edges. When Assumption~\ref{assumption mu_V} is not satisfied this may no longer be true and one expects a different edge behavior. Assuming that $(v_i)$ are i.i.d.\ random variables with law given by a Jacobi measure as in~\eqref{Jacobi measure} this has been studied in~\cite{LS2}. For example, when $\b > 1$ then there exists $0<\lambda_+<\infty$ such that if $\lambda_0<\lambda_+$ then the Assumption~\ref{assumption mu_V} holds and the law of the rescaled largest eigenvalues converges to the Tracy-Widom distribution. When $\lambda_0>\lambda_+$, the Assumption~\ref{assumption mu_V} is not satisfied, the deformed semicircle law $\rho_{fc}$ does not have a square root behavior at the upper edge and the law of the rescaled largest eigenvalue converges to a Weibull distribution. Correspondingly, the eigenvectors associated to the largest eigenvalues are completely delocalized 
for $\lambda_0<\lambda_+$ (see~\cite{LS1}), while they are (partially) localized for $\lambda_0>\lambda_+$ (see~\cite{LS2}).
\end{remark}

\section{Preliminaries} \label{sec:prelim}

\subsection{Notations} \label{notation}
We introduce a notation for high-probability estimates which is suited for our purposes. A slightly different form was first used in~\cite{EKY}.
\begin{definition}
Let
\begin{align}
 X=(X^{(N)}(u)\,:\, N\in\N\,, u\in U^{(N)})\,,\qquad\qquad Y=(Y^{(N)}(u)\,:\, N\in\N,\,u\in U^{(N)})
\end{align}
be two families of nonnegative random variables where $U^{(N)}$ is a possibly $N$-dependent parameter set. We say that $Y$ stochastically dominates $X$, uniformly in $U$, if for all (small) $\epsilon>0$ and (large) $D>0$,
\begin{align}
 \sup_{u\in U^{(N)}}\P\left[X^{(N)}(u)>N^{\epsilon} Y^{(N)}(u) \right]\le N^{-D}\,,
\end{align}
for sufficiently large $N\ge N_0(\epsilon,D)$. If $Y$ stochastically dominates $X$, uniformly in $u$, we write $X \prec Y$. If for some complex family $X$ we have $|X|\prec Y$ we also write $X=\caO(Y)$.  Further, if $\lone(\Xi)|X|\prec Y$ for some possibly $N$-dependent event $\Xi^{(N)}\equiv \Xi$, we also write $X=\caO_{\Xi}(Y)$.
\end{definition}
For example, we have from~\eqref{eq.C0} and Chebyshevs's inequality that
$|h_{ij}|\prec \frac{1}{\sqrt{N}}$. The relation $\prec$ is a partial ordering: it is transitive and it satisfies the arithmetic rules of an order relation, e.g., if $X_1\prec Y_1$ and $X_2\prec Y_2$ then $X_1+X_2\prec Y_1+Y_2$ and $X_1 X_2\prec Y_1 Y_2$.

We use the symbol $O(\,\cdot\,)$ and $o(\,\cdot\,)$ for the standard big-O and little-o notation. The notations $O$, $o$, $\ll$, $\gg$, refer to the limit $N\to \infty$ unless stated otherwise. Here $a\ll b$ means $a=o(b)$. We use $c$ and $C$ to denote positive constants that do not depend on $N$, usually with the convention $c\le C$. Their value may change from line to line. We write $a\sim b$, if there is $C\ge1$ such that $C^{-1}|b|\le |a|\le C |b|$.

Finally, we use double brackets to denote index sets, i.e., 
$$
\llbracket n_1, n_2 \rrbracket \deq [n_1, n_2] \cap \Z\,,
$$
for $n_1, n_2 \in \R$.

\subsection{Green function and minors}
Let $A$ be an $N\times N$ real symmetric matrix. The Green function or resolvent of~$A$ is defined as $G_A(z)\deq(A-z)^{-1}$, $z\in\C^+$, and the averaged Green function of $A$ is defined as $m_A(z)\deq\frac{1}{N}\Tr G_A(z)$, $z\in\C^+$. Bellow we often drop the subscript $A$ and the argument $z$ in $G_A(z)$ and $m_A(z)$.

Let $\T\subset \llbracket1, N \rrbracket$. Then we define $A^{(\T)}$ as the $(N-|\T|)\times(N-|\T|)$ minor of $A$ obtained by removing all columns and rows of $A$ indexed by $i\in\T$. We do not change the names of the indices of $A$ when defining $A^{(\T)}$.
More specifically, we define an operation $\pi_i$, $i\in \llbracket1, N \rrbracket$, on the probability space by
\begin{align}
 (\pi_i(A))_{kl}\deq\lone(k\not=i)\lone(l\not=i)h_{kl}\,.
\end{align}
Then, for $\T\subset \llbracket1, N \rrbracket$, we set $\pi_{\T}\deq\prod_{i\in\T}\pi_i$ and define
\begin{align}
 A^{(\T)}\deq((\pi_{\T}(A)_{ij})_{i,j\not\in\T}\,.
\end{align}

The Green functions $G^{(\T)}$, are defined in an obvious way using $A^{(\T)}$.  Moreover, we use the shorthand notations
\begin{align}
 \sum_{i}^{(\T)}\deq\sum_{\substack{i=1\\i\not\in\T}}^N\,\,,\qquad \qquad\sum_{i\not=j}^{(\T)}\deq\sum_{\substack{i=1,\, j=1\\ i\not=j\,,\,i,j\not\in\T}}^N\,,
\end{align}
abbreviate $(i)=(\{i\})$, $(\T i)=(\T\cup\{i\})$. In Green function entries $(G_{ij}^{(\T)})$ we refer to $\{i,j\}$ as {\it lower indices} and to $\T$ as {\it upper indices}.

Finally, we set
\begin{align}
 m^{(\T)}\deq\frac{1}{N}\sum_{i}^{(\T)}G_{ii}^{(\T)}\,.
\end{align}
Here, we use the normalization $N^{-1}$, instead $(N-|\T|)^{-1}$, since it is more convenient for our computations.

\subsection{Resolvent identities}
The next lemma collects the main identities between resolvent matrix elements of~$A$ and~$A^{(\T)}$.

\begin{lemma}
Let $A=A^*$ be an $N\times N$ matrix. Consider the Green function $G(z)\deq(A-z)^{-1}$, $ z\in\C^+$. Then, for $i,j,k,l\in\llbracket 1,N\rrbracket$, the following identities hold:
\begin{itemize}
 \item[-] {\it Schur complement/Feshbach formula:}\label{feshbach} For any $i$,
\begin{align}\label{schur} 
 G_{ii}=\frac{1}{a_{ii}-z-\sum_{k,l}^{(i)}{a_{ik} G_{kl}^{(i)}}a_{li}}\,.
\end{align}
\item[-] For $i,j\not=k$,
\begin{align}\label{basic resolvent}
 G_{ij}=G_{ij}^{(k)}+\frac{G_{ik}G_{kj}}{G_{kk}}\,.
\end{align}

\item[-] For $i\not=j$,
\begin{align}\label{onesided}
 G_{ij}=-G_{ii}\sum_{k}^{(i)}a_{ik}G_{kj}^{(i)}=-G_{jj}\sum_{k}^{(j)} G_{ik}^{(j)} a_{kj}\,.
\end{align}
\item[-] For $i\not=j$,
\begin{align}\label{twosided}
 G_{ij}=-G_{ii}G_{jj}^{(i)}\left(a_{ij}-\sum_{k,l}^{(ij)}a_{ik}G_{kl}^{(ij)}a_{lj}\right)\,.
\end{align}

\end{itemize}
\end{lemma}
For a proof we refer to, e.g.,~\cite{EKYY1}.

\subsection{Large deviation estimates}

Consider two families of random variables $(X_i)$ and $(Y_i)$, $i\in\llbracket 1,N\rrbracket$,  satisfying
\begin{align}\label{assumption for large deviation}
\E Z_i = 0\,, \qquad\quad \E|Z_i|^2 = 1 \,,\qquad\quad \E |Z_i|^p\le c_p\,,\qquad\quad(p\ge 3)\,,
\end{align}
$Z_i=X_i,Y_i$, for all $p\in\N$ and some constants $c_p$, uniformly in $i\in\llbracket 1,N\rrbracket$. The following lemma, taken from~\cite{EKYY4}, provides useful large deviation estimates.
\begin{lemma}\label{lemma large deviation estimates}
 Let $(X_i)$ and $(Y_i)$ be independent families of random variables and let $(a_{ij})$ and $(b_i)$, $i, j \in\llbracket 1, N\rrbracket$, be collections of complex numbers. Suppose that all entries $(X_i)$ and $(Y_i)$
are independent and satisfy~\eqref{assumption for large deviation}. Then we have the bounds:
\begin{align}
\left|\sum_i b_iX_i\right|\prec \left(\sum_i |b_i|^2\right)^{1/2}\,,\\
\left|\sum_{i}\sum_{j}a_{ij} X_iY_j\right|\prec\left(\sum_{i,j}|a_{ij}|^2 \right)^{1/2}\,,\\
\left|\sum_{i\not=j}a_{ij} X_iX_j\right|\prec\left(\sum_{i\not=j}|a_{ij}|^2 \right)^{1/2}\,.
\end{align}
If the coefficients $(a_{ij})$ and $(b_i)$ depend on an additional parameter $u$, then all of these estimates are uniform
in $u$, i.e. the threshold $N_0 = N_0 (\epsilon, D)$ in the definition of $\prec$ depends only on the family
$c_p$ from~\eqref{assumption for large deviation}; in particular, $N_0$ does not depend on $u$.
\end{lemma}

\section{Renormalization of the model} \label{sec:renormalize}
In this section, we rescale and ``renormalize'' the deformed Wigner matrix $H=\lambda_0 V+W$ in order to setup later computations. 

\subsection{Removal of the diagonal of $W$ and fixing of $V$.}\label{subsection centering}
To simply the notation in the upcoming sections, we replace the Wigner matrix $W$ by
\begin{align}
 W-\diag (w_{11},\ldots ,w_{NN})\,,
\end{align}
i.e., we replace $w_{ij}$ by $w_{ij}-w_{ii}\delta_{ij}$. With this modification we have $H=(h_{ij})$, 
\begin{align}
 h_{ii}=\lambda_0 v_i\,,\qquad\qquad h_{ij}=w_{ij}\,,\qquad(i\not=j)\,.
\end{align}
By the next lemma, this replacement causes a negligible shift in the extremal eigenvalues of $W$ or $H=\lambda_0V+W$ and we thus not explicitly display this modification in our notation. 
\begin{lemma}\label{lemma remove the diagonal}
Suppose that $\lambda_0$, $V$, and $W$ satisfy the assumptions in Theorem \ref{thm main}. Let $H_1 = \lambda_0 V + W$ and $H_2 = H_1 + C_2 \diag (w_{11},\ldots ,w_{NN})$ for some constant $C_2$ independent of $N$. Further, let $\mu_1 (H_1)$ and $\mu_1 (H_2)$ be the largest eigenvalues of $H_1$ and $H_2$, respectively. Then, there exists a constant $\delta > 0$ such that, for any $s \in \mathbb{R}$, we have
\begin{align}
\p \big( N^{2/3} (\mu_1 (H_1) - \widehat E_+) \leq s - N^{-\delta} \big) - N^{-\delta} &\leq \p \big( N^{2/3} (\mu_1 (H_2) - \widehat E_+) \leq s \big)\nonumber \\
&\leq \p \big( N^{2/3} (\mu_1 (H_1) - \widehat E_+) \leq s + N^{-\delta} \big) + N^{-\delta}\,.
\end{align}
\end{lemma}
The proof Lemma~\ref{lemma remove the diagonal} follows almost verbatim from the proof of Lemma 3.5 in~\cite{LY}. (See also Theorem 3.3 in \cite{LY}). We remark that the local deformed semicircle law for $H_1$ and $H_2$, which is the key ingredient of the proof of Lemma 3.5 in~\cite{LY}, is given in Theorem~\ref{local law} below.

To  conveniently cope with the cases when $(v_i)$ are random, respectively deterministic, we introduce an event~$\Xi$ on which the random variables $(v_i)$ exhibit ``typical'' behavior. Recall that we denote by $m_{\widehat\nu}$ and $m_{\nu}$ the Stieltjes transforms of $\widehat\nu$, respectively $\nu$. 
\begin{definition}\label{definition of the event omega}
Let $\Xi\equiv\Xi(N)$ be an event such that the following holds on it:
\begin{itemize}
\item[$(1)$]There is a constant $\alpha_0>0$ such that, for any compact set $\caD\subset\C^+$ with $\mathrm{dist}(\caD,\supp \nu)>0$, there is $C$ such that 
\begin{align}\label{definition of omega_V I}
 \left|m_{\widehat\nu}(z)-m_{\nu}(z) \right|\le CN^{-\alpha_0}\,,
\end{align}
for $N$ sufficiently large.
\item[(2)] Recall the constant $\varpi>0$ and the intervals $I_\nu$, $I_{\widehat\nu}$ in Assumption~\ref{assumption mu_V}. We have
\begin{align}\label{definition of omega_V II}
 \inf_{x\in I_{\widehat\nu}}\int\frac{\dd\widehat\nu(v)}{(v-x)^2}\ge (1+\varpi)\lambda_0^2\,,\qquad\quad \inf_{x\in I_{\nu}}\int\frac{\dd\nu}{(v-x)^2}\ge (1+\varpi)\lambda_0^2\,,
\end{align}
for $N$ sufficiently large.
\end{itemize}
\end{definition}
In case $(v_i)$ are deterministic, $\Xi$ has full probability for $N$ sufficiently large by Assumptions~\ref{assumption mu_V convergence}. 

In the following we usually condition the random variables $(v_i)$ on $\Xi$, i.e., we consider $(v_i)$ as fixed and are such that~\eqref{definition of omega_V I} and~\eqref{definition of omega_V II} hold. 

\subsection{Rescaling of $H$}
Let $\lambda\in[0,\lambda_0]$. In order to compare the local edge statistics of $H=\lambda V+W$ with the local GOE edge statistics, it is natural to rescale $H$ in such a way that the typical size of the eigenvalue spacing at the upper edge of the rescaled matrix match those of $W$. Put slightly differently, we can find $\gamma\in\R$, depending on $\lambda_0$, such that the eigenvalue gaps of $\widetilde H\deq \gamma H$ typically agree for large $N$ with the gaps predicted by the Tracy-Widom distribution.

The scaling factor $\gamma$ may be constructed as follows. For any $\lambda \in [0, \lambda_0]$, let $\zeta \equiv \zeta(\lambda)$ be the largest solution to
\begin{align} \label{zeta}
\int \frac{\dd\widehat\nu(v)}{(\lambda v - \zeta)^2} = 1\,.
\end{align}
We note that, for $\widehat\nu$ satisfying~\eqref{definition of omega_V II}, such $\zeta$ exists for all $\lambda\le \lambda_0$. Note that $\zeta$ depends on $\lambda$ and the measure $\widehat\nu$. We then define the scaling factor $\gamma$ by
\begin{align} \label{gamma}
\gamma \equiv \gamma(\lambda) \deq \left( -\int\frac{\dd\widehat\nu(v)}{(\lambda v - \zeta)^3} \right)^{-1/3}\,,\qquad\quad \gamma_0\deq\gamma(\lambda_0)\,.
\end{align}
It follows from~\eqref{zeta} and H\"{o}lder's inequality that $0<\gamma\le 1$. Note that for $\lambda=0$, we have $\zeta=1$ and $\gamma=1$. 
 
We now set
\begin{align} \label{renormalized matrix}
\wt H \deq \gamma H = \gamma (\lambda V + W)\,,
\end{align}
and define  the Green function, respectively averaged Green function, of $\wt H$ by
\begin{equation*}
G_{\wt H}(z) \deq \frac{1}{\gamma (\lambda V + W) -z}\,,\qquad \qquad m_{\wt H}(z) \deq \frac{1}{N} \Tr G_{\wt H}(z)\,, \qquad \qquad (z \in \C^+)\,.
\end{equation*}

Next, we define $\widehat m_{fc}$ as the solution to the equation
\beq \label{widehat mfc}
\widehat m_{fc}(z) =\int\frac{\dd\widehat\nu(v)}{\lambda \gamma v-z- \gamma^2 \widehat m_{fc}(z)}\,,\qquad\quad \im\widehat m_{fc}(z)\ge 0\,, \qquad\qquad (z \in \C^{+})\,.
\eeq
Following the arguments in Subsection~\ref{subsection deformed semicirle law}, $\widehat m_{fc}(z)$ defines a probability measure $\widehat\rho_{fc}$ whose density is given by $\widehat\rho_{fc}(E)=\frac{1}{\pi}\lim_{\eta\searrow 0}\im \widehat\rho_{fc}(E+\ii\eta)$, $E\in\R$. For simplicity, we omit the $\lambda$- and $\gamma$-dependences of $\widehat m_{fc}$ and $\widehat\rho_{fc}$ from the notation. Note that the measure~$\widehat\rho_{fc}$ depends on $N$ through $\widehat\nu$. However, this being the main point here, $\widehat\rho_{fc}$ does not depend on the matrix~$W$ in any way. On the event $\Xi$ of typical realizations of $(v_i)$, $\widehat\rho_{fc}$ enjoys the following properties:

\begin{lemma} \label{square_root}
On $\Xi$ the following holds. There exist $\widehat L_-, \widehat L_+ \in \R$, with $\widehat L_-  <\widehat L_+$, such that 
\beq \label{supp mu_fc}
\mathrm{supp}\,\widehat \rho_{fc} = [\widehat L_-, \widehat L_+]\,.
\eeq
Denoting by $\kappa_E$ the distance to the endpoints of the support of $\rho_{fc}$, i.e.,
\beq
\kappa_E \deq \min \{|E-\widehat L_-|, |E-\widehat L_+|\}\,,
\eeq
we have
\beq
C^{-1} \sqrt \kappa_E \leq \widehat\rho_{fc}(E) \leq C \sqrt{\kappa_E}\,, \qquad \qquad (E \in [\widehat L_-,\widehat L_+])\,,
\eeq
for some constant $C \geq 1$, which can be chosen uniformly in $\lambda\in[0,\lambda_0]$. In particular, we have
\begin{align}\label{to be emphasized}
 \widehat\rho_{fc}(E)=\frac{1}{\pi}\sqrt{\kappa_E}(1+O(\kappa_E))\,,
\end{align}
as $E\to \widehat L_+$, $E\le\widehat L_+$.

Further, we have the following estimates for the imaginary part of $\widehat m_{fc}$:
\begin{enumerate}
\item
For $z=\widehat L_+-\kappa+\ii\eta$, with $0\le\kappa\le \widehat L_+$ and $0<\eta\le 2$, there exists a constant $C \geq 1$ such that
\beq
C^{-1}\sqrt{\kappa_E+\eta}\le \im \widehat m_{fc}(z)\le C \sqrt{\kappa_E+\eta}\,.
\eeq
\item
For $z=\widehat L_++\kappa+\ii\eta$, with $0\le\kappa\le 1$ and $0<\eta\le 2$, there exists a constant $C \geq 1$ such that
\beq
C^{-1}\frac{\eta}{\sqrt{\kappa_E+\eta}}\le \im \widehat m_{fc}(z)\le C\frac{\eta}{\sqrt{\kappa_E+\eta}}\,.
\eeq
\end{enumerate}
Moreover, all constants can be chosen uniformly in $\lambda\in[0,\lambda_0]$.
\end{lemma}

The proof of Lemma~\ref{square_root} can be found in~\cite{LSSY}. Returning to the introductory remarks of this subsection, we emphasize~\eqref{to be emphasized}: The scaling factor $\gamma$ has been chosen such that~\eqref{to be emphasized} holds for all $\lambda\in [0,\lambda_0]$, i.e., up to a global shift, the measure $\widehat\rho_{fc}^0$ exhibits a square root decay at the upper edge with the same rate as the standard semicircle law~$\rho_{sc}$.

\begin{remark} \label{E hat}
Let $\widehat m_{fc}^0$ be the solution to the equation
\begin{align*}
\widehat m_{fc}^0 (z) = \int \frac{\dd\widehat\nu(v)}{\lambda v-z- \widehat m_{fc}^0 (z)}\,,\qquad\quad \im\widehat m_{fc}^0 (z)\ge 0\,, \qquad\qquad (z \in \C^{+})\,,
\end{align*}
and $\widehat\rho_{fc}^0$ the probability measure whose density is given by $\widehat\rho_{fc}^0(E)\deq\frac{1}{\pi}\lim_{\eta\searrow 0}\im \widehat\rho_{fc}^0(E+\ii\eta)$, $E\in\R$. As in Lemma~\ref{square_root} we find that there are $\widehat E_-, \widehat E_+ \in \R$ such that $\supp \widehat \rho_{fc}^0 = [\widehat E_-, \widehat E_+]$ and that $\widehat\rho_{fc}^0$ has a strictly positive density in $(\widehat E_-, \widehat E_+)$. By definition, it is obvious that
\begin{align} \label{def E hat}
\widehat E_+ = \gamma^{-1} \widehat L_+\,.
\end{align}
\end{remark}

\begin{remark} \label{rem:gamma and tau}
The scaling factor $\gamma$ defined in~\eqref{gamma} satisfies the relation
\begin{align}\label{relation for the gamma}
\gamma = \left( -\int\frac{\dd\widehat\nu(v)}{(\lambda \gamma v - \tau)^3} \right)^{-1/6}\,,
\end{align}
where $\tau=\widehat L_++\gamma^2m_{fc}(\widehat L_+)$.

To see that~\eqref{gamma} implies~\eqref{relation for the gamma}, we note that $\widehat m_{fc}^0 (\widehat E_+) = \gamma \widehat m_{fc} (\widehat L_+)$, hence $\zeta$ in \eqref{zeta} satisfies $\zeta =\widehat E_+ + \widehat m_{fc}^0 (\widehat E_+)$.
This also shows that
\beq \label{upper edge}
\int\frac{\dd\widehat\nu(v)}{\big( \lambda \gamma v-\widehat L_+ - \gamma^2 \widehat m_{fc}(\widehat L_+) \big)^2} = \frac{1}{\gamma^2}\int \frac{\dd\widehat\nu(v)}{\big( \lambda v-\widehat E_+ - \widehat m_{fc}^0(\widehat E_+) \big)^2} = \frac{1}{\gamma^2 }\int \frac{\dd\widehat\nu(v)}{(\lambda v- \zeta)^2} = \frac{1}{\gamma^2}\,.
\eeq
We define $\tau$ by
\begin{align} \label{tau}
\tau = \widehat L_+ + \gamma^2 \widehat m_{fc}(\widehat L_+) = \gamma \zeta\,,
\end{align}
and find that
\begin{align*}
\int\frac{\dd\widehat\nu(v)}{(\lambda \gamma v - \tau)^3} = \frac{1}{\gamma^3 }\int\frac{\dd\widehat\nu(v)}{(\lambda v - \zeta)^3} = -\frac{1}{\gamma^6}\,.
\end{align*}
This proves~\eqref{relation for the gamma}.
\end{remark}

Next, we collect estimates on the Green function of $\widetilde H$. Fix a small $\xi>0$ and define the domain
\begin{align}
 \caD_{\xi}\deq\lbrace z=E+\ii\eta \in\C^+\,:\,  0\le |E|\le \lambda_0+1\,, N^{-1+\xi}\le \eta\le 3 \rbrace\,.
\end{align}
We also introduce the control parameter
\begin{align}\label{le Pi}
 \Pi(z)\deq \sqrt{\frac{\im \widehat m_{fc}(z)}{N\eta}}+\frac{1}{N\eta}\,.
\end{align}
 The next theorem is the local deformed semicircle law for $\widetilde H$, which was established in Theorem 3.3 of~\cite{LSSY}.
\begin{theorem}[Local deformed semicircle law] \label{local law}
On $\Xi$, the following holds true. For any small fixed $\xi>0$, we have 
\begin{align} \label{local 1}
  |m_{\wt H}(z) - \widehat m_{fc}(z)| \prec \frac{1}{N\eta} \,,\qquad \quad\max_{i\not=j}|(G_{\wt H})_{ij}(z)|\prec \Pi(z)\,,
\end{align}
uniformly in $z\in\caD_{\xi}$ and $\lambda\in[0,\lambda_0]$. Further, setting $g_{i}(z)\deq (\lambda\gamma v_i-z-\gamma^2\widehat m_{fc}(z))^{-1}$, we also have  
\begin{align}\label{local 2}
 \max_i|(G_{\wt H})_{ii}(z)-g_i(z)|\prec \Pi(z)\,,
\end{align}
uniformly in $z\in\caD_{\xi}$ and $\lambda\in[0,\lambda_0]$. In particular, we have $|(G_{\wt H})_{ii}(z)|\prec 1$.
\end{theorem}

The following lemma gives a rigidity estimate on the eigenvalue location of $\wt H$. We denote by $\wt \mu_1 \geq \wt \mu_2 \geq \cdots \geq \wt \mu_N$ the eigenvalues of $\wt H$ in descending order. Define the ``classical'' location, $\gamma_k$, of the $k$-th eigenvalue of $\wt H$ by
\beq \label{classical location}
\int_{-\infty}^{\gamma} \widehat\rho_{fc}(x) \dd x = \frac{k}{N}\,,\qquad\qquad (k\in\llbracket 1,N\rrbracket)\,.
\eeq

\begin{lemma}[Rigidity of eigenvalues] \label{rigidity}
On $\Xi$, we have
\beq \label{rigidity of eigenvalues equation}
|\wt \mu_{k} - \gamma_{k}|\prec {N^{-2/3}}\left(\frac{1}{\widehat k}\right)^{1/3}\,,
\eeq
uniformly in $\lambda\in[0,\lambda_0]$, where we have set $\widehat k \deq \min \{k, N-k\}$.
\end{lemma}
Lemma~\ref{rigidity} follows from Theorem~\ref{local law} by an application of the Helffer-Sj\"{o}strand calculus. The proof of Lemma 5.1 in~\cite{EYY} for (generalized) Wigner matrices applies ad verbum to deformed Wigner matrices.

Alluding once more to the introductory remarks of the present subsection, we remark that the classical locations~$(\gamma_\alpha)$ depend on $\lambda$. Yet, close to the upper edge, i.e., $\alpha\ll N^{1/3}$, the gaps $\gamma_{\alpha+1}-\gamma_{\alpha}$ are essentially independent of $\lambda$ as follow from~\eqref{to be emphasized}. From Lemma~\ref{rigidity}, we can extend this conclusion to the eigenvalue gaps $\widetilde\mu_{\alpha+1}-\widetilde\mu_{\alpha}$ at the upper edge on~$\Xi$ for~$N$ sufficiently large. 
\begin{remark}
The local law in Theorem~\ref{local law} and the rigidity result in Lemma \ref{rigidity} are stronger than the corresponding results in \cite{LS1}. The improvement is based on fixing the diagonal element $(v_i)$. See Theorem 2.12, Remark 2.12, and Remark 2.14 in \cite{LS1} for more discussion. In fact, the estimates in Theorem~\ref{local law} and Lemma~\ref{rigidity} are essentially optimal up to corrections $N^{\epsilon}$. 
\end{remark}

\section{Proof of Main results} \label{sec:proof of main results}

\subsection{Density of states and the averaged Green function}

We follow the proof of the edge universality in~\cite{EYY, EKYY2}. Recall that $\wt \mu_1 \geq \wt \mu_2 \geq \ldots \geq \wt \mu_N$ denote the eigenvalues of $\wt H$ and $\widehat L_+$ is the upper edge of $\supp \widehat \rho_{fc}$. Recall the event~$\Xi$ in Definition~\ref{definition of the event omega}. From Lemma~\ref{rigidity}, we find that
\beq
|\wt \mu_1 - \widehat L_+ |\prec N^{-2/3}\,,\nonumber
\eeq	
on $\Xi$. Thus, we may assume in \eqref{eq:main} that $s \prec 1$.

Fix $E_*$ such that
\beq
E_* - \widehat L_+ \prec N^{-2/3}, \qquad\quad  \lone(\mu_1 - E_*>0) \prec 0\,. \nonumber
\eeq

 We note that the choice of $E_*$ guarantees that the event $\mu_1 > E_*$ is negligible. For $E$ satisfying
\beq \label{def_E}
|E- \widehat L_+| \prec N^{-2/3}\,,
\eeq
we let
\beq
\chi_E \deq \lone_{[E, E_*]}\,.\nonumber
\eeq
We also define the Poisson kernel, $\theta_\eta$, for $\eta > 0$, by
\beq
\theta_{\eta} (x) \deq \frac{\eta}{\pi (x^2 + \eta^2)} = \frac{1}{\pi} \im \frac{1}{x - \ii \eta}\,.\nonumber
\eeq
Introduce a smooth cutoff function $K\,:\, \R \to \R$ satisfying
\beq \label{def_K}
K(x) =
	\begin{cases}
	1 & \text{ if } x \leq 1/9\,, \\
	0 & \text{ if } x \geq 2/9\,.
	\end{cases}
\eeq
Let $\caN(E_1, E_2)$ be the number of the eigenvalues in $(E_1, E_2]$, i.e.,
\beq
\caN(E_1,E_2) \deq |\{ \alpha\,:\, E_1 < \wt \mu_{\alpha} \leq E_2\}|\,,\nonumber
\eeq 
and define the density of states in the interval $[E_1,E_2]$ by
\beq
\frn(E_1,E_2) \deq \frac{1}{N} \caN(E_1,E_2)\,.\nonumber
\eeq
In order to estimate $\p (\wt \mu_1 \leq E)$, we consider the following approximation:
\beq \label{edge_approx}
\p (\wt \mu_1 \leq E) = \E K( \caN (E, \infty) ) \simeq \E K( \caN (E, E_*) ) \simeq \E K \left( N \int_E^{E_*} \im m(y + \ii \eta)\, \dd y \right)\,,
\eeq
with $\eta\sim N^{-2/3-\epsilon'}$, for some small $\epsilon'>0$. The first approximation in \eqref{edge_approx} follows from Lemma \ref{rigidity}, the rigidity of the eigenvalues, and the second from
\beq
\caN (E, E_*) = \Tr \chi_E (H) \simeq \Tr \chi_E * \theta_{\eta} (H) = \frac{1}{\pi} N \int_E^{E_*} \im m(y + \ii \eta)\, \dd y\,.\nonumber
\eeq

The following lemma shows that the approximations in~\eqref{edge_approx} indeed hold.
\begin{lemma} \label{lem approx}
Suppose that $E$ satisfies \eqref{def_E}. Let $K$ be a smooth function satisfying \eqref{def_K}. For $\epsilon>0$, let $\ell \deq \frac{1}{2} N^{-2/3 - \epsilon}$ and $\eta \deq N^{-2/3 - 9\epsilon}$. Then, for any sufficiently small $\epsilon > 0$ and any (large) $D > 0$, we have
\beq
\Tr \left( \chi_{E + \ell} * \theta_{\eta} (H) \right) - N^{-\epsilon} \leq \frn (E, \infty) \leq \Tr \left( \chi_{E - \ell} * \theta_{\eta} (H) \right) + N^{-\epsilon}
\eeq
and
\beq
\E K \left( \Tr \left( \chi_{E - \ell} * \theta_{\eta} (H) \right) \right) \leq \p (\wt \mu_1 \leq E) \leq \E K \left( \Tr \left( \chi_{E + \ell} * \theta_{\eta} (H) \right) \right) + N^{-D}\,,
\eeq
for any sufficiently large $N \geq N_0 (\epsilon, D)$.
\end{lemma}

\begin{proof}
We may follow the proof of Corollary 6.2 of \cite{EYY}. (See also Lemma 6.5 of \cite{EKYY2}.) Note that the estimates on $|m (E + \ii \ell) - m_{fc} (E + \ii \ell)|$ and $\im m_{fc} (E - \kappa + \ii \ell)$, which replace similar estimates with respect to $m_{sc}$ in the proof of Corollary 6.2 in \cite{EYY}, are already proved in Lemma~\ref{square_root} and Theorem~\ref{local law}.
\end{proof}

\subsection{Green function comparison and proof of Theorem \ref{thm main}}

We now prove the main result of the paper using the following proposition, which compares the right side of \eqref{edge_approx} and the corresponding expectation with respect to the Wigner matrix $W$. Recall that the averaged Green function of $\wt H$ is defined by
\beq
m_{\wt H}(z) \deq \frac{1}{N} \Tr (\wt H-z)^{-1}\,, \qquad \qquad (z \in \C^+)\,.\nonumber
\eeq
Let $W^{\GOE}$ be a standard GOE matrix which is independent of $V$ and $W$. We define the averaged Green function of $W^{\GOE}$ by
\beq
m^{\GOE}(z) \deq \frac{1}{N} \Tr (W^{\GOE}-z)^{-1}, \qquad\qquad ( z \in \C^+)\,.\nonumber
\eeq

\begin{proposition}[Green function comparison] \label{prop green}
Let $\epsilon>0$ and set $\eta = N^{-2/3 - \epsilon}$. Let $E_1, E_2\in\R$ satisfy
\begin{align}\label{E1 and E2}
|E_1 - 2| \leq N^{-2/3 + \epsilon}\,, \qquad\quad |E_2 - 2| \leq N^{-2/3 + \epsilon}\,.
\end{align}
Let $F : \R \to \R$ be a smooth function satisfying
\beq \label{F bound}
\max_x |F^{(\ell)}(x)| (|x|+1)^{-C} \leq C\,,\qquad \qquad \ell=1,2, 3, 4\,.
\eeq
Then, there exists a constant $C' > 0$ such that, for any sufficiently large $N$ and for any sufficiently small $\epsilon > 0$, we have that, on $\Xi$,
\begin{align} \label{green_comp}
\left| \E F \left( N \int_{E_1}^{E_2} \im m_{\wt H}(x + \widehat L_+ - 2 + \ii \eta) \,\dd x \right) - \E F \left( N \int_{E_1}^{E_2} \im m^{\mathrm{GOE}}(x + \ii \eta)\, \dd x \right) \right| \leq N^{-1/6 + C'\epsilon}\,,
\end{align}
where the expectation $\E$ is with respect to $W$.
\end{proposition}
We prove Proposition~\ref{prop green} in the Section~\ref{sec:dbm}. 

\begin{remark}
Proposition \ref{prop green} can be extended as follows: Let $\epsilon>0$ and set $\eta = N^{-2/3 - \epsilon}$. Let $E_0, E_1, \ldots, E_k \in\R$ satisfy
\beq
|E_0 - 2| \leq N^{-2/3 + \epsilon}\,, \qquad |E_1 - 2| \leq N^{-2/3 + \epsilon}\,, \qquad \ldots \,, \qquad |E_k - 2| \leq N^{-2/3 + \epsilon}\,.\nonumber
\eeq
Let $F : \R^k \to \R$ be a smooth function satisfying
\beq
\max_x |F^{(\ell)}(x)| (|x|+1)^{-C} \leq C\,,\qquad \qquad \ell=1, 2, 3, 4\,.\nonumber
\eeq
Then, there exists a constant $C' > 0$ such that, for any sufficiently large $N$ and for any sufficiently small $\epsilon > 0$, we have on $\Xi$ that
\begin{align} \label{k green_comp}
&\Bigg| \E F \Bigg( N \int_{E_1}^{E_0} \im m_{\wt H}(x + \widehat L_+ - 2 + \ii \eta) \,\dd x \,, \ldots , N \int_{E_k}^{E_0} \im m_{\wt H}(x + \widehat L_+ - 2 + \ii \eta)\, \dd x \Bigg)\nonumber \\
&\qquad \qquad - \E F \Bigg( N \int_{E_1}^{E_0} \im m^{\mathrm{GOE}}(x + \ii \eta)\, \dd x \,, \ldots , N \int_{E_k}^{E_0} \im m^{\mathrm{GOE}}(x + \ii \eta) \,\dd x \Bigg) \Bigg| \leq N^{-1/6 + C'\epsilon}\,. 
\end{align}
The proof of~\eqref{k green_comp} is similar to that of Proposition \ref{prop green} and will be omitted. Assuming the validity of the proposition, we now prove the main result. 
\end{remark}

\begin{proof}[Proof of Theorem \ref{thm main}]
Recall that we denote by $\mu_1^{\GOE} \geq \mu_2^{\GOE} \geq \cdots \geq \mu_N^{\GOE}$ the eigenvalues of $W^{\GOE}$. Since $\P (\Xi) \to 1$ as $N \to \infty$ by assumption, we may assume that $V$ is fixed and condition on $\Xi$. Thus, to prove~\eqref{eq:main}, it suffices to establish
\beq \label{edge}
\p [N^{2/3} (\mu_1^{\GOE} - 2) \leq s] - N^{-\phi} < \p [N^{2/3} \big( \wt \mu_1 - \widehat L_+ \big) \leq s] < \p [N^{2/3} (\mu_1^{\GOE} - 2) \leq s] + N^{-\phi}\,,
\eeq
for some $\phi > 0$.

Fix $s \prec 1$ and let $E \deq \widehat L_+ + s N^{-2/3}$. Let $\ell \deq \frac{1}{2} N^{-2/3 - \epsilon}$ and $\eta := N^{-2/3 - 9\epsilon}$. For any sufficiently small $\epsilon > 0$, we have from Lemma \ref{lem approx} that
\beq
\p (\wt \mu_1 \leq E) \geq \E \left[ K \left( \Tr \left( \chi_{E - \ell} * \theta_{\eta} (H) \right) \right) \right]\,.\nonumber
\eeq
From Proposition \ref{prop green}, we find that
\beq
\E \left[ K \left( \Tr \left( \chi_{E - \ell} * \theta_{\eta} (H) \right) \right) \right] \geq \E \left[ K \left( \Tr \left( \chi_{E - (\widehat L_+ - 2) - \ell} * \theta_{\eta} (W^{\GOE}) \right) \right) \right] - N^{-\phi}\,,\nonumber
\eeq
for some $\phi > 0$. Finally, we have from Corollary 6.2 of \cite{EYY} that
\beq
\E \left[ K \left( \Tr \left( \chi_{E - (\widehat L_+ - 2) - \ell} * \theta_{\eta} (W^{\GOE}) \right) \right) \right] \geq \p \left(\mu_1^{\GOE} \leq E - (\widehat L_+ - 2) \right) - N^{-\phi}\,.\nonumber
\eeq
Altogether, we have shown that
\beq
\p (\wt \mu_1 \leq E) \geq \p \left(\mu_1^{\GOE} \leq E - (\widehat L_+ - 2)\right) - 2 N^{-\phi}\,,\nonumber
\eeq
which proves the first inequality of \eqref{edge}. The second inequality can be proved similarly. 

To complete the proof of the desired theorem, we notice that it was proved in Lemma~C.1 of~\cite{LS2} that there exists a random variable $X \equiv X(N)$, which converges to the Gaussian random variable with mean $0$ and variance $N^{-1} (1 - (m_{fc}(E_+))^2)$, satisfying
\beq
\widehat E_+ - E_+ = X + \caO(N^{-1})\,,\nonumber
\eeq
on $\Xi$. In the proof of Theorem \ref{thm2}, we will show that $\mathrm{Var}(X) \sim N^{-1} \lambda_0^2$, which implies that $\widehat E_+$ converges in probability to $E_+$.
\end{proof}

Using the general form of the Green function comparison as in \eqref{k green_comp}, we can prove \eqref{k main} in a similar manner.

\subsection{Proof of Theorem \ref{thm2}}

We next prove Theorem \ref{thm2}. Recall that $\rho_{fc}$ denotes the deformed semicircle measure, whose Stieltjes transform is denoted by $m_{fc}$.
\begin{proof}[Proof of Theorem \ref{thm2}]
 For simplicity, assume that $\nu$ is centered; the proof is essentially the same even if $\nu$ is not centered. Recall that we denote by $m^{(n)}(\nu)$ the $n$-th central moment of $\nu$. We notice that $\supp \rho_{fc} = [E_-, E_+]$, for some $E_- <0 < E_+$; see Lemma~\ref{lemma vorbereitung}. As pointed out in the proof of Theorem \ref{thm main}, there exists a random variable $X$, which converges to the Gaussian random variable with mean $0$ and variance $N^{-1} (1 - (m_{fc}(E_+))^2)$, satisfying
\beq
\widehat E_+ - E_+ = X + \caO(N^{-1})\,,\nonumber
\eeq
on $\Xi$. Let
\beq
\vartheta \deq E_+ + m_{fc}(E_+) = E_+ + \int \frac{\dd \nu(v)}{\lambda_0 v - \vartheta}\,.\nonumber
\eeq
It was shown in \cite{LS1} that $\vartheta$ is the solution to the equation
\beq \label{eq:xi}
\int \frac{\dd \nu(v)}{(\lambda_0 v - \vartheta)^2} = 1\,,\qquad\qquad \vartheta\ge 0\,,
\eeq
and that there exists a constant $c > 0$, independent of $N$, such that $\vartheta - \lambda_0 v > c$ for any $v \in \supp \nu$.

We first consider the case $\lambda_0 \ll 1$. Expanding \eqref{eq:xi} in terms of $\lambda_0$, we obtain
\begin{align} \label{xi expansion}
1 &= \int \frac{\dd \nu(v)}{(\lambda_0 v - \vartheta)^2} = \int \dd \nu(v) \left( \frac{1}{\vartheta^2} + \frac{2 \lambda_0 v}{\vartheta^3} + \frac{3 \lambda_0^2 v^2}{\vartheta^4} + O(\lambda_0^3) \right) = \frac{1}{\vartheta^2} + \frac{3 \lambda_0^2 m^{(2)}(\nu)}{\vartheta^4} + O(\lambda_0^3)\,,
\end{align}
where we used that $\nu$ is centered and has variance $m^{(2)}(\nu)$. We thus get $\vartheta^2 = 1 + O(\lambda_0^2)$ and, by putting it back into~\eqref{xi expansion},
\beq
\vartheta = 1 + \frac{3 \lambda_0^2m^{(2)}(\nu)}{2} + O(\lambda_0^3)\,.\nonumber
\eeq
We now have
\beq
m_{fc}(E_+) = \int \frac{\dd \nu(v)}{\lambda_0 v - \vartheta} = -\frac{1}{\vartheta} - \frac{\lambda_0^2 m^{(2)}(\nu)}{\vartheta^3} + O(\lambda_0^3) = -1 + \frac{\lambda_0^2m^{(2)}(\nu)}{2} + O(\lambda_0^3)\,,\nonumber
\eeq
and thus
\beq \label{var X}
N\mathrm{Var}(X) = 1 - (m_{fc}(E_+))^2 =  \lambda_0^2 m^{(2)}(\nu) + O( \lambda_0^3)\,.
\eeq
A similar computation yields,
\begin{align}
 \gamma_0=1+O(\lambda_0^2)\,,\nonumber
\end{align}
on $\Xi$. 

Comparing $|\widehat E_+ -E_+|$ and  $|\mu_1 - \widehat E_+|$, with Theorem \ref{thm main} and Equation \eqref{var X}, we can establish the first part of Theorem~\ref{thm2},~\eqref{thm2 case1}, and the second part,~\eqref{thm2 case3}. Similarly, if $N^{-1/6} \ll \lambda_0 \ll 1$, we can also prove \eqref{thm2 case2} by using Theorem \ref{thm main} and \eqref{var X}, and in particular, $\sigma = m^{(2)}(\nu)$.

We next show that, for any $\lambda_0 \gg N^{-1/6}$, \eqref{thm2 case2} holds for some $\sigma \sim 1$. We notice that, if $\lambda_0 < \epsilon$ for some sufficiently small constant $\epsilon > 0$, independent of $N$, we can prove an estimate on the variance of $X$ similar to \eqref{var X}, i.e.,
$$
\frac{1}{2} N^{-1} \lambda_0^2 \leq \mathrm{Var}(X) \leq 2 N^{-1} \lambda_0^2\,,
$$
which shows that $\sigma \sim 1$ when $\lambda_0 < \epsilon$. When $\lambda_0 \geq \epsilon$, it is obvious that
$$
1 - (m_{fc}(E_+))^2 = \int \frac{\dd \nu(v)}{(\lambda_0 v - \vartheta)^2} - \left(\int \frac{\dd \nu(v)}{\lambda_0 v - \vartheta}\right)^2 > c > 0\,,
$$
for some constant $c > 0$, hence $\mathrm{Var} (X) \sim N^{-1} \sim N^{-1} \lambda_0^2$. Thus, we can conclude that $\sigma \sim 1$ in any case. This show statement~$iii$ of Theorem~\ref{thm2}. Since $X$ is a Gaussian random variable with $\mathrm{Var} (X) \sim N^{-1}$, we see that $\widehat E_+ \to E_+$ as $N \to \infty$.

It remains to prove \eqref{E asymptotic}. Expanding \eqref{xi expansion} further and solving it for $\vartheta$, we find that
$$
\vartheta = 1 + \frac{3 \lambda_0^2 m^{(2)}(\nu)}{2} + 2 \lambda_0^3 m^{(3)}(\nu) + \lambda_0^4 \left( \frac{5 m^{(4)}(\nu)}{2} - \frac{45 (m^{(2)}(\nu))^2}{4} \right) + O(\lambda_0^5)\,.
$$
Thus, we obtain that
\begin{align*}
E_+ &= \vartheta - m_{fc}(E_+) = \vartheta - \int \frac{\dd \nu(v)}{\lambda_0 v - \vartheta} \\
&= \vartheta + \frac{1}{\vartheta} + \frac{\lambda_0^2 m^{(2)}(\nu)}{\vartheta^3} + \frac{\lambda_0^3 m^{(3)}(\nu)}{\vartheta^4} + \frac{\lambda_0^4 m^{(4)}(\nu)}{\vartheta^5} + O(\lambda_0^5) \\
&= 2 + \lambda_0^2 m^{(2)}(\nu) + \lambda_0^3 m^{(3)}(\nu) + \lambda_0^4 \left( m^{(4)}(\nu) - \frac{9 (m^{(2)}(\nu))^2 }{4} \right) + O(\lambda_0^5)\,,
\end{align*}
proving~\eqref{E asymptotic}. This completes the proof of Theorem~\ref{thm2}.
\end{proof}

\section{Dyson Brownian motion} \label{sec:dbm}

In this section, we prove Proposition \ref{prop green}. The guiding idea of our proof is that Dyson's Brownian motion (DBM) interpolates (in the sense of distributions) between the deformed Wigner matrix $H$ and the GOE matrix $W^{\GOE}$. Following the flow of the DBM we show that the expectations of the Green functions of $H$ and $W^{\GOE}$ can be compared for appropriately chosen energies.

We first recapitulate Dyson's Brownian motion~\cite{D} in Subsection~\ref{setup}.
\subsection{Preliminaries} \label{setup}
Let $H_0$ be the matrix
\begin{align}\label{le initial condition}
 H_0={ \lambda_0}V+W\,,
\end{align}
where $V=\diag(v_i)$ is a diagonal matrix and $W$ is a real symmetric  Wigner matrix that satisfies the assumptions in Definition \ref{assumption wigner} and has vanishing diagonal entries (see Subsection~\ref{subsection centering}). Here $\lambda_0$ and $(v_i)$ are chosen to satisfy Assumption~\ref{assumption mu_V}. We consider $(v_i)$ to be fixed, in particular, if $V$ is a random we consider them to be conditioned on the event $\Xi$ introduced in Definition~\ref{definition of the event omega}. 

Let $(\beta_{ij}(t))$ be a real symmetric matrix, whose diagonal entries are zero and the off-diagonal entries are a collection of independent, up to the symmetry constraint, real standard Brownian motions, independent of $H_0$. More precisely,  $\beta_{ii}(t)=0$, $t\ge0$, while $\beta_{ij}(t)$, $i<j$, $t\ge0$, is a standard Brownian motion starting at zero.
 
Let $H(t)=(h_{ij}(t))$, $t\ge 0$, satisfy the stochastic differential equation,
\begin{align}\label{le DBM}
\dd h_{ii}=-\frac{1}{2} h_{ii}\,\dd t\,,\qquad\quad \dd h_{ij}=\frac{\dd\beta_{ij}}{\sqrt{N}}-\frac{1}{2}h_{ij}\,\dd t\,,\qquad\qquad(i\not=j)\,, 
\end{align}
with initial condition $H(t=0)=H_0$. In the following we usually write $h_{ij}\equiv h_{ij}(t)$ and we refer to $t$ as time. Note that we consider in~\eqref{le DBM} a matrix-valued Ornstein-Uhlenbeck process with a drift term which insures that the variances of~$h_{ij}(t)$ remain constant over time. In Dyson's original work~\cite{D} this drift term was absent while the diagonal entries were also driven by  Brownian motions. It is easy to check that the distribution of $H(t)$ agrees with the distribution of the matrix
\begin{align*}
 { \lambda_0}\e{-t/2}V+\e{-t/2}W+(1-\e{-t})^{1/2}{W^{\GOE}}\,,
\end{align*}
where $W^{\GOE}$ is a GOE matrix independent of $V$ and $W$, whose diagonal entries are set to zero. Thus the process defined by~\eqref{le DBM} indeed interpolates in the sense of distributions between the deformed ensemble and the GOE (with vanishing diagonal) which is invariant under the process defined in~\eqref{le DBM}.

In the following we denote by~$\E$ to expectation with respect to the off-diagonal random variables~$(h_{ij})$, $i\not=j$, while we use the notation~$\E_V$ for the expectation with respect the diagonal random variables~$(h_{ii})$.

Recall the definition of the $\lambda$-dependent quantities $\zeta\equiv\zeta(\lambda)$ and $\gamma\equiv\gamma(\lambda)$ in~\eqref{zeta} and~\eqref{gamma}. Setting $\lambda(t)\deq \lambda_0\e{-t/2}$, we may now view $\zeta$, $\gamma$ and $\lambda$ as depending on $t$ (and $\lambda_0$) by extending the definitions in~\eqref{zeta} and~\eqref{gamma} in the natural way. In the same way, we obtain a $t$-dependent measure $\widehat\rho_{fc}(t)$ (whose density at $E\in\R$ is denote by $\widehat\rho_{fc}(t,E)$) by choosing $\gamma$ and $\lambda$ depending on $t$ via $\lambda\equiv\lambda_0\e{-t/2}$ in the defining equation~\eqref{widehat mfc} for $\widehat m_{fc}\equiv\widehat m_{fc}(t)$. Note that the statements of Lemma~\ref{square_root} directly carry over to~$\widehat\rho_{fc}(t)$ and~$\widehat m_{fc}(t)$. We denote by~$\widehat L_+ \equiv \widehat L_+(t)$ the upper endpoint of the support of the measure~$\widehat\rho_{fc}(t)$.

We now consider the Green function of  the rescaled random matrix  $\widetilde H(t)\deq\gamma(t) H(t)$. To prove Proposition~\ref{prop green}, we also have to choose the spectral parameter $z$ as time dependent.  Fix some small $\epsilon>0$ and define the domain, $\caE_\epsilon\equiv\caE_\epsilon(t)$, of the spectral parameter $z$ by
\begin{align}
 \caE_{\epsilon}(t)\deq\{z=L_+(t)+y+\ii\eta\in\C^+\,:\, y\in[-N^{-2/3+\epsilon}, N^{-2/3+\epsilon}]\,,\eta=N^{-2/3-\epsilon}\}\,,\qquad\quad (t\ge 0)\,.
\end{align}
For $z\equiv z(t)\in\caE_{\epsilon}(t)$, we consider the Green functions
\begin{align}
G(t,z) \deq \frac{1}{\gamma(t) H(t)-z(t)}\,,\qquad\qquad m(t,z)=\frac{1}{N}\Tr G(t,z)\,,\qquad\qquad (t\ge0)\,.\nonumber
\end{align}
Recalling Lemma~\ref{square_root} and Theorem~\ref{local law} we obtain that, on $\Xi$,
\begin{align}\label{local law adapted}
 |m(t,z)-\widehat m_{fc}(t,z)|\prec \Psi\,,\qquad\quad \max_{i\not=j} |G_{ij}(t,z)|\prec\Psi \,,\qquad\quad \max_{i}|G_{ii}(t,z)|\prec 1\,,
\end{align}
uniformly in $\caE_{\epsilon}(t)$, $t\ge 0$, where we have set
\begin{align}\label{le psi}
\Psi \deq N^{-1/3 + C' \epsilon}\,,
\end{align}
for some constant $C'$ independent of $N$, $\epsilon$ and $t$. For simplicity, we abbreviate$$
G \equiv G(t, z), \qquad\qquad m \equiv m(t,z)\,,\qquad\qquad \widehat L_+\equiv\widehat L_+(t)\,,
$$
 etc., in the following. Note that, for fixed $t\ge 0$ and $\epsilon>0$, the spectral parameter $z$ is a function of $y\in\R$ (with $|y|\le N^{-2/3+\epsilon}$) and so are $G$ and $m$.

\subsection{Proof of Proposition~\ref{prop green}}
In this subsection we give the proof of Proposition~\ref{prop green}.

Using It\={o}'s lemma we derive the stochastic differential equation for the matrix entries $G_{ij}(t,z)$ in Subsection~\ref{subsection: green function flow}. Anticipating this computation and further calculations of the remaining sections, we next state the key result, Proposition~\ref{prop expansion} below, that directly leads to the proof of Proposition~\ref{prop green}. Recall that we use the symbol $\E$ to denote the expectation with respect to the off-diagonal random variables~$(h_{ij})$, $i\not=j$, while we use the notation~$\E_V$ for the expectation with respect to the diagonal random variables~$(h_{ii})$.

\begin{proposition} \label{prop expansion}
Let $H\equiv H(t)$ be the solution of~\eqref{le DBM}. Let
\begin{align}\label{le definition of le X}
\frX \deq N \int_{E_1}^{E_2} \im m(\widehat L_+ +x  -2 + \ii \eta)\, \dd x = N \int_{E_1 -2}^{E_2 -2} \im m(\widehat L_+ +\wt y + \ii \eta) \,\dd \wt y\,,
\end{align}
where $E_1\,,E_2$ satisfy~\eqref{E1 and E2}. Let $F : \R \to \R$ be a smooth function satisfying~\eqref{F bound}. Then there exist a real-valued function $(t,H)\mapsto\Theta(t, H)\equiv\Theta$ and a martingale $t\mapsto M(t)$ with zero expectation such that
\begin{align}
 \dd F(\frX)=\Theta\,\dd t+\dd M\,.
\end{align}
On $\Xi$, $\Theta$ satisfies
\begin{align} \label{theta bound}
| \E [\Theta(t,H(t))]|\le C N^{1/2}\Psi^2\,,
\end{align}
uniformly in $t\ge 0$, where $\Psi$ is given by~\eqref{le psi}.

  \end{proposition}

Assuming that Proposition \ref{prop expansion} holds, we can easily prove the desired result, Proposition \ref{prop green}. 
\begin{proof}[Proof of Proposition \ref{prop green}]\label{Proof of Proposition green}
Since $|\E\Theta(t,H(t))|\le C N^{1/2}\Psi^2$, integrating $\dd F(\frX)$ from $t=0$ to $t = 4 \log N$ and taking the expectation, we find from Proposition~\ref{prop expansion} that, on~$\Xi$,
\begin{align}
&\left| \E\left[ F \left( N \int_{E_1}^{E_2} \im m( \widehat L_+ +x - 2 + \ii \eta)\Big|_{t=0}\, \dd x \right)\right] - \E \left[F \left( N \int_{E_1}^{E_2} \im m(\widehat L_+ +x - 2 + \ii \eta)\Big|_{t=4 \log N}\, \dd x \right) \right]\right| \nonumber \\
& \qquad\qquad\leq N^{-1/6 + C'\epsilon}\,,
\end{align}
for some constant $C' > 0$, where we used~\eqref{le psi}.

At $t = 4 \log N$, we have $\lambda = \lambda_0 N^{-2}$, hence $\gamma = 1 + O(N^{-2})$ on $\Xi$. In particular, the distribution of $\widetilde H(t)$ with $t=4\log N$ agrees with the distribution of the matrix
\begin{equation}\label{long time evolution}
\frac{\gamma}{N^2} (\lambda_0 V + W) + \gamma \left(1 - \frac{1}{N^4} \right)^{1/2} W^{\GOE}\,.
\end{equation}
Denoting by $\mu_1^{\GOE} \geq \mu_2^{\GOE} \geq \cdots \geq \mu_N^{\GOE}$ the eigenvalues of $W^{\GOE}$, and by $\widetilde\mu_1\ge\widetilde\mu_2\ge\cdots\ge\widetilde\mu_{N}$ the eigenvalues of the matrix in~\eqref{long time evolution}, we have at $t=4\log N$,
$$
|\wt \mu_j - \mu_j^{\GOE}| = \caO_{\Xi}(N^{-2})\,.
$$
Thus, we have that
$$
m(z) = \frac{1}{N} \sum_j \frac{1}{\wt \mu_j - z} =\frac{1}{N} \sum_j \frac{1}{\mu_j^{\GOE} - z} + \caO(N^{-4/3 + \epsilon}) = m^{\GOE}(z) + \caO_{\Xi}(N^{-4/3 + \epsilon})\,,
$$
and, since $\widehat L_+ - 2 = O(N^{-2})$ on $\Xi$, we have
\begin{align}
\int_{E_1}^{E_2} \im m(x + \widehat L_+ - 2 + \ii \eta)\Big|_{t=4 \log N} \dd x - \int_{E_1}^{E_2} \im m^{\GOE} (x + \ii \eta) \dd x = \caO_{\Xi}(N^{-2 + C \epsilon})\,.\nonumber
\end{align}
Hence
\begin{align}
\left| \E\left[ F \left( N \int_{E_1}^{E_2} \im m(x + \widehat L_+ - 2 + \ii \eta)\Big|_{t=4 \log N} \dd x \right)\right] - \E\left[ F \left( N \int_{E_1}^{E_2} \im m^{\GOE} (x + \ii \eta) \dd x \right) \right]\right| \leq N^{-1 + C\epsilon}\,.\nonumber
\end{align}
Using the uniform boundedness of $F$, we obtain
\begin{align*}
 &\left|\E_V\E\left[ F \left( N \int_{E_1}^{E_2} \im m( \widehat L_+ +x - 2 + \ii \eta)\Big|_{t=0} \dd x \right) \right]-\E_V\E\left[ F \left( N \int_{E_1}^{E_2} \im m^{\GOE} (x + \ii \eta) \dd x \right)\right] \right|\nonumber\\
&\qquad\le CN^{-1/6+C'\epsilon}+CN^{-\epsilon'\alpha}\,,
\end{align*}
where we used that $\P_V(\Xi)\le C N^{-\epsilon'\alpha}$ with $\alpha=\min\{\beta_0,\beta_1\}$, by Assumption~\ref{assumption mu_V convergence} and Assumption~\ref{assumption mu_V}. This completes the proof of Proposition \ref{prop green}
\end{proof}

\subsection{Green function flow}\label{subsection: green function flow}
Recall that we let $\dd F(\frX)=\Theta\dd t+\dd M$. To prove Proposition~\ref{prop expansion}, we first describe $\Theta(t,H(t))$, where $H(t)$ is the solution to~\eqref{le DBM}, in terms of the entries of $G(t,z)$.

\begin{lemma} \label{lem:E theta}\label{lem:dot z}
Under the assumptions of Proposition \ref{prop expansion} we have 
\begin{align} \label{EF G_ii expand}
\E [\Theta(t&,H(t))] \nonumber \\
&= \sum_{i, a} \left(-\partial_t(\lambda\gamma) v_a \E \left[ F'(\frX) \left( \im \int_{E_1 -2}^{E_2 -2} G_{ia}G_{ai} \dd y \right) \right] + \dot z\,\E \left[F'(\frX) \left( \im \int_{E_1 -2}^{E_2 -2} G_{ia}G_{ai} \dd y \right) \right] \right) \nonumber \\
&\quad+ \frac{2\dot\gamma\gamma}{{ N}} \sum_{i}\sum_{a\not= b} \left(\E \left[F'(\frX) \left( \im \int_{E_1 -2}^{E_2 -2} G_{ia}G_{ab}G_{bi} \dd y \right) \right] + \E \left[ F'(\frX) \left( \im \int_{E_1 -2}^{E_2 -2} G_{ia}G_{bb}G_{ai} \dd y \right) \right] \right) \nonumber \\
&\quad+ \frac{2\dot\gamma\gamma}{{ N}} \sum_{i,j}\sum_{a\not= b} \E \left[ F''(\frX) \left( \im \int_{E_1 -2}^{E_2 -2} G_{ja} G_{bj} \dd y \right) \left( \im \int_{E_1 -2}^{E_2 -2} G_{ia}G_{bi} \dd y \right) \right] +\caO_{\Xi}(N^{1/2}\Psi^2)\,,
\end{align}
where we abbreviate $G \equiv G(\widehat L_+ + y + \ii \eta)$. Moreover, recalling that $z\equiv z(t)=\widehat L_+(t)+y+\ii\eta\in\caE_\epsilon(t)$, we have 
\begin{align}\label{le dot z}
 \dot z(t)=-2\gamma\dot\gamma \widehat m_{fc}(\widehat L_+) +\gamma^2\partial_t(\lambda\gamma)\frac{1}{N}\sum_{i=1}^N\frac{v_i}{(\lambda\gamma v_i-\tau)^2}\,,
\end{align}
for $t\ge 0$ on $\Xi$.
\end{lemma}

We prove Lemma \ref{lem:E theta} in Subsection~\ref{pf first}. To illustrate the essence of the proof, we first consider the differentials~$\dd G_{ij}$ of the Green function $G(t,z)$. Using It\={o}'s lemma, we compute
\begin{align}\label{from ito}
 \dd G_{ij} &=\frac{\partial G_{ij}}{\partial t}\dd t+\sum_{a\le b}\frac{\partial G_{ij}}{\partial h_{ab}}\dd h_{ab}+\sum_{a\le b}\sum_{c\le d}\frac{1}{2}\frac{\partial^2 G_{ij}}{\partial h_{ab}\partial h_{cd}}\,\dd\langle h_{ab}, h_{cd}\rangle_t\nonumber\\
&=\frac{\partial G_{ij}}{\partial t}\dd t+\sum_{a< b}\frac{\partial G_{ij}}{\partial h_{ab}}\left(\frac{\dd\beta_{ab}}{\sqrt{N}}-\frac{h_{ab}}{2}\dd t \right)
-\frac{1}{2}\sum_{a}\frac{\partial G_{ij}}{\partial h_{aa}}{h_{aa}}\dd t \nonumber\\
&\qquad+\sum_{a\le b}\sum_{c\le d}\frac{1}{2}\frac{\partial^2 G_{ij}}{\partial h_{ab}\partial h_{cd}}\,\dd\langle h_{ab}, h_{cd}\rangle_t\,.
\end{align}
The cross-variance in~\eqref{from ito} is explicitly given by
\begin{align}\label{cross bracket}
 \langle h_{ab}, h_{ab}\rangle_t&=\langle h_{ab},h_{ba}\rangle_t=
	\begin{cases}
	\frac{1}{N} & \text{ if } a \neq b\,, \\
	0 & \text{ if } a = b\,,
	\end{cases}
\end{align}
with $\langle h_{ab}, h_{cd}\rangle_t=0$ if $\{a,b\}\not=\{c,d\}$. 
Using~\eqref{cross bracket} and the symmetries $h_{ab} = h_{ba}$, we obtain from~\eqref{from ito} that

\begin{align}\label{from ito 2}
 \dd G_{ij}=\left(\frac{\partial G_{ij}}{\partial t}- \frac{1}{2} \sum_a \frac{\partial G_{ij}}{\partial h_{aa}} h_{aa} - \frac{1}{4} \sum_{a\neq b}\frac{\partial G_{ij}}{\partial h_{ab}}h_{ab} +\frac{1}{4N}\sum_{a\not= b} \frac{\partial G_{ij}^2}{\partial h_{ab}^2}\right)\dd t+\dd M_{ij}\,,
\end{align}
where we have introduced the martingale term
\begin{align}\label{le baby martingale}
\dd M_{ij} \deq\sum_{a< b}\frac{\partial G_{ij}}{\partial h_{ab}}\frac{\dd \beta_{ab}}{\sqrt{N}}\,.
\end{align}
Next, we compute the derivatives in~\eqref{from ito 2}. For the time derivative we obtain
\begin{align*}
\frac{\partial G_{ij}}{\partial t}=\sum_{a}\left(-\dot\gamma\,G_{ia} \lambda v_{a}G_{aj}+\dot z\,G_{ia}G_{aj}\right)-\sum_{a\neq b}\left(\dot\gamma\,G_{ia} w_{ab}G_{bj}\right)\,.
\end{align*}
For the first spatial derivate we obtain, for $a\not=b$, 
\begin{align*}
 -\frac{\partial G_{ij}}{\partial h_{aa}}h_{aa}=\lambda\gamma v_aG_{ia}G_{aj} \,,\qquad\qquad-\frac{\partial G_{ij}}{\partial h_{ab}}h_{ab}=2{\gamma} G_{ia} w_{ab} G_{bj}\,.
\end{align*}
For the second spatial derivatives we find, for $a\not=b$, 
\begin{align*}
 \frac{\partial^2 G_{ii}}{\partial h_{ab}^2}=2{\gamma^2}\left(G_{ia}G_{ab}G_{bj}+G_{ib}G_{ba}G_{aj}+G_{ib}G_{aa}G_{bj}+G_{ia}G_{bb}G_{aj} \right)\,.
\end{align*}
Thus, using $\partial_t(\lambda\gamma)=-\lambda\gamma/2+\dot\gamma\lambda$, we can rewrite~\eqref{from ito 2} as
\begin{align} \label{G_ii expand general}
\dd G_{ij} &=\sum_{a}\bigg(-\partial_t(\gamma\lambda) v_{a}G_{ia}  G_{aj}+\dot z\,G_{ia}G_{aj} \bigg)\dd t+\dd M_{ij}\nonumber \\
&\quad + \sum_{a\not= b} \left(-\dot\gamma\,G_{ia} w_{ab}G_{bj}+ \frac{\gamma}{2}\, G_{ia}w_{ab} G_{bj}+\frac{{\gamma^2}}{N} G_{ia}G_{ab}G_{bj}+\frac{{\gamma^2}}{N}G_{ia}G_{bb}G_{aj}\right)\dd t\,. 
\end{align}

\begin{example} \label{ex:wigner}
In the simple case where $\lambda=\lambda_0 = 0$ and $\gamma=\gamma_0 = 1$, we have $\dot z = \partial_t \widehat L_+ = 0$ and Equation \eqref{G_ii expand general} becomes
\begin{align} \label{G_ii expand 2}
\dd G_{ij}&=
\sum_{a , b} \left( \frac{1}{2} G_{ia}w_{ab}G_{bj} + \frac{1}{N} G_{ia}G_{ab}G_{bj} + \frac{1}{N} G_{ia}G_{bb}G_{aj} \right) \dd t+\dd M_{ij}\,.
\end{align}
Note that in this simple example $H(t=0)$ reduces to a real symmetric Wigner matrix (with vanishing diagonal) and we have $\widehat m_{fc}\equiv m_{sc}$, where $m_{sc}$ is the Stieltjes transform of the standard semicircle law~$\rho_{sc}$.

Eventually, we are going to take the expectation of~\eqref{G_ii expand 2}. To compute the expectation of $G_{ia}w_{ab}G_{bj}$, we use the following lemma that was used in the context of random matrix theory before in~\cite{KKP}, see also~\cite{Su,So09}. For a function~$Q$ of the matrix entry~$h_{ab}$, we denote $\partial_{ab}^mQ\equiv\frac{\partial^m}{\partial h_{ab}^m}Q$, $m\in\N$.

\begin{lemma}\label{stein lemma}
 Assume that $Q\in C^{M+1}(\R)$ for some $M\in\N$. Then,
\begin{align}\label{le ansatz}
\E \partial_{{ab}}Q(h_{ab})h_{ab}=\sum_{m=1}^{M}\frac{ \kappa_{ab}^{(m)}}{(m-1)!}\E [\partial^m_{{ab}}Q(h_{ab})]+ O(\|\partial_{{ab}}^{M+1}Q\|_\infty \kappa_{ab}^{(M+1)})\,,
\end{align}
where $(\kappa_{ab}^{(m)})$, $m\in\N$,  are the cumulants of $(h_{ab})$.
\end{lemma}
\begin{proof}
 By assumption it suffices to check~\eqref{le ansatz} for monomials up to order $M$. For monomials~\eqref{le ansatz} is a direct consequence of the moment-cumulant relation
\begin{align}
 \kappa_{ab}^{(n)}=M_{ab}^{(n)}-\sum_{m=1}^{n-1}\binom{n-1}{m-1}\kappa_{ab}^{(n)}M_{ab}^{(n-m)}\,,\nonumber
\end{align}
where $(M_{ab}^{(n)})$, $(\kappa_{ab}^{(n)})$  are the moments, respectively cumulants of $(h_{ab})$.	
\end{proof}
Note that we have, by Assumption~\ref{assumption wigner} and the definition of the cumulants, for $a\not=b$,
\begin{align}
 \kappa^{(1)}_{ab}= 0 \,,\qquad\quad \kappa_{ab}^{(2)}=\frac{1}{N}\,,\qquad \quad\kappa_{ab}^{(p)}\le \frac{k_p}{N^{p/2}}\,,\qquad (k\ge 3)\,,\nonumber
\end{align}
for constants $(k_p)$ independent of $N$.

Choosing $Q\equiv G_{ij}$, we get from Lemma~\ref{stein lemma}, for $a\not=b$,
\begin{align}\label{Q expansion}
 \E \,\frac{\partial G_{ij}}{\partial h_{ab}}h_{ab}=\frac{1}{N}\E\frac{\partial^2 G_{ij}}{\partial h_{ab}^2}+\frac{1}{2}\kappa_{ab}^{(3)}\E\frac{\partial^3 G_{ij}}{\partial h_{ab}^3}+\caO(N^{-2}{\Psi^2})\,.
\end{align}

The first term on the right side of~\eqref{Q expansion} can be handed with the following lemma whose proof appeared first in~\cite{EYY}.
\begin{lemma}\label{baby lemma}
 For $i\neq a \neq b \neq j$, we have
\begin{align}\label{estimate in baby lemma}
\left|\E\,\frac{\partial^3 G_{ij}}{\partial h_{ab}^3}\right|\prec \Psi^3\,.
\end{align}

\end{lemma}

\begin{proof}
 To show~\eqref{estimate in baby lemma} it clearly suffices to control $\frac{\partial^2 }{ \partial w_{ab}^2 }(G_{ia} G_{bj})$. Observe that each  term in this last expression contains at least three off-diagonal resolvent entries except $G_{ia} G_{aa} G_{bb} G_{bj}$ and $G_{ib}G_{aa}G_{bb}G_{aj}$. We will focus on the former term, the latter can be treated in the very same way. Using $|G_{aa}-m_{sc}|\prec \Psi$ (see Theorem~\ref{local law}, with $\lambda=0$, $\gamma=1$) we get
\begin{align}\label{in baby lemma}
 \E\, G_{ia} G_{aa} G_{bb} G_{bj}=m_{sc}^2\E\, G_{ia}G_{bj}+\caO(\Psi^3)\,.
\end{align}
Using the resolvent identity~\eqref{twosided}, we may write
\begin{align*}
 G_{ia}=-G_{aa} G_{ii}^{(a)}\left (h_{ia} - \sum_{p, q}^{(a)} h_{ip} G_{pq}^{(a)} h_{qa}\right) =m_{sc}^2\left (h_{ia} - \sum_{p, q}^{(a)} h_{ip} G_{pq}^{(a)} h_{qa}\right) + \caO(\Psi^2)\,,
\end{align*}
where we used once more the local law~\eqref{local 2}. Since the first term on the very right side has vanishing expectation, we obtain from~\eqref{in baby lemma} that $| \E\,G_{ia} G_{aa} G_{bb} G_{bj}|\prec \Psi^3$ which implies the claim.
\end{proof}
Returning to~\eqref{Q expansion}, we obtain, for $i\neq a \neq b \neq j$, 
\begin{align*}
\E \frac{\partial G_{ij}}{\partial w_{ab}}w_{ab}&=- \frac{1}{N}\E G_{ia}G_{ab}G_{bj}- \frac{1}{N}\E G_{ib}G_{ab}G_{aj} -\frac{1}{N}\E G_{ib}G_{aa}G_{bj} \nonumber\\
&\qquad-\frac{1}{N}\E G_{ia}G_{bb}G_{aj} +\caO(N^{-3/2}\Psi^3)\,.
\end{align*}
In sum, we have shown that
\begin{align*}
 \sum_{a\neq b}\E G_{ia}w_{ab}G_{bj}&=-\frac{2}{N} \sum_{a\neq b} \left( \E G_{ia}G_{ab}G_{bj} + \E G_{ib}G_{aa}G_{bj} \right)+\caO(N^{1/2}\Psi^3)\,,
\end{align*}
 where we used $|G_{ii}w_{ib}G_{bj}|\,,|G_{ia}w_{aj}G_{jj}|\,,|\E_iG_{ii}w_{ij}G_{jj} |\prec N^{-1/2}\Psi$, to cope with the cases $a=i$, $a\not=b$, etc.. This shows that~\eqref{G_ii expand 2} can be written as
\begin{align*}
 \E\dd G_{ij}(z, t) = \caO(N^{1/2}\Psi^3)\dd t+\caO(\Psi^2)\dd t+\caO(N^{-1/2}\Psi)\dd t=\caO(N^{1/2}\Psi^3)\dd t\,,
\end{align*}
(uniformly in $t\ge 0$), where we used that the expectation of the martingale term defined in~\eqref{le baby martingale} vanishes. Integration over $t$ from $0$ to $4 \log N$, leads to
\begin{align*}
\E \left( \frac{1}{W -z} \right)_{ij} - \E \left( \frac{1}{W^{\GOE} -z} \right)_{ij} = \caO(N^{1/2}\Psi^3)\,,
\end{align*}
which is stronger an estimate than the trivial bound $\caO(\Psi)$ obtained from the local laws in Theorem~\ref{local law}.
\end{example}

\subsection{Computation of $\dd F(\frX)$ and proof of Lemma \ref{lem:E theta}} \label{pf first}
We now turn to the computation of the differential $\dd F(\frX)$, where $F$ is a smooth function satisfying~\eqref{F bound} and where $\frX$ is defined in~\eqref{le definition of le X}. Choosing $i=j$ in~\eqref{G_ii expand general}, we get
\begin{align}\label{G_ii expand general 2}
 \dd G_{ii} &=\sum_{a}\bigg(-\partial_t(\gamma\lambda) v_{a}G_{ia}  G_{ai}+\dot z\,G_{ia}G_{ai} \bigg)\dd t+\dd M_{ii}\nonumber \\
&\quad + \sum_{a\not= b} \left(-\dot\gamma\,G_{ia} w_{ab}G_{bi}+ \frac{\gamma}{2}\, G_{ia}w_{ab} G_{bi}+\frac{\gamma^2}{N} G_{ia}G_{ab}G_{bi}+\frac{\gamma^2}{N}G_{ia}G_{bb}G_{ai}\right)\dd t\,. 
\end{align}
with the martingale term
\begin{align}
\dd M_{ii} = \sum_{a< b}\frac{\partial G_{ii}}{\partial h_{ab}}\frac{\dd \beta_{ab}}{\sqrt{N}} = \frac{\gamma}{\sqrt N} \sum_{a\not= b} G_{ia} G_{bi} \dd \beta_{ab}\,.\nonumber
\end{align}
Recalling the definitions of $\frX$ in Proposition~\ref{prop expansion}, we obtain from It\={o}'s lemma and~\eqref{G_ii expand general 2},
\begin{align} \label{F G_ii expand 1}
\dd F(\frX) &= F'(\frX) \sum_{i, a}\bigg( \im \int_{E_1 -2}^{E_2 -2}\dd y\, \Big( -\partial_t(\gamma\lambda)\,v_aG_{ia}  G_{ai}+\dot z\,G_{ia}G_{ai} \Big) \bigg)\dd t\nonumber \\
&\qquad + F'(\frX)  \sum_{i}\sum_{a\not= b} \left( \im \int_{E_1 -2}^{E_2 -2}\dd y\, \Big( \left(\frac{\gamma}{2}-\dot\gamma\right)\,G_{ia} w_{ab}G_{bi}+\frac{\gamma^2}{N}G_{ia}G_{ab}G_{bi}+\frac{\gamma^2}{N} G_{ia}G_{bb}G_{ai} \Big)  \right)\dd t \nonumber \\
&\qquad + F''(\frX) \frac{\gamma^2}{{ N}} \sum_{i, j}\sum_{a\not= b} \left( \im \int_{E_1 -2}^{E_2 -2}\dd y\, G_{ia}  G_{bi} \right) \left( \im \int_{E_1 -2}^{E_2 -2}\dd y\, G_{ja} G_{bj}  \right) \dd t + \dd M \,,
\end{align}
for some martingale $M$ of vanishing expectation. Here, we use the notation $G \equiv G(y + L_+ + \ii \eta)$. We remark that~\eqref{F G_ii expand 1} gives rise to the definitions of $\Theta$ and $M$ in $\dd F(\frX) = \Theta \dd t + \dd M$ in Proposition~\ref{prop expansion}.

Next, we take the expectation in~\eqref{F G_ii expand 1}. The resulting expression can be treated following the lines of Example~\ref{ex:wigner}:
For $a \neq b$, set
\begin{align}\label{definition of the Q}
R (w_{ab})\deq F'(\frX) G_{ia}G_{bi}\,.
\end{align}
The following lemma bounds $R$. 
\begin{lemma} \label{lem:Q'' estimate}
Let $R(w_{ab})\deq F'(\frX) G_{ia}G_{bi}$. Then, for $i \neq a \neq b \neq i$, we have
\begin{align}
\E \frac{\partial^2 R}{\partial w_{ab}^2} = \caO_{\Xi}(\Psi^3)\,.
\end{align}
Here $\Xi$ denotes the event defined in Definition~\ref{definition of the event omega}.
\end{lemma}
Lemma~\ref{lem:Q'' estimate} is proven in the same way as Lemma~\ref{baby lemma}, but its proof is lengthier due to more notation and is therefore postponed to the Appendix~\ref{appendix II}.

From Lemma \ref{lem:Q'' estimate}, with $i \neq a \neq b \neq i$, we obtain
\begin{align*}
\E [R(w_{ab})w_{ab}] 
&= -\frac{{\gamma}}{N}\Big( 2\E \left[F'(\frX) G_{ia}G_{ab}G_{bi} \right] +\E \left[ F'(\frX) G_{ib}G_{aa}G_{bi} \right] +\E \left[ F'(\frX) G_{ia}G_{bb}G_{ai} \right] \Big) \nonumber\\
&\quad -2\frac{\gamma}{N} \E \left[ F''(\frX) \sum_j \left( \im \int_{E_1 -2}^{E_2 -2} \dd \wt y\,\wt G_{ja} \wt G_{bj}   \right) G_{ia}G_{bi} \right]+\caO_{\Xi}(N^{-3/2}\Psi^3)\,,
\end{align*}
where we use the notation $\wt G \equiv G(\widehat L_++\wt y  + \ii \eta)$, respectively $G \equiv G(\widehat L_++y + \ii \eta)$.  Altogether, we have that
\begin{align}\label{altogether gleichung}
\sum_{a\neq b} \E [F'(\frX) G_{ia}w_{ab}G_{bi}]&= -\frac{2\gamma}{N} \sum_{a\not= b} \Big( \E [F'(\frX) G_{ia}G_{ab}G_{bi}] + \E [F'(\frX)G_{ib}G_{aa}G_{bi}] \Big)\nonumber \\
&\quad - \frac{2 \gamma}{N} \sum_{j}\sum_{a\not= b} \E \left[ F''(\frX) \left( \im \int_{E_1 -2}^{E_2 -2} \dd \wt y\, \wt G_{ja} \wt G_{bj} \right) G_{ia}G_{bi} \right] +\caO_{\Xi}(N^{1/2}\Psi^3)\,,
\end{align}
uniformly in $t\ge 0$. Next we prove Lemma~\ref{lem:E theta}.
\begin{proof}[Proof of Lemma~\ref{lem:E theta}]
Combining~\eqref{altogether gleichung} with~\eqref{F G_ii expand 1}, we obtain~\eqref{EF G_ii expand}.  We remark that, after integrating over the interval $[E_1 -2, E_2 -2]$ and summing over the index $i$, we get an additional factor $N \Psi^2$. Thus, the error term of order $\caO_{\Xi}(N^{1/2}\Psi^3)$ in~\eqref{altogether gleichung} becomes $\caO_{\Xi}(N^{3/2}\Psi^5)$, which is equivalent to $\caO_{\Xi}(N^{1/2}\Psi^2)$.

It remains to prove~\eqref{le dot z}. Recall the definition of $\tau$ in~\eqref{tau}. Using that $\widehat L_+=\tau-\gamma^2 \widehat m_{fc}(\widehat L_+)$, we compute
\begin{align*}
 \partial_t \widehat L_+&=\partial_t\tau-2\gamma\dot\gamma \widehat m_{fc}(\widehat L_+)-\gamma^2\partial_t \widehat m_{fc}(\widehat L_+) \nonumber\\
&=\partial_t\tau-2\gamma\dot\gamma \widehat m_{fc}(\widehat L_+)+\gamma^2\frac{1}{N}\sum_{j=1}^N\frac{1}{(\lambda\gamma v_j- \tau)^2}(\partial_t(\lambda\gamma)v_j-\partial_t\tau)\,.
\end{align*}
From~\eqref{upper edge}, we find
\begin{align*}
 \frac{1}{N}\sum_{j=1}^N\frac{1}{(\lambda\gamma v_j-\tau)^2}=\frac{1}{\gamma^2}\,,
\end{align*}
which, after differentiation with respect to $t$, yields
\begin{align*}
 \partial_t \widehat L_+&=\partial_t\tau-2\gamma\dot\gamma \widehat m_{fc}(\widehat L_+)+\gamma^2\partial_t(\lambda\gamma)\frac{1}{N}\sum_{j=1}^N\frac{v_j}{(\lambda\gamma v_j-\tau)^2}-\partial_t\tau\nonumber\\
&=-2\gamma\dot\gamma \widehat m_{fc}(\widehat L_+)+\gamma^2\partial_t(\lambda\gamma)\frac{1}{N}\sum_{j=1}^N\frac{v_j}{(\lambda\gamma v_j-\tau)^2}\,.
\end{align*}
This shows~\eqref{le dot z} and thus completes the proof of Lemma~\ref{lem:E theta}.
\end{proof}
To conclude this section, we return to Example~\ref{ex:wigner}, where we considered $\lambda=\lambda_0 = 0$ and $\gamma=\gamma_0 = 1$. Setting $\lambda=\dot\gamma=\dot z=0$ in Equation~\eqref{EF G_ii expand}, we find that
\begin{align*}
 &\E [\Theta(t,H(t))] =\caO(N^{1/2} \Psi^2)\,,
\end{align*}
for all $t\ge 0$. Integration of $\dd F(\frX)=\caO(N^{1/2} \Psi^2)\dd t$ from $t=0$ to $t=4\log N$ as in the proof of Proposition~\ref{Proof of Proposition green} yields now a simple proof of the fact that the distribution of the largest eigenvalue of  Wigner matrix is given by the Tracy-Widom distribution. 

If, however, $\lambda_0\not=0$, $\gamma_0\not=1$, and thus $\dot\gamma\not=0$, $\dot z\not=0$, the leading terms on the right side of~\eqref{EF G_ii expand} are a priori of order one.  In the remaining sections, we are going to show that these terms cancel for our choices of $\gamma$ and $\dot z$ up to errors of order $N^{1/2}\Psi^2$. Since this cancellation mechanism in the Green function flow involves rather subtle computations, we first present the main ideas in Section~\ref{sec:fluctuation} for the simple case $\lambda_0\ll 1$.

\section{Green function flow - a simple case} \label{sec:fluctuation}

In this section, we assume that $\lambda_0 = N^{-\delta}$ for some $\delta > 0$, i.e., we consider $H = N^{-\delta} V + W$. For such small $\lambda$, we get from~\eqref{gamma} that $\gamma_0=1+O(\lambda)$. As shown below, we may set $\gamma_0\equiv\gamma\equiv 1$ for simplicity of the exposition since the error term of order $\lambda$ is negligible. Furthermore, we let $F' \equiv 1$ so that the conclusion of Proposition \ref{prop expansion} becomes
\begin{align} \label{EF expand simple}
\E [\Theta] = \frac{1}{N} \sum_{i, a} \im \int_{E_1 -2}^{E_2 -2} \bigg(-\dot\lambda v_a \E [G_{ia}G_{ai}] +\dot z\,\E [G_{ia}G_{ai}] \bigg)\dd y +\caO_{\Xi}(N^{1/2}{\Psi^3})\,.
\end{align}
In this section we prove that
\begin{align} \label{EF bound'}
\im \sum_a \left( -\dot\lambda v_a \E [G_{ia}G_{ai}] +\dot z\,\E [G_{ia}G_{ai}] \right) = \caO_{\Xi}(N^{-\delta} \Psi)\,,
\end{align}
which also implies that $\E[\Theta] = \caO_{\Xi}(N^{-\delta + C \epsilon})$. We remark that the bound in \eqref{EF bound'} is non-trivial in the sense that the naive power-counting from the local law only yields a bound of $\caO_{\Xi}(\Psi)$.

The main difficulty to overcome in the proof of~\eqref{EF bound'} is that the index $a$ appears in the deterministic part $v_a$ as well as the random part $G_{ia} G_{ai}$ in the first term. If we can ``decouple'' the index $a$ from the resolvent entries in the sense that we can choose a (non-random) function $f$ such that 
$$
\sum_a \dot\lambda v_a \E [G_{ia}G_{ai}] = \left( \frac{1}{N} \sum_a f(v_a) \right) \sum_{s=1}^N \E [G_{is} G_{si}] + \caO_{\Xi}(N^{-\delta} \Psi)\,,
$$
we can prove~\eqref{EF bound'} by comparing the coefficient $N^{-1} \sum_a f(v_a)$ with $\dot z$ in the second term in \eqref{EF bound'}. In Subsection~\ref{subsection: le expansion of GG} we illustrate the ideas behind this ``decoupling mechanism'' for the index $a$.

\subsection{Expansion of $\E [G_{ia}G_{ai}]$}\label{subsection: le expansion of GG}

We introduce a procedure that renders the indices of the random part $G_{ia} G_{ai}$ free of the index $a$. We proceed in three steps:

\textit{Step 1.} In a first step, we remove the index $a$ in the lower indices of the resolvents. We begin by using the resolvent formulas in~\eqref{onesided} that read
\begin{align*}
G_{ia} = -G_{aa} \sum_s^{(a)} G_{is}^{(a)} h_{sa}\,,\quad \qquad G_{ai} = -G_{aa} \sum_t^{(a)} h_{at} G_{ti}^{(a)},
\end{align*}
where we assume $a \neq i$ at first. Later, we are going to add the term $a=i$, which in fact is negligible for the case at hand $(\lambda_0\ll 1,\gamma=1)$. We then have
\begin{align*}
G_{ia}G_{ai} = G_{aa}^2 \sum_{s, t}^{(a)} G_{is}^{(a)} h_{sa} h_{at} G_{ti}^{(a)}\,.
\end{align*}
Using Schur's complement formula~\eqref{schur}, we rewrite this as
\begin{align*}
G_{ia}G_{ai} = \frac{1}{(\lambda v_a - z - \sum_{p, q}^{(a)} h_{ap} G_{pq}^{(a)} h_{qa})^2} \sum_{s, t}^{(a)} G_{is}^{(a)} h_{sa} h_{at} G_{ti}^{(a)}\,.
\end{align*}
Let $ \tau$ be the largest solution to the equation
\begin{align*}
\frac{1}{N} \sum_j \frac{1}{(\lambda v_j -  \tau)^2} = 1\,.
\end{align*}
Then, from the large deviation estimates in Lemma~\ref{lemma large deviation estimates} and the local laws in~\eqref{local law adapted}, we have
\begin{align*}
z + \sum_{p, q}^{(a)} h_{ap} G_{pq}^{(a)} h_{qa} - \tau = \caO_{\Xi}(\Psi)\,, \quad\qquad z + m^2 - \tau = \caO_{\Xi}(\Psi)\,.
\end{align*}
Thus, we get
\begin{align} \label{step1}
G_{ia}G_{ai} &= \frac{1}{(\lambda v_a - \tau)^2} \sum_{s, t}^{(a)} G_{is}^{(a)} h_{sa} h_{at} G_{ti}^{(a)} + \frac{2}{(\lambda v_a - \tau)^3} \left(z + \sum_{p, q}^{(a)} h_{ap} G_{pq}^{(a)} h_{qa} - \tau \right) \sum_{s, t}^{(a)} G_{is}^{(a)} h_{sa} h_{at} G_{ti}^{(a)} \nonumber \\
&\quad + \caO_{\Xi}(\Psi^4)\,.
\end{align}
Note that the index $a$ does not appear as an lower index of the resolvent terms on the right side, yet every resolvent term has $a$ in the upper index. We remark that the terms of $\caO_{\Xi}(\Psi^4)$ is negligible for $\lambda_0\ll1$.

\textit{Step 2.} In a second step, we integrate out the matrix entries labeled by the index $a$ (i.e., $h_{ax}$ and $h_{ya}$, for some $x$ and $y$) on the right side of~\eqref{step1} by taking the partial expectation $\E_a$ with respect to the $a$-th column and row. More precisely, we consider
\begin{align*}
\E[ G_{ia} G_{ai}] &= \E[ \E_a [G_{ia} G_{ai}]] \nonumber \\
&= \E\left[ \frac{1}{(\lambda v_a - \tau)^2} \E_a \left[ \sum_{s, t}^{(a)} G_{is}^{(a)} h_{sa} h_{at} G_{ti}^{(a)} \right] \right] \\
&\qquad + \E\left[ \frac{2}{(\lambda v_a - \tau)^3} \E_a \left[ \left(z + \sum_{p, q}^{(a)} h_{ap} G_{pq}^{(a)} h_{qa} - \tau \right) \sum_{s, t}^{(a)} G_{is}^{(a)} h_{sa} h_{at} G_{ti}^{(a)} \right] \right] + \caO_{\Xi}(\Psi^4)\,. \nonumber
\end{align*}
In the first term we have
\begin{align*}
\E_a \left[ \sum_{s, t}^{(a)} G_{is}^{(a)} h_{sa} h_{at} G_{ti}^{(a)} \right] = \frac{1}{N} \sum_s^{(a)} G_{is}^{(a)} G_{si}^{(a)}\,,
\end{align*}
and, similarly in the second term, we have
\begin{align*}
&\E_a \left[ \left(z + \sum_{p, q}^{(a)} h_{ap} G_{pq}^{(a)} h_{qa} - \tau \right) \sum_{s, t}^{(a)} G_{is}^{(a)} h_{sa} h_{at} G_{ti}^{(a)} \right] \\
&\qquad\quad= (z + m^{(a)} - \tau) \frac{1}{N} \sum_s^{(a)} G_{is}^{(a)} G_{si}^{(a)} + \frac{2}{N^2} \sum_{p, q, s}^{(a)} G_{ip}^{(a)} G_{pq}^{(a)} G_{qi}^{(a)} + \caO_{\Xi}(\Psi^4)\,, \nonumber
\end{align*}
where the first term comes from the case $p=q$ and $s=t$, while the second term from $p=s$ and $q=t$ or $p=t$ and $q=s$. Here we also used the fact that $G$ is symmetric and that the contribution from the case $p=q=s=t$ is negligible. We thus have
\begin{align} \label{step2}
\E[ G_{ia} G_{ai}] &= \frac{1}{(\lambda v_a - \tau)^2} \E\left[ \frac{1}{N} \sum_s^{(a)} G_{is}^{(a)} G_{si}^{(a)} \right] + \frac{2}{(\lambda v_a - \tau)^3} \E\left[ (z + m^{(a)} - \tau) \frac{1}{N} \sum_s^{(a)} G_{is}^{(a)} G_{si}^{(a)} \right] \nonumber\\
&\qquad + \frac{4}{(\lambda v_a - \tau)^3} \E\left[ \frac{1}{N^2} \sum_{p, q, s}^{(a)} G_{ip}^{(a)} G_{pq}^{(a)} G_{qi}^{(a)} \right] + \caO_{\Xi}(\Psi^4)\,. 
\end{align}
Note that, at the end of this second step, only resolvent terms remain. Also note that the index $a$ appears now as an upper index.

\textit{Step 3.} In a third step, we remove the upper index $a$ in the resolvent entries by using the formula~\eqref{basic resolvent} that reads
\begin{align*}
G_{is}^{(a)} = G_{is} - \frac{G_{ia} G_{as}}{G_{aa}}\,.
\end{align*}
Recall that $G_{ia} G_{as} = \caO_{\Xi}(\Psi^2)$, if $a \neq i, s$. 

We thus have for the first term in \eqref{step2} that
\begin{align*}
\E\left[ \frac{1}{N} \sum_s^{(a)} G_{is}^{(a)} G_{si}^{(a)} \right] &=\frac{1}{N}\sum_s^{(a)} \E\left[ G_{is} G_{si}- \frac{G_{ia} G_{as}}{G_{aa}} G_{si}^{(a)} - G_{is} \frac{G_{sa} G_{ai}}{G_{aa}}\right]. \nonumber
\end{align*}
In the first term of the right side of the last equation, we notice that
\begin{align*}
\frac{1}{N} \sum_s^{(a)} G_{is} G_{si} = \frac{1}{N} \sum_s G_{is} G_{si} + \caO_{\Xi}(\Psi^5)\,.
\end{align*}
Thus, we arrive at
\begin{align} \label{step3}
\E[ G_{ia} &G_{ai}] \nonumber \\
&= \frac{1}{(\lambda v_a - \tau)^2} \E\left[ \frac{1}{N} \sum_s G_{is} G_{si} \right] - \frac{1}{(\lambda v_a - \tau)^2} \E\left[ \frac{1}{N} \sum_s^{(a)} \frac{G_{ia} G_{as}}{G_{aa}} G_{si}^{(a)} \right]\nonumber \\
&\quad - \frac{1}{(\lambda v_a - \tau)^2} \E\left[ \frac{1}{N} \sum_s^{(a)} G_{is} \frac{G_{sa} G_{ai}}{G_{aa}}\right] + \frac{2}{(\lambda v_a - \tau)^3} \E\left[ (z + m^{(a)} - \tau) \frac{1}{N} \sum_s^{(a)} G_{is}^{(a)} G_{si}^{(a)} \right] \nonumber \\
&\quad + \frac{4}{(\lambda v_a - \tau)^3} \E\left[ \frac{1}{N^2} \sum_{p, q, s}^{(a)} G_{ip}^{(a)} G_{pq}^{(a)} G_{qi}^{(a)} \right] + \caO_{\Xi}(\Psi^4)\,. 
\end{align}
Note that the first term on the right side of \eqref{step3} neither has $a$ as a lower nor an upper index.

After following \textit{Steps 1-3}, we obtain a term we desire: the first term on the right side of \eqref{step3}. When a term contains~$a$ neither in the lower index nor in the upper index, as in the first term of \eqref{step3}, we call it fully expanded. For the other terms in \eqref{step3}, we repeat \textit{Steps 1-3} until every non-negligible term is fully expanded.

For example, we apply \textit{Step 1} to the second term in \eqref{step3} to get
\begin{align*}
\frac{1}{N} \sum_s^{(a)} \frac{G_{ia} G_{as}}{G_{aa}} G_{si}^{(a)} &= \frac{1}{N} \sum_s^{(a)} \frac{1}{G_{aa}} G_{aa}^2 \sum_{k, m}^{(a)} G_{ik}^{(a)} h_{ka} h_{am} G_{ms}^{(a)} G_{si}^{(a)} \\
&= \frac{1}{\lambda v_a - \tau} \frac{1}{N} \sum_{s, k, m}^{(a)} G_{ik}^{(a)} h_{ka} h_{am} G_{ms}^{(a)} G_{si}^{(a)} + \caO_{\Xi}(\Psi^4)\,. \nonumber
\end{align*}
By taking the partial expectation, i.e., from \textit{Step 2}, we find that
\begin{align*}
\E \left[ \frac{1}{N} \sum_s^{(a)} \frac{G_{ia} G_{as}}{G_{aa}} G_{si}^{(a)} \right] = \frac{1}{\lambda v_a - \tau} \E \left[ \frac{1}{N^2} \sum_{s, k}^{(a)} G_{ik}^{(a)} G_{ks}^{(a)} G_{si}^{(a)} \right] + \caO_{\Xi}(\Psi^4)\,.
\end{align*}
Since $G_{is} - G_{is}^{(a)} = \caO_{\Xi}(\Psi^2)$, we find after performing \textit{Step 3} that
\begin{align*}
\frac{1}{(\lambda v_a - \tau)^2} \E \left[ \frac{1}{N} \sum_s^{(a)} \frac{G_{ia} G_{as}}{G_{aa}} G_{si}^{(a)} \right] &= \frac{1}{(\lambda v_a - \tau)^3} \E \left[ \frac{1}{N^2} \sum_{s, k}^{(a)} G_{ik}^{(a)} G_{ks}^{(a)} G_{si}^{(a)} \right] + \caO_{\Xi}(\Psi^4) \\
&= \frac{1}{(\lambda v_a - \tau)^3} \E \left[ \frac{1}{N^2} \sum_{s, k} G_{ik} G_{ks} G_{si} \right] + \caO_{\Xi}(\Psi^4)\,. \nonumber
\end{align*}
We go through the same procedure for the third term in \eqref{step3}. Since it contains $G_{is}$, which does not have $a$ in the indices, we begin by
\begin{align*}
\frac{1}{N} \sum_s^{(a)} G_{is} \frac{G_{sa} G_{ai}}{G_{aa}} = \frac{1}{N} \sum_s^{(a)} G_{is}^{(a)} \frac{G_{sa} G_{ai}}{G_{aa}} + \caO_{\Xi}(\Psi^4)\,.
\end{align*}
Then, after following \textit{Steps 1-3} again, we obtain that
\begin{align*}
\frac{1}{(\lambda v_a - \tau)^2} \E \left[ \frac{1}{N} \sum_s^{(a)} G_{is} \frac{G_{sa} G_{ai}}{G_{aa}} \right] = \frac{1}{(\lambda v_a - \tau)^3} \E \left[ \frac{1}{N^2} \sum_{s, k} G_{ik} G_{ks} G_{si} \right] + \caO_{\Xi}(\Psi^4)\,.
\end{align*}
The fourth term and the fifth term in~\eqref{step3} require \textit{Step 3} only, and we can easily see that
\begin{align*}
\frac{2}{(\lambda v_a - \tau)^3} \E\left[ (z + m^{(a)} - \tau) \frac{1}{N} \sum_s^{(a)} G_{is}^{(a)} G_{si}^{(a)} \right] = \frac{2}{(\lambda v_a - \tau)^3} \E\left[ (z + m - \tau) \frac{1}{N} \sum_s G_{is} G_{si} \right] + \caO_{\Xi}(\Psi^4)
\end{align*}
and
\begin{align*}
\frac{4}{(\lambda v_a - \tau)^3} \E\left[ \frac{1}{N^2} \sum_{p, q, s}^{(a)} G_{ip}^{(a)} G_{pq}^{(a)} G_{qi}^{(a)} \right] = \frac{4}{(\lambda v_a - \tau)^3} \E\left[ \frac{1}{N^2} \sum_{p, q, s} G_{ip} G_{pq} G_{qi} \right] + \caO_{\Xi}(\Psi^4)\,.
\end{align*}
Thus, we now have from \eqref{step3} that
\begin{align} \label{step_final}
\E[ G_{ia} G_{ai}] 
&= \frac{1}{(\lambda v_a - \tau)^2} \E\left[ \frac{1}{N} \sum_s G_{is} G_{si} \right] + \frac{2}{(\lambda v_a - \tau)^3} \E \left[ \frac{1}{N^2} \sum_{s, k} G_{ik} G_{ks} G_{si} \right] \nonumber \\
&\qquad + \frac{2}{(\lambda v_a - \tau)^3} \E\left[ (z + m - \tau) \frac{1}{N} \sum_s G_{is} G_{si} \right] + \caO_{\Xi}(\Psi^4)\,,
\end{align}
where every non-negligible term is fully expanded. We remark that the coefficients of the last two terms in \eqref{step_final}, which are of $\caO_{\Xi}(\Psi^3)$ by a naive power-counting, contain the same factor $(\lambda v_a - \tau)^{-3}$. This is not a mere coincidence but an intrinsic structure of the procedure.

\subsection{Proof of Equation~\eqref{EF bound'}}

From \eqref{step_final}, we find that
\begin{align}\label{EF bound' expanded}
&\sum_a \left( -\dot\lambda v_a \E [G_{ia}G_{ai}] +\dot z\,\E [G_{ia}G_{ai}] \right) \nonumber \\
&\qquad= -\dot \lambda \left( \frac{1}{N} \sum_a \frac{v_a}{(\lambda v_a - \tau)^2} \right) \E\left[ \sum_s G_{is} G_{si} \right] + \dot z\, \E \left[ \sum_s G_{is}G_{si} \right] - \dot \lambda v_i \,\E[(G_{ii})^2 ]\nonumber \\
&\qquad\qquad - \dot \lambda \left( \frac{1}{N} \sum_a \frac{2 v_a}{(\lambda v_a - \tau)^3} \right) \E \left[ (z + m - \tau) \sum_s G_{is} G_{si} + \frac{1}{N} \sum_{s, k} G_{ik} G_{ks} G_{si} \right] + \caO_{\Xi}(N^{-\delta} \Psi)\,, 
\end{align}
where we use that $\dot \lambda = O(N^{-\delta})$ and that the contribution from the case $a=i$ in the summation is negligible.

From the explicit computation of $\dot z$ (with $\gamma\equiv 1$) in~\eqref{le dot z}, we see that the first two terms on the right side of~\eqref{EF bound' expanded} add up to zero. After taking the imaginary part of the third term, we find from the local law that
\begin{align*}
 \im \dot \lambda v_i \E[G_{ii}^2 ] = 2 \dot \lambda v_i \E[ \re G_{ii} \cdot \im G_{ii}] = \caO_{\Xi}(N^{-\delta} \Psi)\,.
\end{align*}
Thus, in order to prove \eqref{EF bound'}, it suffices to show that
\begin{align} \label{3 ward}
\E \left[ (z + m - \tau) \sum_s G_{is} G_{si} + \frac{1}{N} \sum_{s, k} G_{ik} G_{ks} G_{si} \right] = \caO_{\Xi}(\Psi)\,.
\end{align}
We remark that the naive size of the left side of~\eqref{3 ward} obtained by power counting is $\caO_{\Xi}(1)$, hence the estimate \eqref{3 ward} is non-trivial. This type of estimate will be referred to as ``optical theorem'' in the sequel. 

To prove~\eqref{3 ward}, we go back to~\eqref{step_final}. After summing over $a$, we have
\begin{align*}
&\E \left[ \frac{1}{N} \sum_a G_{ia} G_{ai} \right] = \left( \frac{1}{N} \sum_a \frac{1}{(\lambda v_a - \tau)^2} \right) \E\left[ \frac{1}{N} \sum_s G_{is} G_{si} \right] \\
&\qquad + \left( \frac{2}{N} \sum_a \frac{1}{(\lambda v_a - \tau)^3} \right) \E \left[ (z + m - \tau) \frac{1}{N} \sum_s G_{is} G_{si} + \frac{1}{N^2} \sum_{s, k} G_{ik} G_{ks} G_{si} \right] + \caO_{\Xi}(\Psi^4)\,. \nonumber 
\end{align*}
Recall that we have set $\gamma = 1$, hence
\begin{equation}\label{le sum relation}
\frac{1}{N} \sum_a \frac{1}{(\lambda v_a - \tau)^2} = 1\,,
\end{equation}
and we obtain the following non-trivial estimate,
\begin{align}
\left( \frac{2}{N} \sum_a \frac{1}{(\lambda v_a - \tau)^3} \right) \E \left[ (z + m - \tau) \frac{1}{N} \sum_s G_{is} G_{si} + \frac{1}{N^2} \sum_{s, k} G_{ik} G_{ks} G_{si} \right] = \caO_{\Xi}(\Psi^4)\,.
\end{align}
Since the coefficient $N^{-1} \sum_a (\lambda v_a - \tau)^{-3} $ is bounded uniformly away from zero on $\Xi$, we find that the optical theorem~\eqref{3 ward} indeed holds. This completes the proof of Equation~\eqref{EF bound'}.

To conclude this section, we mention that the ``optical theorem'' is a consequence of the ``sum rule''~\eqref{le sum relation}: In the expansion of $\sum_{a} G_{ia}G_{ai}$, for $z$ close to the spectral edge, the leading terms cancel due to~\eqref{le sum relation}. In the bulk of the spectrum, the leading terms do not cancel but the expansion can be used to obtain optimal bounds on the average $\frac{1}{N}\sum_{a}G_{ia}G_{ai}$ in the bulk of the spectrum. This mechanism has been studied in details for banded Wigner matrices in~\cite{EKY}.

\section{Proof of Proposition \ref{prop expansion}} \label{sec:proof of prop}

In this section, we prove Proposition \ref{prop expansion} using the following result.

\begin{lemma} \label{lem:G_ii estimate}
For $n \in \N$, let
\begin{align} \label{A_n}
A_n \deq \frac{1}{N} \sum_{j=1}^N \frac{1}{(\lambda \gamma v_j - \tau)^n}\,,\quad \qquad A_n' \deq\frac{1}{N} \sum_{j=1}^N \frac{v_j}{(\lambda \gamma v_j - \tau)^n}\,.
\end{align}
Let $H(t)$ be the solution to~\eqref{le DBM} with initial condition $H(0)=H_0$; see~\eqref{le initial condition}. Let $F$ be a smooth function satisfying~\eqref{F bound} and let $\frX$ be given by~\eqref{le definition of le X}. Let $\Theta\equiv \Theta(t,H(t))$ denote the function in Proposition~\ref{prop expansion}.

Then there exist random variables $X_2\equiv X_2(z)$ and $X_3\equiv X_3(z)$ with $X_2 = \caO_{\Xi}(\Psi^2)$ and $X_3 = \caO_{\Xi}(\Psi^3)$ such that
\begin{align} \label{G_ii estimate}
 \E [\Theta] = N \im \int_{E_1 -2}^{E_2 -2} \E \left[ C_2 N X_2 + C_3 N X_3 + C_0 F'(\frX) + C_0' F'(\frX) (z + \gamma^2 m - \tau) \right] \dd y + \caO_{\Xi}(N^{1/2} \Psi^2)\,,
\end{align}
uniformly in $t\ge 0$, where we use the notation $z \equiv z(t) = \widehat L_+(t) + y + \ii \eta$.

 The coefficients $C_2$, $C_3$, $C_0$, $C_0'$ in~\eqref{G_ii estimate} are functions of $t$ and $\widehat\nu$ only that are explicitly given by
\begin{align}
C_2 &= -\partial_t(\lambda\gamma) \gamma^2 A_2' + \dot z + 2\dot\gamma \gamma A_1 \,,\label{le C_2} \\
C_3 &= 2 \gamma^2 \left( -\partial_t(\lambda\gamma) \left( A_3' - \frac{A_3 A_4'}{A_4} \right) + \dot\gamma \gamma^{-1} \left( \gamma^{-2} - \frac{2 A_3^2}{A_4} \right) \right)\,,\label{le C_3} \\
C_0 &= -\partial_t(\lambda\gamma) \left( A_2' - \frac{A_2 A_4'}{A_4} \right) - 2 \dot\gamma \gamma \frac{A_3}{\gamma^2 A_4}\,,\label{le C_0} \\
C_0' &= -\partial_t(\lambda\gamma) \left( A_3' - \frac{A_3 A_4'}{A_4} \right) + \dot\gamma \gamma^{-1} \left( \gamma^{-2} - \frac{2 A_3^2}{A_4} \right)\,,
\end{align}
with $\dot z = \partial_t\widehat L_+(t)$.
\end{lemma}

\begin{remark}
The random variable $X_2$ in Lemma \ref{lem:G_ii estimate} is defined by
$$
X_2 \deq \frac{1}{N^2} \sum_{i, s} F'(\frX) G_{is} G_{si}\,.
$$
The precise definition of $X_3$ is given in \eqref{X_3} below. (See also \eqref{caX_32} and \eqref{caX_33}.)
\end{remark}

\begin{remark} \label{rem:A_n}
By definition, we have for $n \geq 2$ that
\beq
\lambda A_n' - \tau A_n = A_{n-1}\,.\nonumber
\eeq
We also have
\begin{align} \label{A_n values}
A_1 = \widehat m_{fc}(\widehat L_+)\,, \qquad\quad A_2 = \gamma^{-2}\,, \quad\qquad A_3 = -\gamma^6\,.
\end{align}
See Remark \ref{rem:gamma and tau} for more details.
\end{remark}

We prove Lemma \ref{lem:G_ii estimate} in the Appendix~\ref{appendix V} by using the ideas demonstrated in Section~\ref{sec:fluctuation} and estimates carried out in the Appendices~\ref{appendix III} and~\ref{appendix IV}. Applying the three step expansion procedure of Section~\ref{sec:fluctuation} to the right side of~\eqref{EF G_ii expand}, we obtain the leading order term $N X_2$, which corresponds to the term $\sum_a G_{ia} G_{ai}$ in Section~\ref{sec:fluctuation}. By power counting we have $X_2=\caO_{\Xi}(\Psi^2)$. Continuing the expansion procedure, we obtain the next order term $X_3$, with $X_3 = \caO_{\Xi}(\Psi^3)$. We also compute the third order term, but it can be absorbed into $X_3$ by using a higher order ``optical theorem'', i.e., an extension of~\eqref{3 ward}. The next higher order terms are negligible. Finally, unlike as in the expansion in Section \ref{sec:fluctuation}, the proof of Lemma~\ref{lem:G_ii estimate} requires estimates on the diagonal terms corresponding to the terms from 
the case $i=a$ in \eqref{EF bound' expanded}. These terms 
become the sub-leading, but not negligible, terms $F'(\frX)$ and $F'(\frX) (z + \gamma^2 m - \tau)$ in~\eqref{G_ii estimate}.

\begin{proof}[Proof of Proposition \ref{prop expansion}]
Assuming the validity of~\eqref{G_ii estimate}, it suffices to show that $C_2 = C_3 = C_0'=0$, since $C_0$ and $F'(\frX)$ are real, hence they vanish after taking the imaginary part. 

In \eqref{le dot z}, we showed that
\begin{align}
\dot z = \partial_t \widehat L_+ = -2\dot\gamma \gamma A_1 + \partial_t(\lambda\gamma) \gamma^2 A_2'\,,\nonumber
\end{align}
thus we find from~\eqref{le C_2} that $C_2 = 0$. 

From the definition of $A_2$, we have
\begin{align}
A_2=\frac{1}{N} \sum_a \frac{1}{(\lambda \gamma v_a - \tau)^2} = \gamma^{-2}\,.\nonumber
\end{align}
Taking the partial derivative with respect to $t$, we obtain
\beq\label{C3 relation 1}
\partial_t (\lambda \gamma) A_3' = \dot \tau A_3 + \gamma^{-3} \dot \gamma\,.
\eeq

Similarly, from the definition of $A_3$ we obtain
\beq\label{C3 relation 2}
\partial_t (\lambda \gamma) A_4' = \dot \tau A_4 - \frac{\dot A_3}{3}\,.
\eeq

Thus, combining~\eqref{le C_3},~\eqref{C3 relation 1} and~\eqref{C3 relation 2} we get
\beq \begin{split}
C_3 &= \left( -\partial_t(\lambda\gamma) \left( A_3' - \frac{A_3 A_4'}{A_4} \right) + \dot\gamma \gamma^{-1} \left( \gamma^{-2} - \frac{2 A_3^2}{A_4} \right) \right)\nonumber \\
&= -\dot \tau A_3 - \dot \gamma \gamma^{-3} + \frac{A_3}{A_4} \left( \dot \tau A_4 - \frac{\dot A_3}{3} \right) + \dot \gamma \gamma^{-1} \left( \gamma^{-2} - \frac{2 A_3^2}{A_4} \right)\nonumber \\
&= -\frac{A_3}{3 \gamma A_4} ( \gamma \dot A_3 + 6 \dot \gamma A_3 )\,.
\end{split} \eeq
Since $\gamma^{-6} = -A_3$, we obtain $C_3=0$.

Similarly, one shows that $C_0' = 0$ as well.

Therefore, we have shown that
\begin{align*}
\E [\Theta] = \caO_{\Xi}(N^{1/2} \Psi^2)\,,
\end{align*}
which completes the proof of Proposition \ref{prop expansion}.
\end{proof}

\begin{remark}
From the relation
\begin{align*}
-\partial_t (\lambda \gamma) A_2' + \dot \tau A_2 = \dot A_1 = \partial_t \widehat m_{fc}(\widehat L_+)\,,
\end{align*}
we conclude, together with~\eqref{le C_0} and the identity $\gamma^{-6} = -A_3$, that
\begin{align*}
C_0 = \partial_t \widehat m_{fc}(\widehat L_+) - \dot \tau A_2 + \frac{A_2}{A_4} \left(\dot \tau A_4 - \frac{\dot A_3}{3}\right) - \frac{2 \dot \gamma A_2 A_3}{\gamma A_4} =  \partial_t \widehat m_{fc}(\widehat L_+)\,,
\end{align*}
as was to be expected.
\end{remark}
The proof of Lemma~\ref{lem:G_ii estimate} is divided into three steps. These steps are outlined in the Appendices~\ref{appendix III},~\ref{appendix IV} and~\ref{appendix V}.
\begin{appendix}

\section{}\label{appendix III}
In a first step of the proof of Lemma~\ref{lem:G_ii estimate},  we expand in this appendix the first term on the right side of \eqref{EF G_ii expand}. The aim of this expansion is to decouple the deterministic part $v_a$ from the random resolvent part by deriving an approximation of the form
\begin{align}\label{le aim of the expansion}
  \E [F'(\frX) G_{ia}(z)G_{ai}(z)] = \sum_k\sum_l  f_{k,l} (v_a) \, \E\,[ Y_{k,l,i}(z)] + \caO_{\Xi}(\Psi^5)\,,
\end{align}
for a finite family of deterministic functions $(f_{k,l})$ and a finite family of random variables $(Y_{k,l,i})\equiv (Y_{k,l,i}(z))$, both indexed by natural numbers $l$, $k$ and the index $i$ such that $Y_{k,l,i}=\caO_{\Xi}(\Psi^{k+1})$ for all $l$ and $i$. In~\eqref{le aim of the expansion}, we have implicitly chosen $z=\widehat{L}_+ +y+\ii\eta$, $y\in[-N^{-2/3+\epsilon},N^{-2/3+\epsilon}]$ and we will do so hereafter. Later, we will see that it suffices to consider $k = 1, 2, 3$. 

The purpose of the approximation in~\eqref{le aim of the expansion} is twofold. First, after multiplying~\eqref{le aim of the expansion} by $v_a$, we obtain the expression corresponding to~\eqref{EF bound' expanded}. This enables us to estimate the right side of~\eqref{EF G_ii expand}. Second, we can prove from~\eqref{le aim of the expansion} an optical theorem, which is essential in the proof of Lemma \ref{lem:G_ii estimate}. (The corresponding result for small $\lambda$ was given in~\eqref{3 ward}.)

We remark that $Y_{1,l,i}$ and $Y_{2,l,i}$ can be written in terms of the $z$-dependent random variables
\begin{align}
\caX_{22} &\deq \frac{1}{N} \sum_p G_{ip} G_{pi}\,,\label{le definitionof x22} \\
\caX_{32} &\deq ( z + \gamma^2 m - \tau ) \frac{1}{N} \sum_p G_{ip} G_{pi}\,, \label{caX_32} \\	
\caX_{33} &\deq \frac{1}{N^2} \sum_{p, q} G_{ip} G_{pq} G_{qi} \,. \label{caX_33}
\end{align}
To simplify the notation slightly, we drop the subscript $\Xi$ in $\caO_{\Xi}$, yet we always condition on $\Xi$.

From the local law in~\eqref{local law adapted}, it can be checked that
$$
z + \gamma^2 m - \tau = \caO(\Psi)\,.
$$
Thus, we have the a priori bounds
\begin{align*}
\caX_{22} = \caO(\Psi^2)\,, \qquad \caX_{32} = \caO(\Psi^3)\,, \qquad \caX_{33} = \caO(\Psi^3)\,.
\end{align*}

In the expansion of \eqref{EF G_ii expand}, we will also use the following $z$-dependent random variables:
\begin{align}
\caX_{42} &\deq ( z + \gamma^2 m - \tau )^2 \frac{1}{N} \sum_{a} G_{ia} G_{ai}\,, \\
\caX_{43} &\deq ( z + \gamma^2 m - \tau ) \frac{1}{N^2} \sum_{a, b} G_{ia} G_{ab} G_{bi}\,, \\
\caX_{44} &\deq \frac{1}{N^3} \sum_{a, b, c} G_{ia} G_{ab} G_{bc} G_{ci}\,, \\
\caX_{44}' &\deq \frac{1}{N^3} \sum_{a, b, c} G_{ia} G_{ai} G_{bc} G_{cb} \,.
\end{align}
We remark that
$$
\caX_{42}\,, \caX_{43}\,, \caX_{44}\,, \caX_{44}' = \caO(\Psi^4)\,.
$$

Let
\begin{align}
\frX^{(a)} \deq N \int_{E_1}^{E_2} \im m^{(a)}(\widehat{L}_++x -2 + \ii \eta) \,\dd x = N \int_{E_1 -2}^{E_2 -2} \im m^{(a)}( \widehat{L}_++y+ \ii \eta) \,\dd y \,.
\end{align}
Since
$$
m(\widehat{L}_++x -2 + \ii \eta) - m^{(a)}( \widehat{L}_++x  -2 + \ii \eta) = \caO(\Psi^2)\,,
$$
we have
$$
\frX - \frX^{(a)} = \caO(\Psi)\,,
$$
hence by a Taylor expansion
\begin{align} \label{frX taylor}
F'(\frX) = F'(\frX^{(a)}) + F''(\frX^{(a)}) (\frX - \frX^{(a)}) + \frac{F'''(\frX^{(a)})}{2} (\frX - \frX^{(a)})^2 + \caO(\Psi^3)\,.
\end{align}
In particular,
$$
F'(\frX) - F'(\frX^{(a)}) = \caO(\Psi)\,.
$$

We consider first the case $a \neq i$. (Later, we will add the term $F'(\frX) G_{ii}^2$ for the case $a=i$.) The general idea of the expansion is the same as in Section \ref{sec:fluctuation}, and the ultimate goal of the expansion is to decouple the index $a$ appearing as a lower or upper index of the resolvent entries from all other indices. In a first step, using the resolvent formula~\eqref{onesided}, we find
\begin{align*}
G_{ia} G_{ai} = G_{aa}^2 \sum_{s, t}^{(a)} G_{is}^{(a)} h_{sa} h_{at} G_{ti}^{(a)} = \frac{1}{(\lambda \gamma v_a - z - \sum_{p, q}^{(a)} h_{ap} G_{pq}^{(a)} h_{qa})^2} \sum_{s, t}^{(a)} G_{is}^{(a)} h_{sa} h_{at} G_{ti}^{(a)}\,.
\end{align*}
Applying the large deviation estimates of~Lemma~\ref{lemma large deviation estimates} to the term $\sum_{p, q}^{(a)} h_{ap} G_{pq}^{(a)} h_{qa}$, we find 
$$
z + \sum_{p, q}^{(a)} h_{ap} G_{pq}^{(a)} h_{qa} - \tau = \caO(\Psi)\,.
$$
Thus, by expanding around $(\lambda \gamma v_a - \tau)$, we get
\begin{align}
G_{ia} G_{ai} &= \frac{1}{(\lambda \gamma v_a - \tau)^2} \sum_{s, t}^{(a)} G_{is}^{(a)} h_{sa} h_{at} G_{ti}^{(a)} \nonumber\\
& \qquad + \frac{2}{(\lambda \gamma v_a - \tau)^3} \left( z + \sum_{p, q}^{(a)} h_{ap} G_{pq}^{(a)} h_{qa} - \tau \right) \sum_{s, t}^{(a)} G_{is}^{(a)} h_{sa} h_{at} G_{ti}^{(a)}\nonumber \\
& \qquad + \frac{3}{(\lambda \gamma v_a - \tau)^4} \left( z + \sum_{p, q}^{(a)} h_{ap} G_{pq}^{(a)} h_{qa} - \tau \right)^2 \sum_{s, t}^{(a)} G_{is}^{(a)} h_{sa} h_{at} G_{ti}^{(a)} + \caO(\Psi^5)\,.
 \end{align}
Using \eqref{frX taylor}, we also find
\begin{align} \label{eq:1+2+3+4+5+6}
F'(\frX) G_{ia} G_{ai}& = \frac{F'(\frX^{(a)})}{(\lambda \gamma v_a - \tau)^2} \sum_{s, t}^{(a)} G_{is}^{(a)} h_{sa} h_{at} G_{ti}^{(a)} \nonumber \\
& \qquad + \frac{F''(\frX^{(a)})}{(\lambda \gamma v_a - \tau)^2} (\frX - \frX^{(a)}) \sum_{s, t}^{(a)} G_{is}^{(a)} h_{sa} h_{at} G_{ti}^{(a)} + \frac{F'''(\frX^{(a)})}{2(\lambda \gamma v_a - \tau)^2} (\frX - \frX^{(a)})^2 \sum_{s, t}^{(a)} G_{is}^{(a)} h_{sa} h_{at} G_{ti}^{(a)} \nonumber \\
& \qquad + \frac{2 F'(\frX^{(a)})}{(\lambda \gamma v_a - \tau)^3} \left( z + \sum_{p, q}^{(a)} h_{ap} G_{pq}^{(a)} h_{qa} - \tau \right) \sum_{s, t}^{(a)} G_{is}^{(a)} h_{sa} h_{at} G_{ti}^{(a)}\nonumber \\
& \qquad + \frac{2 F''(\frX^{(a)})}{(\lambda \gamma v_a - \tau)^3} (\frX - \frX^{(a)}) \left( z + \sum_{p, q}^{(a)} h_{ap} G_{pq}^{(a)} h_{qa} - \tau \right) \sum_{s, t}^{(a)} G_{is}^{(a)} h_{sa} h_{at} G_{ti}^{(a)} \nonumber \\
& \qquad + \frac{3 F'(\frX^{(a)})}{(\lambda \gamma v_a - \tau)^4} \left( z + \sum_{p, q}^{(a)} h_{ap} G_{pq}^{(a)} h_{qa} - \tau \right)^2 \sum_{s, t}^{(a)} G_{is}^{(a)} h_{sa} h_{at} G_{ti}^{(a)} + \caO(\Psi^5)\,.
\end{align}

We notice that, after taking the partial expectation $\E_a$, i.e., performing \textit{Step 2} of Section~\ref{sec:fluctuation}, the right side of~\eqref{eq:1+2+3+4+5+6} does no more contain $a$ as a lower index. In the next step, \textit{Step 3}, we remove the upper index $a$ in the resolvent entries. After completing one cycle of \textit{Steps 1-3}, we find that the index $a$ in the leading order term is decoupled, hence the expansion for the leading term is finished. We will repeat the same procedure until we obtain an expansion where all non-negligible terms are fully expanded. In the rest of this section, we expand each term in \eqref{eq:1+2+3+4+5+6} by following the procedure in Section \ref{sec:fluctuation}.

\subsection{Expansion of the first term in \eqref{eq:1+2+3+4+5+6}}

We begin by taking the partial expectation $\E_a$ of the first term in~\eqref{eq:1+2+3+4+5+6}. This corresponds to \textit{Step 2} in Section~\ref{sec:fluctuation}. From the relation
$$
G_{is}^{(a)} G_{si}^{(a)} = G_{is} G_{si} - \frac{G_{ia} G_{as}}{G_{aa}} G_{si} - G_{is}^{(a)} \frac{G_{sa} G_{ai}}{G_{aa}}\,,
$$
we have that
\begin{align} \label{eq:11+12+13+14+15}
&\E_a \left[ \frac{F'(\frX^{(a)})}{(\lambda \gamma v_a - \tau)^2} \sum_{s, t}^{(a)} G_{is}^{(a)} h_{sa} h_{at} G_{ti}^{(a)} \right] = \frac{\gamma^2}{N} \frac{F'(\frX^{(a)})}{(\lambda \gamma v_a - \tau)^2} \sum_s^{(a)} G_{is}^{(a)} G_{si}^{(a)} \nonumber\\
&\qquad\qquad= \frac{\gamma^2}{N} \frac{F'(\frX)}{(\lambda \gamma v_a - \tau)^2} \sum_s G_{is} G_{si} - \frac{\gamma^2}{N} \frac{F''(\frX^{(a)})}{(\lambda \gamma v_a - \tau)^2} (\frX - \frX^{(a)}) \sum_s^{(a)} G_{is} G_{si}\nonumber \\
&\qquad\qquad\qquad - \frac{\gamma^2}{2N} \frac{F'''(\frX^{(a)})}{(\lambda \gamma v_a - \tau)^2} (\frX - \frX^{(a)})^2 \sum_s^{(a)} G_{is} G_{si} \nonumber\\
&\qquad\qquad\qquad - \frac{\gamma^2}{N} \frac{F'(\frX^{(a)})}{(\lambda \gamma v_a - \tau)^2} \sum_s^{(a)} \frac{G_{ia} G_{as}}{G_{aa}} G_{si} - \frac{\gamma^2}{N} \frac{F'(\frX^{(a)})}{(\lambda \gamma v_a - \tau)^2} \sum_s^{(a)}  G_{is}^{(a)} \frac{G_{sa} G_{ai}}{G_{aa}} + \caO(\Psi^5)\,.
\end{align}
By the definition of $\caX_{22}$ in~\eqref{le definitionof x22}, we have for the first term on the right side of \eqref{eq:11+12+13+14+15} that
\begin{align} \label{eq:11}
\E \left[ \frac{\gamma^2}{N} \frac{F'(\frX)}{(\lambda \gamma v_a - \tau)^2} \sum_s G_{is} G_{si} \right] = \frac{\gamma^2}{(\lambda \gamma v_a - \tau)^2} \E [ F'(\frX) \caX_{22}]\,,
\end{align}
and we stop expanding it since $v_a$ is already decoupled from the random part. All the other terms in \eqref{eq:11+12+13+14+15} need to be expanded further. For that purpose we repeat the same procedure again. We remark that we may take the partial expectation $\E_a$ as many times as we want, since for any random variable $X$,
$$
\E X = \E\, \E_a X = \E\, \E_a \E_a X\,,
$$
and, although not written explicitly in~\eqref{eq:11+12+13+14+15}, the object we expand is $\E [F'(\frX) G_{ia}G_{ai}]$, which has the full expectation.

\subsubsection{Expansion of the second term in \eqref{eq:11+12+13+14+15}}

We now expand the second term in \eqref{eq:11+12+13+14+15}. We begin with
\begin{align} \label{eq:121+122+123}
\frac{\gamma^2}{N} \frac{F''(\frX^{(a)})}{(\lambda \gamma v_a - \tau)^2} &(\frX - \frX^{(a)}) \sum_s^{(a)} G_{is} G_{si} \nonumber \\
&= \frac{\gamma^2}{N} \frac{F''(\frX^{(a)})}{(\lambda \gamma v_a - \tau)^2} (\frX - \frX^{(a)}) \sum_s^{(a)} \left( G_{is}^{(a)} G_{si}^{(a)} + \frac{G_{ia} G_{as}}{G_{aa}} G_{si} + G_{is}^{(a)} \frac{G_{sa} G_{ai}}{G_{aa}} \right) \nonumber\\
&= \frac{\gamma^2}{N} \frac{F''(\frX^{(a)})}{(\lambda \gamma v_a - \tau)^2} (\frX - \frX^{(a)}) \sum_s^{(a)} G_{is}^{(a)} G_{si}^{(a)} + \frac{\gamma^2}{N} \frac{F''(\frX^{(a)})}{(\lambda \gamma v_a - \tau)^3} (\frX - \frX^{(a)}) \sum_{s, p, q}^{(a)} G_{ip}^{(a)} h_{pa} h_{aq} G_{qs}^{(a)} G_{si} \nonumber \\
&\qquad + \frac{\gamma^2}{N} \frac{F''(\frX^{(a)})}{(\lambda \gamma v_a - \tau)^3} (\frX - \frX^{(a)}) \sum_{s, p, q}^{(a)} G_{is}^{(a)} G_{sp}^{(a)} h_{pa} h_{aq} G_{qi}^{(a)} + \caO(\Psi^5)\,. 
\end{align}
By the definition of $\frX$ and $\frX^{(a)}$, we have
\begin{align*}
\frX - \frX^{(a)} = \im \int_{E_1 -2}^{E_2 -2} \sum_j \frac{G_{ja} G_{aj} (\widehat{L}_++\wt y +\ii \eta)}{G_{aa}(\widehat{L}_++\wt y +\ii \eta)}\, \dd \wt y \,.
\end{align*}
For simplicity, abbreviate $\wt G \equiv \wt G(\widehat{L}_++y+ + \ii \eta)$, where $\wt y \in [E_1 -2, E_2 -2]$. We then have
\begin{align*}
\frX - \frX^{(a)} = \im \int_{E_1 -2}^{E_2 -2} \sum_j \frac{\wt G_{ja} \wt G_{aj}}{\wt G_{aa}}\, \dd \wt y \,,
\end{align*}
and $\frX - \frX^{(a)} = \caO(\Psi)$. Following the procedure above, we consider
\begin{align}
\frac{\wt G_{ja} \wt G_{aj}}{\wt G_{aa}} &= \wt G_{aa} \sum_{p, q}^{(a)} \wt G^{(a)}_{jp} h_{pa} h_{aq} \wt G^{(a)}_{qj} \nonumber \\
&= \frac{1}{\lambda \gamma v_a - \tau} \sum_{p, q}^{(a)} \wt G^{(a)}_{jp} h_{pa} h_{aq} \wt G^{(a)}_{qj}\nonumber \\
&\qquad + \frac{1}{(\lambda \gamma v_a - \tau)^2} \left( z + \sum_{r, s}^{(a)} h_{ar} \wt G_{rs}^{(a)} h_{sa} - \tau \right) \sum_{p, q}^{(a)} \wt G^{(a)}_{jp} h_{pa} h_{aq} \wt G^{(a)}_{qj} + \caO(\Psi^4)\,. 
\end{align}
Taking the partial expectation $\E_a$, we find that
\begin{align}
\E_a \frac{\wt G_{ja} \wt G_{aj}}{\wt G_{aa}} &= \frac{\gamma^2}{\lambda \gamma v_a - \tau} \frac{1}{N} \sum_p^{(a)} \wt G^{(a)}_{jp} \wt G^{(a)}_{pj} + \frac{\gamma^2}{(\lambda \gamma v_a - \tau)^2} \left( z + \gamma^2 \wt m^{(a)} - \tau \right) \frac{1}{N} \sum_p^{(a)} \wt G^{(a)}_{jp} \wt G^{(a)}_{pj} \nonumber \\
& \qquad + \frac{2 \gamma^4}{(\lambda \gamma v_a - \tau)^2} \frac{1}{N^2} \sum_{p, q}^{(a)} \wt G^{(a)}_{jp} \wt G^{(a)}_{pq} \wt G^{(a)}_{qj} + \caO(\Psi^4)\,,
\end{align}
where we let
\begin{align*}
\wt m \deq \frac{1}{N} \sum_m \wt G_{mm}\,, \qquad \wt m^{(a)} \deq \frac{1}{N} \sum_m^{(a)} \wt G^{(a)}_{mm}\,.
\end{align*}

We define
\begin{align}
\frX_{22} & \deq \im \int_{E_1 -2}^{E_2 -2} \frac{1}{N} \sum_{j, k} \wt G_{jk} \wt G_{kj}\, \dd \wt y \,, \\
\frX_{32} & \deq \im \int_{E_1 -2}^{E_2 -2} (z + \gamma^2 \wt m - \tau) \frac{1}{N} \sum_{j, k}  \wt G_{jk} \wt G_{kj}\, \dd \wt y \,, \\
\frX_{33} & \deq \im \int_{E_1 -2}^{E_2 -2} \frac{1}{N^2} \sum_{j, k, m} \wt G_{jk} \wt G_{km} \wt G_{mj} \,\dd \wt y \,,
\end{align}
and similarly,
\begin{align}
\frX_{22}^{(a)} & \deq \im \int_{E_1 -2}^{E_2 -2} \frac{1}{N} \sum_{j, k}^{(a)} \wt G^{(a)}_{jk} \wt G^{(a)}_{kj}\, \dd \wt y \,, \\
\frX_{32}^{(a)} & \deq \im \int_{E_1 -2}^{E_2 -2} (z + \gamma^2 \wt m^{(a)} - \tau) \frac{1}{N} \sum_{j, k}^{(a)}  \wt G^{(a)}_{jk} \wt G^{(a)}_{kj}\, \dd \wt y \,, \\
\frX_{33}^{(a)} & \deq \im \int_{E_1 -2}^{E_2 -2} \frac{1}{N^2} \sum_{j, k, m}^{(a)} \wt G^{(a)}_{jk} \wt G^{(a)}_{km} \wt G^{(a)}_{mj}\, \dd \wt y \,.
\end{align}
We notice that
\begin{align*}
\frX_{22}, \frX_{22}^{(a)} = \caO(\Psi)\,, \qquad \qquad\frX_{32}, \frX_{33}, \frX_{32}^{(a)}, \frX_{33}^{(a)} = \caO(\Psi^2)\,.
\end{align*}
Thus, we have
\begin{align} \label{frX-frX^a}
\E_a [\frX - \frX^{(a)}] = \frac{\gamma^2}{\lambda \gamma v_a - \tau} \frX_{22}^{(a)} + \frac{\gamma^2}{(\lambda \gamma v_a - \tau)^2} \frX_{32}^{(a)} + \frac{2 \gamma^4}{(\lambda \gamma v_a - \tau)^2} \frX_{33}^{(a)} + \caO(\Psi^3)\,,
\end{align}
which also yields
\begin{align} \label{eq:1211+1212+1213}
&\E_a \left[ \frac{\gamma^2}{N} \frac{F''(\frX^{(a)})}{(\lambda \gamma v_a - \tau)^2} (\frX - \frX^{(a)}) \sum_s^{(a)} G_{is}^{(a)} G_{si}^{(a)} \right] \nonumber \\
&\qquad= \frac{\gamma^4}{N} \frac{F''(\frX^{(a)})}{(\lambda \gamma v_a - \tau)^3} \frX_{22}^{(a)} \sum_s^{(a)} G_{is}^{(a)} G_{si}^{(a)} + \frac{\gamma^4}{N} \frac{F''(\frX^{(a)})}{(\lambda \gamma v_a - \tau)^4} \frX_{32}^{(a)} \sum_s^{(a)} G_{is}^{(a)} G_{si}^{(a)}\nonumber \\
&\qquad\qquad + \frac{2 \gamma^6}{N} \frac{F''(\frX^{(a)})}{(\lambda \gamma v_a - \tau)^4} \frX_{33}^{(a)} \sum_s^{(a)} G_{is}^{(a)} G_{si}^{(a)} + \caO(\Psi^5)\,.
\end{align}

We repeat the same procedure again for the first term in \eqref{eq:1211+1212+1213}. We expand
\begin{align} \label{eq:12111+12112+12113+12114+12115}
&\frac{\gamma^4}{N} \frac{F''(\frX^{(a)})}{(\lambda \gamma v_a - \tau)^3} \frX_{22}^{(a)} \sum_s^{(a)} G_{is}^{(a)} G_{si}^{(a)} \nonumber \\
&\qquad= \frac{\gamma^4}{(\lambda \gamma v_a - \tau)^3} F''(\frX) \frX_{22} \caX_{22} - \frac{\gamma^4}{N} \frac{F''(\frX^{(a)})}{(\lambda \gamma v_a - \tau)^3} \frX_{22}^{(a)} \sum_s^{(a)} G_{is}^{(a)} \frac{G_{sa} G_{ai}}{G_{aa}} - \frac{\gamma^4}{N} \frac{F''(\frX^{(a)})}{(\lambda \gamma v_a - \tau)^3} \frX_{22}^{(a)} \sum_s^{(a)} \frac{G_{ia} G_{as}}{G_{aa}} G_{si}\nonumber\\
&\qquad\qquad  - \frac{\gamma^4}{N} \frac{F''(\frX^{(a)})}{(\lambda \gamma v_a - \tau)^3} ( \frX_{22} - \frX_{22}^{(a)} ) \sum_s^{(a)} G_{is} G_{si} - \frac{\gamma^4}{N} \frac{F'''(\frX^{(a)})}{(\lambda \gamma v_a - \tau)^3} ( \frX - \frX^{(a)} ) \frX_{22} \sum_s^{(a)} G_{is} G_{si} + \caO(\Psi^5)\,, 
\end{align}
where we used that
\begin{align*}
F''(\frX) = F''(\frX^{(a)}) + F'''(\frX^{(a)}) (\frX - \frX^{(a)}) + \caO(\Psi^2)
\end{align*}
and that
\begin{align*}
\frX_{22} - \frX_{22}^{(a)} = \im \int_{E_1 -2}^{E_2 -2} \frac{1}{N} \sum_{j, k} \left( \frac{\wt G_{ja} \wt G_{ak}}{\wt G_{aa}} \wt G_{kj} + \wt G^{(a)}_{jk} \frac{\wt G_{ka} \wt G_{aj}}{\wt G_{aa}} \right) \dd \wt y = \caO(\Psi^2)\,.
\end{align*}

The first term in \eqref{eq:12111+12112+12113+12114+12115} is fully expanded, i.e., $v_a$ is decoupled from the random part. All the other terms in \eqref{eq:12111+12112+12113+12114+12115} are of $\caO(\Psi^4)$. To a $\caO(\Psi^4)$-term we can freely add or remove the upper index $(a)$ at the expense of an error term of $\caO(\Psi^5)$, which is negligible in the calculation. For example, the second term in \eqref{eq:12111+12112+12113+12114+12115} satisfies
\begin{align} \label{eq:12112}
&\E_a \left[ \frac{\gamma^4}{N} \frac{F''(\frX^{(a)})}{(\lambda \gamma v_a - \tau)^3} \frX_{22}^{(a)} \sum_s^{(a)} G_{is}^{(a)} \frac{G_{sa} G_{ai}}{G_{aa}} \right] = \E_a \left[ \frac{\gamma^4}{N} \frac{F''(\frX^{(a)})}{(\lambda \gamma v_a - \tau)^3} \frX_{22}^{(a)} G_{aa} \sum_{s, p, q}^{(a)} G_{is}^{(a)} G_{sp}^{(a)} h_{pa} h_{aq} G_{qi}^{(a)} \right] \nonumber \\
&\qquad= \frac{\gamma^6}{N^2} \frac{F''(\frX^{(a)})}{(\lambda \gamma v_a - \tau)^4} \frX_{22}^{(a)} \sum_{s, p}^{(a)} G_{is}^{(a)} G_{sp}^{(a)} G_{pi}^{(a)} + \caO(\Psi^5) = \frac{\gamma^6}{N^2} \frac{F''(\frX)}{(\lambda \gamma v_a - \tau)^4} \frX_{22} \sum_{s, p} G_{is} G_{sp} G_{pi} + \caO(\Psi^5) \nonumber \\
&\qquad= \frac{\gamma^6}{(\lambda \gamma v_a - \tau)^4} F''(\frX) \frX_{22} \caX_{33} + \caO(\Psi^5)\,.
\end{align}
Similarly, for the third term in \eqref{eq:12111+12112+12113+12114+12115} we have
\begin{align} \label{eq:12113}
&\E_a \left[ \frac{\gamma^4}{N} \frac{F''(\frX^{(a)})}{(\lambda \gamma v_a - \tau)^3} \frX_{22}^{(a)} \sum_s^{(a)} \frac{G_{ia} G_{as}}{G_{aa}} G_{si} \right] = \frac{\gamma^6}{(\lambda \gamma v_a - \tau)^4} F''(\frX) \frX_{22} \caX_{33} + \caO(\Psi^5)\,.
\end{align}
In order to control the fourth term in \eqref{eq:12111+12112+12113+12114+12115}, we first consider $\E_a (\frX_{22} - \frX_{22}^{(a)})$ as in \eqref{frX-frX^a} and get
\begin{align} \label{frX_22-frX_22^a}
\E_a (\frX_{22} - \frX_{22}^{(a)}) = \frac{2 \gamma^2}{\lambda \gamma v_a - \tau} \frX_{33} + \caO(\Psi^3)\,.
\end{align}
We then obtain that
\begin{align} \label{eq:12114}
&\E_a \left[ \frac{\gamma^4}{N} \frac{F''(\frX^{(a)})}{(\lambda \gamma v_a - \tau)^3} ( \frX_{22} - \frX_{22}^{(a)} ) \sum_s^{(a)} G_{is} G_{si} \right] = \frac{2 \gamma^6}{(\lambda \gamma v_a - \tau)^4} F''(\frX) \frX_{33} \caX_{22} + \caO(\Psi^5)\,.
\end{align}
Finally, the last term in \eqref{eq:12111+12112+12113+12114+12115} becomes
\begin{align} \label{eq:12115}
&\E_a \left[ \frac{\gamma^4}{N} \frac{F'''(\frX^{(a)})}{(\lambda \gamma v_a - \tau)^3} ( \frX - \frX^{(a)} ) \frX_{22} \sum_s^{(a)} G_{is} G_{si} \right] = \frac{\gamma^6}{(\lambda \gamma v_a - \tau)^4} F'''(\frX) (\frX_{22})^2 \caX_{22} + \caO(\Psi^5)\,.
\end{align}

We thus have shown from \eqref{eq:12111+12112+12113+12114+12115}, \eqref{eq:12112}, \eqref{eq:12113}, \eqref{eq:12114} and \eqref{eq:12115} that
\begin{align} \label{eq:1211}
&\E \left[ \frac{\gamma^4}{N} \frac{F''(\frX^{(a)})}{(\lambda \gamma v_a - \tau)^3} \frX_{22}^{(a)} \sum_s^{(a)} G_{is}^{(a)} G_{si}^{(a)} \right] \nonumber \\
&\qquad= \frac{\gamma^4}{(\lambda \gamma v_a - \tau)^3} \E[ F''(\frX) \frX_{22} \caX_{22}] - \frac{2 \gamma^6}{(\lambda \gamma v_a - \tau)^4} \E[ F''(\frX) \frX_{22} \caX_{33} ] \nonumber\\
&\qquad\qquad - \frac{2 \gamma^6}{(\lambda \gamma v_a - \tau)^4} \E[ F''(\frX) \frX_{33} \caX_{22} ] - \frac{\gamma^6}{(\lambda \gamma v_a - \tau)^4} \E[ F'''(\frX) (\frX_{22})^2 \caX_{22} ] + \caO(\Psi^5)\,.
\end{align}

We next return to the second and the third terms in \eqref{eq:1211+1212+1213}. Since these are of $\caO(\Psi^4)$, we observe that
\begin{align} \label{eq:1212}
\E \left[ \frac{\gamma^4}{N} \frac{F''(\frX^{(a)})}{(\lambda \gamma v_a - \tau)^4} \frX_{32}^{(a)} \sum_s^{(a)} G_{is}^{(a)} G_{si}^{(a)} \right] = \frac{\gamma^4}{(\lambda \gamma v_a - \tau)^4} \E[ F''(\frX) \frX_{32} \caX_{22} ] + \caO(\Psi^5)
\end{align}
and that
\begin{align} \label{eq:1213}
\E \left[ \frac{2 \gamma^6}{N} \frac{F''(\frX^{(a)})}{(\lambda \gamma v_a - \tau)^4} \frX_{33}^{(a)} \sum_s^{(a)} G_{is}^{(a)} G_{si}^{(a)} \right] = \frac{2 \gamma^6}{(\lambda \gamma v_a - \tau)^4} \E[ F''(\frX) \frX_{33} \caX_{22} ] + \caO(\Psi^5)\,.
\end{align}
We then obtain from \eqref{eq:1211+1212+1213}, \eqref{eq:1211}, \eqref{eq:1212} and \eqref{eq:1213} that
\begin{align} \label{eq:121}
& \E \left[ \frac{\gamma^2}{N} \frac{F''(\frX^{(a)})}{(\lambda \gamma v_a - \tau)^2} (\frX - \frX^{(a)}) \sum_s^{(a)} G_{is}^{(a)} G_{si}^{(a)} \right] \nonumber \\
&\qquad= \frac{\gamma^4}{(\lambda \gamma v_a - \tau)^3} \E[ F''(\frX) \frX_{22} \caX_{22}] + \frac{\gamma^4}{(\lambda \gamma v_a - \tau)^4} \E[ F''(\frX) \frX_{32} \caX_{22} ] \nonumber\\
&\qquad\qquad - \frac{2 \gamma^6}{(\lambda \gamma v_a - \tau)^4} \E[ F''(\frX) \frX_{22} \caX_{33} ] - \frac{\gamma^6}{(\lambda \gamma v_a - \tau)^4} \E[ F'''(\frX) (\frX_{22})^2 \caX_{22} ] + \caO(\Psi^5)\,. 
\end{align}
This finishes the expansion for the first term in \eqref{eq:121+122+123}.

We now return to the second term in~\eqref{eq:121+122+123}. We note that
\begin{align}
&\frac{\gamma^2}{N} \frac{F''(\frX^{(a)})}{(\lambda \gamma v_a - \tau)^3} (\frX - \frX^{(a)}) \sum_{s, p, q}^{(a)} G_{ip}^{(a)} h_{pa} h_{aq} G_{qs}^{(a)} G_{si} \nonumber \\
&\qquad= \frac{\gamma^2}{N} \frac{F''(\frX^{(a)})}{(\lambda \gamma v_a - \tau)^3} \left( \im \int_{E_1 -2}^{E_2 -2} \sum_j \frac{\wt G_{ja} \wt G_{aj}}{\wt G_{aa}} \dd \wt y \right) \sum_{s, p, q}^{(a)} G_{ip}^{(a)} h_{pa} h_{aq} G_{qs}^{(a)} G_{si}\nonumber \\
&\qquad= \frac{\gamma^2}{N} \frac{F''(\frX^{(a)})}{(\lambda \gamma v_a - \tau)^4} \left( \im \int_{E_1 -2}^{E_2 -2} \sum_j \sum_{r, s}^{(a)} \wt G^{(a)}_{jr} h_{ra} h_{as} \wt G^{(a)}_{sj} \dd \wt y \right) \sum_{s, p, q}^{(a)} G_{ip}^{(a)} h_{pa} h_{aq} G_{qs}^{(a)} G_{si}^{(a)} + \caO(\Psi^5)\,.
\end{align}
Taking the partial expectation $\E_a$, we find
\begin{align}
&\E_a \left[ \frac{\gamma^2}{N} \frac{F''(\frX^{(a)})}{(\lambda \gamma v_a - \tau)^3} (\frX - \frX^{(a)}) \sum_{s, p, q}^{(a)} G_{ip}^{(a)} h_{pa} h_{aq} G_{qs}^{(a)} G_{si} \right] \nonumber \\
&= \frac{\gamma^6}{(\lambda \gamma v_a - \tau)^4} F''(\frX^{(a)}) \frX_{22} \caX_{33}  + \frac{2 \gamma^6}{N^3} \frac{F''(\frX^{(a)})}{(\lambda \gamma v_a - \tau)^4} \left( \im \int_{E_1 -2}^{E_2 -2} \sum_j \sum_{r, t}^{(a)} \wt G^{(a)}_{jr} \wt G^{(a)}_{tj} \dd \wt y \right) \sum_s^{(a)} G_{ir}^{(a)} G_{ts}^{(a)} G_{si}^{(a)} + \caO(\Psi^5)\,. \nonumber
\end{align}
Thus, we obtain 
\begin{align} \label{eq:122}
&\E \left[ \frac{\gamma^2}{N} \frac{F''(\frX^{(a)})}{(\lambda \gamma v_a - \tau)^3} (\frX - \frX^{(a)}) \sum_{s, p, q}^{(a)} G_{ip}^{(a)} h_{pa} h_{aq} G_{qs}^{(a)} G_{si} \right] \nonumber \\
&\qquad= \frac{\gamma^6}{(\lambda \gamma v_a - \tau)^4} \E[ F''(\frX) \frX_{22} \caX_{33} ]\nonumber \\
&\qquad\qquad + \frac{2 \gamma^6}{(\lambda \gamma v_a - \tau)^4} \E \left[ F''(\frX) \frac{1}{N^3} \sum_{j, r, t, s} \left( \im \int_{E_1 -2}^{E_2 -2} \wt G_{jr} \wt G_{tj} \,\dd \wt y \right) G_{ir} G_{ts} G_{si} \right] + \caO(\Psi^5)\,. 
\end{align}
The third term in \eqref{eq:121+122+123} can be dealt with in a similar manner, and we get
\begin{align} \label{eq:123}
&\E \left[ \frac{\gamma^2}{N} \frac{F''(\frX^{(a)})}{(\lambda \gamma v_a - \tau)^3} (\frX - \frX^{(a)}) \sum_{s, p, q}^{(a)} G_{is}^{(a)} G_{sp}^{(a)} h_{pa} h_{aq} G_{qi}^{(a)} \right] \\
&= \frac{\gamma^6}{(\lambda \gamma v_a - \tau)^4} \E[ F''(\frX) \frX_{22} \caX_{33} ] + \frac{2 \gamma^6}{(\lambda \gamma v_a - \tau)^4} \E \left[ F''(\frX) \frac{1}{N^3} \sum_{j, r, t, s} \left( \im \int_{E_1 -2}^{E_2 -2} \wt G_{jr} \wt G_{tj} \dd \wt y \right) G_{ir} G_{ts} G_{si} \right] + \caO(\Psi^5)\,. \nonumber
\end{align}
So far, we have shown from \eqref{eq:121+122+123}, \eqref{eq:121}, \eqref{eq:122} and \eqref{eq:123} that
\begin{align} \label{eq:12}
&\E \left[ \frac{\gamma^2}{N} \frac{F''(\frX^{(a)})}{(\lambda \gamma v_a - \tau)^2} (\frX - \frX^{(a)}) \sum_s^{(a)} G_{is} G_{si} \right] \nonumber\\
&\qquad= \frac{\gamma^4}{(\lambda \gamma v_a - \tau)^3} \E[ F''(\frX) \frX_{22} \caX_{22}] + \frac{\gamma^4}{(\lambda \gamma v_a - \tau)^4} \E[ F''(\frX) \frX_{32} \caX_{22} ] - \frac{\gamma^6}{(\lambda \gamma v_a - \tau)^4} \E[ F'''(\frX) (\frX_{22})^2 \caX_{22} ] \nonumber \\
&\qquad\qquad + \frac{4 \gamma^6}{(\lambda \gamma v_a - \tau)^4} \E \left[ F''(\frX) \frac{1}{N^3} \sum_{j, r, t, s} \left( \im \int_{E_1 -2}^{E_2 -2} \wt G_{jr} \wt G_{tj} \dd \wt y \right) G_{ir} G_{ts} G_{si} \right] + \caO(\Psi^5)\,. 
\end{align}
This completes the expansion of the second term in \eqref{eq:11+12+13+14+15}.

\subsubsection{Expansion of the third term in \eqref{eq:11+12+13+14+15}}

Recall that
\begin{align}
\frX - \frX^{(a)} = \im \int_{E_1 -2}^{E_2 -2} \sum_j \frac{\wt G_{ja} \wt G_{aj}}{\wt G_{aa}} \dd \wt y
\end{align}
and that $\frX - \frX^{(a)} = \caO(\Psi)$. Let $\widehat G \equiv G(\widehat{L}_++\widehat y + \ii \eta)$, for $\widehat y \in [E_1 -2, E_2 -2]$. Then, we may write
\begin{align} \label{frX-frX^a^2}
\left(\frX - \frX^{(a)}\right)^2 = \left( \im \int_{E_1 -2}^{E_2 -2} \sum_j \frac{\wt G_{ja} \wt G_{aj}}{\wt G_{aa}} \dd \wt y \right) \left( \im \int_{E_1 -2}^{E_2 -2} \sum_k \frac{\widehat G_{ka} \widehat G_{ak}}{\widehat G_{aa}} \dd \widehat y \right) \,.
\end{align}
Applying the resolvent formula~\eqref{onesided} to both of the integrands, we find
\begin{align*}
&\left(\frX - \frX^{(a)}\right)^2 \\
&\qquad= \left( \im \int_{E_1 -2}^{E_2 -2} \wt G_{aa} \sum_j \sum_{p, q}^{(a)} \wt G^{(a)}_{jp} h_{pa} h_{aq} \wt G^{(a)}_{qj} \dd \wt y \right)\left( \im \int_{E_1 -2}^{E_2 -2} \widehat G_{aa} \sum_k \sum_{r, t} \widehat G^{(a)}_{kr} h_{ra} h_{at} \widehat G^{(a)}_{tk} \dd \widehat y \right) \nonumber \\
&\qquad= \frac{1}{(\lambda \gamma v_a - \tau)^2} \sum_{j, k} \sum_{p, q, r, t}^{(a)} \left( \im \int_{E_1 -2}^{E_2 -2} \wt G^{(a)}_{jp} h_{pa} h_{aq} \wt G^{(a)}_{qj} \dd \wt y \right) \left( \im \int_{E_1 -2}^{E_2 -2} \widehat G^{(a)}_{kr} h_{ra} h_{at} \widehat G^{(a)}_{tk} \dd \widehat y \right) \,. \nonumber
\end{align*}
Thus, by taking the partial expectation $\E_a$, we get
\begin{align*}
\E_a \left(\frX - \frX^{(a)}\right)^2 \nonumber &= \frac{\gamma^4}{(\lambda \gamma v_a - \tau)^2} \frac{1}{N^2} \sum_{j, k} \sum_{p, r}^{(a)} \left( \im \int_{E_1 -2}^{E_2 -2} \wt G^{(a)}_{jp} \wt G^{(a)}_{pj} \dd \wt y \right) \left( \im \int_{E_1 -2}^{E_2 -2} \widehat G^{(a)}_{kr} \widehat G^{(a)}_{rk} \dd \widehat y \right) \\
&\qquad + \frac{2 \gamma^4}{(\lambda \gamma v_a - \tau)^2} \frac{1}{N^2} \sum_{j, k} \sum_{r, t}^{(a)} \left( \im \int_{E_1 -2}^{E_2 -2} \wt G^{(a)}_{jr} \wt G^{(a)}_{tj} \dd \wt y \right) \left( \im \int_{E_1 -2}^{E_2 -2} \widehat G^{(a)}_{kr} \widehat G^{(a)}_{tk} \dd \widehat y \right) + \caO(\Psi^3) \,, \nonumber
\end{align*}
which implies
\begin{align*}
\E_a \left(\frX - \frX^{(a)}\right)^2 &= \frac{\gamma^4}{(\lambda \gamma v_a - \tau)^2} (\frX_{22})^2 \\
& \qquad+ \frac{2 \gamma^4}{(\lambda \gamma v_a - \tau)^2} \frac{1}{N^2} \sum_{j, k, r, t} \left( \im \int_{E_1 -2}^{E_2 -2} \wt G_{jr} \wt G_{tj} \dd \wt y \right) \left( \im \int_{E_1 -2}^{E_2 -2} \widehat G_{kr} \widehat G_{tk} \dd \widehat y \right) + \caO(\Psi^3) \,. \nonumber
\end{align*}
We remark that the terms with the fourth or higher moments of the entries of $H(t)$ have negligible contributions, hence we omit the details on those terms in the calculation. Thus returning to the third term in \eqref{eq:11+12+13+14+15}, we obtain
\begin{align} \label{eq:13}
&\E \left[ \frac{\gamma^2}{2N} \frac{F'''(\frX^{(a)})}{(\lambda \gamma v_a - \tau)^2} (\frX - \frX^{(a)})^2 \sum_s^{(a)} G_{is} G_{si} \right] \nonumber \\
&\qquad= \frac{\gamma^6}{2(\lambda \gamma v_a - \tau)^4} \E \left[ F'''(\frX) (\frX_{22})^2 \caX_{22} \right]\nonumber \\
&\qquad\qquad + \frac{\gamma^6}{(\lambda \gamma v_a - \tau)^4} \E \left[ F'''(\frX) \frac{1}{N^2} \sum_{j, k, r, t} \left( \im \int_{E_1 -2}^{E_2 -2} \wt G_{jr} \wt G_{tj} \dd \wt y \right) \left( \im \int_{E_1 -2}^{E_2 -2} \widehat G_{kr} \widehat G_{tk} \dd \widehat y \right) \caX_{22} \right] + \caO(\Psi^5)\,.
\end{align}

\subsubsection{Expansion of the fourth term and the fifth term in \eqref{eq:11+12+13+14+15}}

We first begin with the expansion
\begin{align} \label{eq:141+142}
\frac{\gamma^2}{N} \frac{F'(\frX^{(a)})}{(\lambda \gamma v_a - \tau)^2} \sum_s^{(a)} \frac{G_{ia} G_{as}}{G_{aa}} G_{si} = \frac{\gamma^2}{N} \frac{F'(\frX^{(a)})}{(\lambda \gamma v_a - \tau)^2} \sum_s^{(a)} \frac{G_{ia} G_{as}}{G_{aa}} G_{si}^{(a)} + \frac{\gamma^2}{N} \frac{F'(\frX^{(a)})}{(\lambda \gamma v_a - \tau)^2} \sum_s^{(a)} \frac{G_{ia} G_{as}}{G_{aa}} \frac{G_{sa} G_{ai}}{G_{aa}} \,. 
\end{align}
Expanding the first term in \eqref{eq:141+142}, we get
\begin{align} \label{eq:1411+1412}
&\frac{\gamma^2}{N} \frac{F'(\frX^{(a)})}{(\lambda \gamma v_a - \tau)^2} \sum_s^{(a)} \frac{G_{ia} G_{as}}{G_{aa}} G_{si}^{(a)} = \frac{\gamma^2}{N} \frac{F'(\frX^{(a)})}{(\lambda \gamma v_a - \tau)^3} \sum_{s, r, t}^{(a)} G_{ir}^{(a)} h_{ra} h_{at} G_{ts}^{(a)} G_{si}^{(a)} \nonumber\\
&\qquad \qquad \qquad+ \frac{\gamma^2}{N} \frac{F'(\frX^{(a)})}{(\lambda \gamma v_a - \tau)^4} \left( z + \sum_{p, q}^{(a)} h_{ap} G_{pq}^{(a)} h_{qa} - \tau \right) \sum_{s, r, t}^{(a)} G_{ir}^{(a)} h_{ra} h_{at} G_{ts}^{(a)} G_{si}^{(a)} + \caO(\Psi^5) \,. 
\end{align}
Taking the partial expectation $\E_a$, we obtain
\begin{align} \label{eq:14111+14112+14113+14114+14115}
&\E_a \left[ \frac{\gamma^2}{N} \frac{F'(\frX^{(a)})}{(\lambda \gamma v_a - \tau)^3} \sum_{s, r, t}^{(a)} G_{ir}^{(a)} h_{ra} h_{at} G_{ts}^{(a)} G_{si}^{(a)} \right] = \frac{\gamma^4}{N^2} \frac{F'(\frX^{(a)})}{(\lambda \gamma v_a - \tau)^3} \sum_{s, r}^{(a)} G_{ir}^{(a)} G_{rs}^{(a)} G_{si}^{(a)}\nonumber \\
&= \frac{\gamma^4}{N^2} \frac{F'(\frX)}{(\lambda \gamma v_a - \tau)^3} \sum_{s, r}^{(a)} G_{ir} G_{rs} G_{si} - \frac{\gamma^4}{N^2} \frac{F'(\frX^{(a)})}{(\lambda \gamma v_a - \tau)^3} \sum_{s, r}^{(a)} G_{ir}^{(a)} G_{rs}^{(a)} \frac{G_{sa} G_{ai}}{G_{aa}} - \frac{\gamma^4}{N^2} \frac{F'(\frX^{(a)})}{(\lambda \gamma v_a - \tau)^3} \sum_{s, r}^{(a)} G_{ir}^{(a)} \frac{G_{ra} G_{as}}{G_{aa}} G_{si} \nonumber \\
&\qquad\qquad - \frac{\gamma^4}{N^2} \frac{F'(\frX^{(a)})}{(\lambda \gamma v_a - \tau)^3} \sum_{s, r}^{(a)} \frac{G_{ia} G_{ar}}{G_{aa}} G_{rs} G_{si} - \frac{\gamma^4}{N^2} \frac{F''(\frX^{(a)})}{(\lambda \gamma v_a - \tau)^3} (\frX - \frX^{(a)}) \sum_{s, r}^{(a)} G_{ir} G_{rs} G_{si} + \caO(\Psi^5)\,. 
\end{align}
Taking the full expectation, the first term in \eqref{eq:14111+14112+14113+14114+14115} becomes
\begin{align} \label{eq:14111}
\E \left[ \frac{\gamma^4}{N^2} \frac{F'(\frX)}{(\lambda \gamma v_a - \tau)^3} \sum_{s, r}^{(a)} G_{ir} G_{rs} G_{si} \right] = \frac{\gamma^4}{(\lambda \gamma v_a - \tau)^3} \E[ F'(\frX) \caX_{33} ] + \caO(\Psi^5)\,.
\end{align}
For the second term in \eqref{eq:14111+14112+14113+14114+14115}, we have that
\begin{align} \label{eq:14112}
\E \left[ \frac{\gamma^4}{N^2} \frac{F'(\frX^{(a)})}{(\lambda \gamma v_a - \tau)^3} \sum_{s, r}^{(a)} G_{ir}^{(a)} G_{rs}^{(a)} \frac{G_{sa} G_{ai}}{G_{aa}} \right] &= \E \left[ \frac{\gamma^4}{N^2} \frac{F'(\frX^{(a)})}{(\lambda \gamma v_a - \tau)^4} \sum_{s, r}^{(a)} G_{ir}^{(a)} G_{rs}^{(a)} \, \E_a \sum_{p, q}^{(a)} G_{sp}^{(a)} h_{pa} h_{aq} G_{qi}^{(a)} \right] + \caO(\Psi^5) \nonumber \\
&= \frac{\gamma^6}{(\lambda \gamma v_a - \tau)^4} \E[ F'(\frX) \caX_{44}] + \caO(\Psi^5) \,.
\end{align}
Similarly, we can also obtain	
\begin{align} \label{eq:14113}
\E \left[ \frac{\gamma^4}{N^2} \frac{F'(\frX^{(a)})}{(\lambda \gamma v_a - \tau)^3} \sum_{s, r}^{(a)} G_{ir}^{(a)} \frac{G_{ra} G_{as}}{G_{aa}} G_{si} \right] = \frac{\gamma^6}{(\lambda \gamma v_a - \tau)^4} \E[ F'(\frX) \caX_{44}] + \caO(\Psi^5)
\end{align}
and
\begin{align} \label{eq:14114}
\E \left[ \frac{\gamma^4}{N^2} \frac{F'(\frX^{(a)})}{(\lambda \gamma v_a - \tau)^3} \sum_{s, r}^{(a)} \frac{G_{ia} G_{ar}}{G_{aa}} G_{rs} G_{si} \right] = \frac{\gamma^6}{(\lambda \gamma v_a - \tau)^4} \E[ F'(\frX) \caX_{44}] + \caO(\Psi^5) \,.
\end{align}
For the last term in \eqref{eq:14111+14112+14113+14114+14115}, we use \eqref{frX-frX^a} to obtain that
\begin{align} \label{eq:14115}
&\E \left[ \frac{\gamma^4}{N^2} \frac{F''(\frX^{(a)})}{(\lambda \gamma v_a - \tau)^3} (\frX - \frX^{(a)}) \sum_{s, r}^{(a)} G_{ir} G_{rs} G_{si} \right] = \frac{\gamma^6}{(\lambda \gamma v_a - \tau)^4} \E [ F''(\frX) \frX_{22} \caX_{33}] + \caO(\Psi^5)\,.
\end{align}

Thus, from \eqref{eq:14111+14112+14113+14114+14115}, \eqref{eq:14111}, \eqref{eq:14112}, \eqref{eq:14113}, \eqref{eq:14114} and \eqref{eq:14115}, we find that
\begin{align} \label{eq:1411}
&\E \left[ \frac{\gamma^2}{N} \frac{F'(\frX^{(a)})}{(\lambda \gamma v_a - \tau)^3} \sum_{s, r, t}^{(a)} G_{ir}^{(a)} h_{ra} h_{at} G_{ts}^{(a)} G_{si}^{(a)} \right]\nonumber \\
&\qquad= \frac{\gamma^4}{(\lambda \gamma v_a - \tau)^3} \E[ F'(\frX) \caX_{33} ] - \frac{3 \gamma^6}{(\lambda \gamma v_a - \tau)^4} \E[ F'(\frX) \caX_{44}] - \frac{\gamma^6}{(\lambda \gamma v_a - \tau)^4} \E [ F''(\frX) \frX_{22} \caX_{33}] + \caO(\Psi^5) \,. 
\end{align}

The second term in \eqref{eq:1411+1412} satisfies 
\begin{align}
&\E_a \left[ \frac{\gamma^2}{N} \frac{F'(\frX^{(a)})}{(\lambda \gamma v_a - \tau)^4} \left( z + \sum_{p, q}^{(a)} h_{ap} G_{pq}^{(a)} h_{qa} - \tau \right) \sum_{s, r, t}^{(a)} G_{ir}^{(a)} h_{ra} h_{at} G_{ts}^{(a)} G_{si}^{(a)} \right] \nonumber \\
&\qquad= \frac{\gamma^4}{N^2} \frac{F'(\frX^{(a)})}{(\lambda \gamma v_a - \tau)^4} \left( z + \gamma^2 m^{(a)} - \tau \right) \sum_{s, t}^{(a)} G_{it}^{(a)} G_{ts}^{(a)} G_{si}^{(a)}  + \frac{2 \gamma^6}{N^3} \frac{F'(\frX^{(a)})}{(\lambda \gamma v_a - \tau)^4} \sum_{s, r, t}^{(a)} G_{ir}^{(a)} G_{rt}^{(a)} G_{ts}^{(a)} G_{si}^{(a)} + \caO(\Psi^5) \,. \nonumber
\end{align}
Thus, we get
\begin{align} \label{eq:1412}
&\E \left[ \frac{\gamma^2}{N} \frac{F'(\frX^{(a)})}{(\lambda \gamma v_a - \tau)^4} \left( z + \sum_{p, q}^{(a)} h_{ap} G_{pq}^{(a)} h_{qa} - \tau \right) \sum_{s, r, t}^{(a)} G_{ir}^{(a)} h_{ra} h_{at} G_{ts}^{(a)} G_{si}^{(a)} \right] \nonumber \\
&\qquad\qquad= \frac{\gamma^4}{(\lambda \gamma v_a - \tau)^4} \E [F'(\frX) \caX_{43}] + \frac{2 \gamma^6}{(\lambda \gamma v_a - \tau)^4} \E [F'(\frX) \caX_{44}] + \caO(\Psi^5) \,.
\end{align}
We then have from \eqref{eq:1411+1412}, \eqref{eq:1411} and \eqref{eq:1412} that
\begin{align} \label{eq:141}
&\E \left[ \frac{\gamma^2}{N} \frac{F'(\frX^{(a)})}{(\lambda \gamma v_a - \tau)^2} \sum_s^{(a)} \frac{G_{ia} G_{as}}{G_{aa}} G_{si}^{(a)} \right] \nonumber \\
& \qquad=\frac{\gamma^4}{(\lambda \gamma v_a - \tau)^3} \E[ F'(\frX) \caX_{33} ] + \frac{\gamma^4}{(\lambda \gamma v_a - \tau)^4} \E [F'(\frX) \caX_{43}] \nonumber\\
&\qquad\qquad - \frac{\gamma^6}{(\lambda \gamma v_a - \tau)^4} \E[ F'(\frX) \caX_{44}] - \frac{\gamma^6}{(\lambda \gamma v_a - \tau)^4} \E [ F''(\frX) \frX_{22} \caX_{33}] + \caO(\Psi^5) \,. 
\end{align}

Expanding the second term in \eqref{eq:141+142}, we have
\begin{align} \label{eq:142}
&\E \left[ \frac{\gamma^2}{N} \frac{F'(\frX^{(a)})}{(\lambda \gamma v_a - \tau)^2} \sum_s^{(a)} \frac{G_{ia} G_{as}}{G_{aa}} \frac{G_{sa} G_{ai}}{G_{aa}} \right] \nonumber \\
&\qquad= \E \left[ \frac{\gamma^2}{N} \frac{F'(\frX^{(a)})}{(\lambda \gamma v_a - \tau)^4} \sum_s^{(a)} \sum_{p, q}^{(a)} G_{ip}^{(a)} h_{pa} h_{aq} G_{qs}^{(a)} \sum_{r, t}^{(a)} G_{ir}^{(a)} h_{ra} h_{at} G_{ts}^{(a)} \right] + \caO(\Psi^5) \nonumber\\
&\qquad= \E \left[ \frac{2 \gamma^6}{N^3} \frac{F'(\frX^{(a)})}{(\lambda \gamma v_a - \tau)^4} \sum_{s, p, r}^{(a)} G_{ip}^{(a)} G_{ps}^{(a)} G_{ir}^{(a)} G_{rs}^{(a)} \right] + \E \left[ \frac{\gamma^6}{N^3} \frac{F'(\frX^{(a)})}{(\lambda \gamma v_a - \tau)^4} \sum_{s, r, t}^{(a)} G_{ir}^{(a)} G_{ts}^{(a)} G_{ir}^{(a)} G_{ts}^{(a)} \right] + \caO(\Psi^5) \nonumber \\
&\qquad= \frac{2 \gamma^6}{(\lambda \gamma v_a - \tau)^4} \E [F'(\frX) \caX_{44}] + \frac{\gamma^6}{(\lambda \gamma v_a - \tau)^4} \E [F'(\frX) \caX_{44}'] + \caO(\Psi^5)\,. 
\end{align}
Thus, from \eqref{eq:141+142}, \eqref{eq:141} and \eqref{eq:142}, we get
\begin{align} \label{eq:14}
&\E \left[ \frac{\gamma^2}{N} \frac{F'(\frX^{(a)})}{(\lambda \gamma v_a - \tau)^2} \sum_s^{(a)} \frac{G_{ia} G_{as}}{G_{aa}} G_{si} \right] \nonumber \\
&\qquad= \frac{\gamma^4}{(\lambda \gamma v_a - \tau)^3} \E[ F'(\frX) \caX_{33} ] + \frac{\gamma^4}{(\lambda \gamma v_a - \tau)^4} \E [F'(\frX) \caX_{43}] + \frac{\gamma^6}{(\lambda \gamma v_a - \tau)^4} \E[ F'(\frX) \caX_{44}] \nonumber\\
&\qquad\qquad + \frac{\gamma^6}{(\lambda \gamma v_a - \tau)^4} \E [F'(\frX) \caX_{44}'] - \frac{\gamma^6}{(\lambda \gamma v_a - \tau)^4} \E [ F''(\frX) \frX_{22} \caX_{33}] + \caO(\Psi^5) \,. 
\end{align}

The fifth term in \eqref{eq:11+12+13+14+15} can be expanded analogously to \eqref{eq:141}, and we find that
\begin{align} \label{eq:15}
&\E \left[ \frac{\gamma^2}{N} \frac{F'(\frX^{(a)})}{(\lambda \gamma v_a - \tau)^2} \sum_s^{(a)}  G_{is}^{(a)} \frac{G_{sa} G_{ai}}{G_{aa}} \right] \nonumber \\
&\qquad= \frac{\gamma^4}{(\lambda \gamma v_a - \tau)^3} \E[ F'(\frX) \caX_{33} ] + \frac{\gamma^4}{(\lambda \gamma v_a - \tau)^4} \E [F'(\frX) \caX_{43}] \nonumber\\
&\qquad\qquad - \frac{\gamma^6}{(\lambda \gamma v_a - \tau)^4} \E[ F'(\frX) \caX_{44}] - \frac{\gamma^6}{(\lambda \gamma v_a - \tau)^4} \E [ F''(\frX) \frX_{22} \caX_{33}] + \caO(\Psi^5) \,.
\end{align}
This finishes the expansion of the fourth term and the fifth term in \eqref{eq:11+12+13+14+15}.

\subsubsection{Expansion of the first term in \eqref{eq:1+2+3+4+5+6}}

Combining \eqref{eq:11+12+13+14+15}, \eqref{eq:11}, \eqref{eq:12}, \eqref{eq:13}, \eqref{eq:14} and \eqref{eq:15}, we conclude that
\begin{align} \label{eq:1}
&\E \left[ \frac{F'(\frX^{(a)})}{(\lambda \gamma v_a - \tau)^2} \sum_{s, t}^{(a)} G_{is}^{(a)} h_{sa} h_{at} G_{ti}^{(a)} \right] \nonumber \\
&\qquad= \frac{\gamma^2}{(\lambda \gamma v_a - \tau)^2} \E [ F'(\frX) \caX_{22}] - \frac{2 \gamma^4}{(\lambda \gamma v_a - \tau)^3} \E[ F'(\frX) \caX_{33} ] - \frac{\gamma^4}{(\lambda \gamma v_a - \tau)^3} \E[ F''(\frX) \frX_{22} \caX_{22}]\nonumber \\
&\qquad\qquad - \frac{\gamma^4}{(\lambda \gamma v_a - \tau)^4} \E[ F''(\frX) \frX_{32} \caX_{22} ] + \frac{\gamma^6}{2(\lambda \gamma v_a - \tau)^4} \E[ F'''(\frX) (\frX_{22})^2 \caX_{22} ]  \nonumber \\
&\qquad\qquad  - \frac{4 \gamma^6}{(\lambda \gamma v_a - \tau)^4} \E \left[ F''(\frX) \frac{1}{N^3} \sum_{j, r, t, s} \left( \im \int_{E_1 -2}^{E_2 -2} \wt G_{jr} \wt G_{tj}\, \dd \wt y \right) G_{ir} G_{ts} G_{si} \right] \nonumber \\
&\qquad\qquad - \frac{\gamma^6}{(\lambda \gamma v_a - \tau)^4} \E \left[ F'''(\frX) \frac{1}{N^2} \sum_{j, k, r, t} \left( \im \int_{E_1 -2}^{E_2 -2} \wt G_{jr} \wt G_{tj}\, \dd \wt y \right) \left( \im \int_{E_1 -2}^{E_2 -2} \widehat G_{kr} \widehat G_{tk} \,\dd \widehat y \right) \caX_{22} \right] \nonumber \\
&\qquad\qquad  - \frac{2 \gamma^4}{(\lambda \gamma v_a - \tau)^4} \E [F'(\frX) \caX_{43}] - \frac{\gamma^6}{(\lambda \gamma v_a - \tau)^4} \E [F'(\frX) \caX_{44}'] + \frac{\gamma^6}{(\lambda \gamma v_a - \tau)^4} \E [ F''(\frX) \frX_{22} \caX_{33}] + \caO(\Psi^5) \,.
\end{align}
This finishes the expansion of the first term in \eqref{eq:1+2+3+4+5+6}.

\subsection{Expansion of the second term in \eqref{eq:1+2+3+4+5+6}}

Following the expansion procedure for the first term in \eqref{eq:1+2+3+4+5+6}, we now expand the second term. By the definitions of $\frX$ and $\frX^{(a)}$, we have
\begin{align} \label{eq:21+22}
&\frac{F''(\frX^{(a)})}{(\lambda \gamma v_a - \tau)^2} (\frX - \frX^{(a)}) \sum_{s, t}^{(a)} G_{is}^{(a)} h_{sa} h_{at} G_{ti}^{(a)} \nonumber \\
&\qquad= \frac{F''(\frX^{(a)})}{(\lambda \gamma v_a - \tau)^2} \left( \im \int_{E_1 -2}^{E_2 -2} \sum_j \frac{\wt G_{ja} \wt G_{aj}}{\wt G_{aa}} \,\dd \wt y \right) \sum_{s, t}^{(a)} G_{is}^{(a)} h_{sa} h_{at} G_{ti}^{(a)}\nonumber \\
&\qquad= \frac{F''(\frX^{(a)})}{(\lambda \gamma v_a - \tau)^3} \left( \im \int_{E_1 -2}^{E_2 -2} \sum_j^{(a)} \sum_{p, q}^{(a)} \wt G^{(a)}_{jp} h_{pa} h_{aq} \wt G^{(a)}_{qj} \,\dd \wt y \right) \sum_{s, t}^{(a)} G_{is}^{(a)} h_{sa} h_{at} G_{ti}^{(a)} \nonumber \\
&\qquad\qquad + \frac{F''(\frX^{(a)})}{(\lambda \gamma v_a - \tau)^4} \left( \im \int_{E_1 -2}^{E_2 -2} \Bigg( z + \sum_{k, m}^{(a)} h_{ak} \wt G^{(a)}_{km} h_{ma} - \tau \Bigg) \sum_j^{(a)} \sum_{p, q}^{(a)} \wt G^{(a)}_{jp} h_{pa} h_{aq} \wt G^{(a)}_{qj} \,\dd \wt y \right) \sum_{s, t}^{(a)} G_{is}^{(a)} h_{sa} h_{at} G_{ti}^{(a)}\nonumber\\
&\qquad\qquad + \caO(\Psi^5) \,. 
\end{align}

\subsubsection{Expansion of the first term in \eqref{eq:21+22}}

After taking the partial expectation $\E_a$, the first term on the right side of~\eqref{eq:21+22} becomes
\begin{align} \label{eq:211+212}
&\E_a \left[ \frac{F''(\frX^{(a)})}{(\lambda \gamma v_a - \tau)^3} \left( \im \int_{E_1 -2}^{E_2 -2} \sum_j^{(a)} \sum_{p, q}^{(a)} \wt G^{(a)}_{jp} h_{pa} h_{aq} \wt G^{(a)}_{qj} \,\dd \wt y \right) \sum_{s, t}^{(a)} G_{is}^{(a)} h_{sa} h_{at} G_{ti}^{(a)} \right] \nonumber\\
&\qquad= \frac{\gamma^4}{N^2} \frac{F''(\frX^{(a)})}{(\lambda \gamma v_a - \tau)^3} \frX_{22}^{(a)} \sum_s^{(a)} G_{is}^{(a)} G_{si}^{(a)}  + \frac{2 \gamma^4}{N^2} \frac{F''(\frX^{(a)})}{(\lambda \gamma v_a - \tau)^3} \sum_{j, s, t}^{(a)} \left( \im \int_{E_1 -2}^{E_2 -2} \wt G^{(a)}_{js} \wt G^{(a)}_{tj} \,\dd \wt y \right) G_{is}^{(a)} G_{ti}^{(a)}\,.
\end{align}

The expectation of the first term in \eqref{eq:211+212} is calculated in \eqref{eq:1211}, which is
\begin{align} \label{eq:211}
&\E \left[ \frac{\gamma^4}{N} \frac{F''(\frX^{(a)})}{(\lambda \gamma v_a - \tau)^3} \frX_{22}^{(a)} \sum_s^{(a)} G_{is}^{(a)} G_{si}^{(a)} \right] \nonumber \\
&\qquad= \frac{\gamma^4}{(\lambda \gamma v_a - \tau)^3} \E[ F''(\frX) \frX_{22} \caX_{22}] - \frac{2 \gamma^6}{(\lambda \gamma v_a - \tau)^4} \E [ F''(\frX) \frX_{22} \caX_{33} ] \nonumber\\
&\qquad\qquad - \frac{2 \gamma^6}{(\lambda \gamma v_a - \tau)^4} \E [ F''(\frX) \frX_{33} \caX_{22} ] - \frac{\gamma^6}{(\lambda \gamma v_a - \tau)^4} \E [ F'''(\frX) (\frX_{22})^2 \caX_{22} ] + \caO(\Psi^5)\,.
\end{align}

We now consider the second term in \eqref{eq:211+212}. We notice that
\begin{align} \label{eq:2121+2122+2123+2124+2125+2126}
&\frac{2 \gamma^4}{N^2} \frac{F''(\frX^{(a)})}{(\lambda \gamma v_a - \tau)^3} \sum_{j, s, t}^{(a)} \left( \im \int_{E_1 -2}^{E_2 -2} \wt G^{(a)}_{js} \wt G^{(a)}_{tj} \,\dd \wt y \right) G_{is}^{(a)} G_{ti}^{(a)} \nonumber \\
&\qquad= \frac{2 \gamma^4}{N^2} \frac{F''(\frX)}{(\lambda \gamma v_a - \tau)^3} \sum_{j, s, t} \left( \im \int_{E_1 -2}^{E_2 -2} \wt G_{js} \wt G_{tj} \,\dd \wt y \right) G_{is} G_{ti}\nonumber \\
&\qquad\qquad - \frac{2 \gamma^4}{N^2} \frac{F''(\frX^{(a)})}{(\lambda \gamma v_a - \tau)^3} \sum_{j, s, t}^{(a)} \left( \im \int_{E_1 -2}^{E_2 -2} \wt G^{(a)}_{js} \wt G^{(a)}_{tj} \,\dd \wt y \right) \left( G_{is}^{(a)} \frac{G_{ta} G_{ai}}{G_{aa}} +  \frac{G_{ia} G_{si}}{G_{aa}} G_{ti} \right) \nonumber \\
&\qquad\qquad - \frac{2 \gamma^4}{N^2} \frac{F''(\frX^{(a)})}{(\lambda \gamma v_a - \tau)^3} \sum_{j, s, t}^{(a)} \left[ \left( \im \int_{E_1 -2}^{E_2 -2} \wt G^{(a)}_{js} \frac{\wt G_{ta} \wt G_{aj}}{\wt G_{aa}} \,\dd \wt y \right) +  \left( \im \int_{E_1 -2}^{E_2 -2} \frac{\wt G_{ja} \wt G_{as}}{\wt G_{aa}} \wt G_{tj}  \,\dd \wt y \right) \right] G_{is} G_{ti} \nonumber \\
&\qquad\qquad - \frac{2 \gamma^4}{N^2} \frac{F'''(\frX^{(a)})}{(\lambda \gamma v_a - \tau)^3} (\frX - \frX^{(a)}) \sum_{j, s, t}^{(a)} \left( \im \int_{E_1 -2}^{E_2 -2} \wt G^{(a)}_{js} \wt G^{(a)}_{tj} \,\dd \wt y \right) G_{is} G_{ti} + \caO(\Psi^5) \,. 
\end{align}
We proceed as in the expansion of \eqref{eq:211} to obtain
\begin{align} \label{eq:212}
&\E \left[ \frac{2 \gamma^4}{N^2} \frac{F''(\frX^{(a)})}{(\lambda \gamma v_a - \tau)^3} \sum_{j, s, t}^{(a)} \left( \im \int_{E_1 -2}^{E_2 -2} \wt G^{(a)}_{js} \wt G^{(a)}_{tj} \,\dd \wt y \right) G_{is}^{(a)} G_{ti}^{(a)} \right] \nonumber \\
&\qquad= \frac{2 \gamma^4}{(\lambda \gamma v_a - \tau)^3} \E \left[ F''(\frX) \frac{1}{N^2} \sum_{j, s, t} \left( \im \int_{E_1 -2}^{E_2 -2} \wt G_{js} \wt G_{tj} \,\dd \wt y \right) G_{is} G_{ti} \right] \nonumber \\
&\qquad\qquad - \frac{4 \gamma^6}{(\lambda \gamma v_a - \tau)^4} \E \left[ F''(\frX) \frac{1}{N^3} \sum_{j, s, t, r} \left( \im \int_{E_1 -2}^{E_2 -2} \wt G_{js} \wt G_{tj} \,\dd \wt y \right) G_{is} G_{tr} G_{ri} \right] \nonumber\\
&\qquad\qquad - \frac{4 \gamma^6}{(\lambda \gamma v_a - \tau)^4} \E \left[ F''(\frX) \frac{1}{N^3} \sum_{j, s, t, r} \left( \im \int_{E_1 -2}^{E_2 -2} \wt G_{jr} \wt G_{rs} \wt G_{tj} \,\dd \wt y \right) G_{is} G_{ti} \right] \nonumber \\
&\qquad\qquad - \frac{2 \gamma^6}{(\lambda \gamma v_a - \tau)^4} \E \left[ F'''(\frX) \frX_{22} \frac{1}{N^2} \sum_{j, s, t} \left( \im \int_{E_1 -2}^{E_2 -2} \wt G_{js} \wt G_{tj} \,\dd \wt y \right) G_{is} G_{ti} \right] + \caO(\Psi^5)\,. 
\end{align}

\subsubsection{Expansion of the second term in \eqref{eq:21+22}}

We first notice that the second term in \eqref{eq:21+22} is of $\caO(\Psi^4)$. Taking the partial expectation $\E_a$, we get
\begin{align}
&\E_a \left[ \left( \im \int_{E_1 -2}^{E_2 -2} \Bigg( z + \sum_{k, m}^{(a)} h_{ak} \wt G^{(a)}_{km} h_{ma} - \tau \Bigg) \sum_j^{(a)} \sum_{p, q}^{(a)} \wt G^{(a)}_{jp} h_{pa} h_{aq} \wt G^{(a)}_{qj} \,\dd \wt y \right) \sum_{s, t}^{(a)} G_{is}^{(a)} h_{sa} h_{at} G_{ti}^{(a)} \right] \nonumber \\
&\qquad= \frac{\gamma^4}{N^2} \sum_{j, p, s}^{(a)} \left( \im \int_{E_1 -2}^{E_2 -2} (z + \gamma^2 \wt m^{(a)} - \tau) \wt G^{(a)}_{jp} \wt G^{(a)}_{pj} \,\dd \wt y \right) G_{is}^{(a)} G_{si}^{(a)}\nonumber \\
& \qquad\qquad + \frac{2 \gamma^6}{N^3} \sum_{j, p, q, s}^{(a)} \left( \im \int_{E_1 -2}^{E_2 -2} \wt G^{(a)}_{jp} \wt G^{(a)}_{pq} \wt G^{(a)}_{qj} \,\dd \wt y \right) G_{is}^{(a)} G_{si}^{(a)} + \frac{2 \gamma^6}{N^3} \sum_{j, p, s, t}^{(a)} \left( \im \int_{E_1 -2}^{E_2 -2} \wt G^{(a)}_{st} \wt G^{(a)}_{jp} \wt G^{(a)}_{pj} \,\dd \wt y \right) G_{is}^{(a)} G_{ti}^{(a)} \nonumber \\
& \qquad\qquad + \frac{2 \gamma^4}{N^2} \sum_{j, s, t}^{(a)} \left( \im \int_{E_1 -2}^{E_2 -2} (z + \gamma^2 \wt m^{(a)} - \tau) \wt G^{(a)}_{js} \wt G^{(a)}_{tj} \,\dd \wt y \right) G_{is}^{(a)} G_{ti}^{(a)} \nonumber \\
& \qquad \qquad+ \frac{8 \gamma^6}{N^3} \sum_{j, p, s, t}^{(a)} \left( \im \int_{E_1 -2}^{E_2 -2} \wt G^{(a)}_{ps} \wt G^{(a)}_{jp} \wt G^{(a)}_{tj} \,\dd \wt y \right) G_{is}^{(a)} G_{ti}^{(a)} + \caO(\Psi^5)\,. 
\end{align}
Thus, we have
\begin{align} \label{eq:22}
&\E \Bigg[ \frac{F''(\frX^{(a)})}{(\lambda \gamma v_a - \tau)^4} \left( \im \int_{E_1 -2}^{E_2 -2} \Bigg( z + \sum_{k, m}^{(a)} h_{ak} \wt G^{(a)}_{km} h_{ma} - \tau \Bigg) \sum_j^{(a)} \sum_{p, q}^{(a)} \wt G^{(a)}_{jp} h_{pa} h_{aq} \wt G^{(a)}_{qj} \,\dd \wt y \right) \sum_{s, t}^{(a)} G_{is}^{(a)} h_{sa} h_{at} G_{ti}^{(a)} \Bigg] \nonumber \\
&\qquad= \frac{\gamma^4}{(\lambda \gamma v_a - \tau)^4} \E [F''(\frX) \frX_{32} \caX_{22}] + \frac{2 \gamma^6}{(\lambda \gamma v_a - \tau)^4} \E [F''(\frX) \frX_{33} \caX_{22}] \nonumber \\
&\qquad\qquad + \frac{2 \gamma^4}{(\lambda \gamma v_a - \tau)^4} \E \left[ F''(\frX) \frac{1}{N^2} \sum_{j, s, t} \left( \im \int_{E_1 -2}^{E_2 -2} (z + \gamma^2 \wt m - \tau) \wt G_{js} \wt G_{tj} \,\dd \wt y \right) G_{is} G_{ti} \right]\nonumber \\
& \qquad\qquad + \frac{2 \gamma^6}{(\lambda \gamma v_a - \tau)^4} \E \left[ F''(\frX) \frac{1}{N^3} \sum_{j, p, s, t} \left( \im \int_{E_1 -2}^{E_2 -2} \wt G_{st} \wt G_{jp} \wt G_{pj} \,\dd \wt y \right) G_{is} G_{ti} \right] \nonumber \\
& \qquad\qquad + \frac{8 \gamma^6}{(\lambda \gamma v_a - \tau)^4} \E \left[ F''(\frX) \frac{1}{N^3} \sum_{j, p, s, t} \left( \im \int_{E_1 -2}^{E_2 -2} \wt G_{ps} \wt G_{jp} \wt G_{tj} \,\dd \wt y \right) G_{is} G_{ti} \right] + \caO(\Psi^5)\,. 
\end{align}

\subsubsection{Expansion of the second term in \eqref{eq:1+2+3+4+5+6}}

We have from \eqref{eq:21+22}, \eqref{eq:211+212}, \eqref{eq:211}, \eqref{eq:212} and \eqref{eq:22} that
\begin{align} \label{eq:2}
&\E \left[ \frac{F''(\frX^{(a)})}{(\lambda \gamma v_a - \tau)^2} (\frX - \frX^{(a)}) \sum_{s, t}^{(a)} G_{is}^{(a)} h_{sa} h_{at} G_{ti}^{(a)} \right] \nonumber \\
&\qquad= \frac{\gamma^4}{(\lambda \gamma v_a - \tau)^3} \E[ F''(\frX) \frX_{22} \caX_{22}] - \frac{2 \gamma^6}{(\lambda \gamma v_a - \tau)^4} \E [ F''(\frX) \frX_{22} \caX_{33} ] \nonumber \\
&\qquad\qquad + \frac{\gamma^4}{(\lambda \gamma v_a - \tau)^4} \E [ F'''(\frX) \frX_{32} \caX_{22} ] - \frac{\gamma^6}{(\lambda \gamma v_a - \tau)^4} \E [ F'''(\frX) (\frX_{22})^2 \caX_{22} ] \nonumber \\
&\qquad\qquad + \frac{2 \gamma^4}{(\lambda \gamma v_a - \tau)^3} \E \left[ F''(\frX) \frac{1}{N^2} \sum_{j, s, t} \left( \im \int_{E_1 -2}^{E_2 -2} \wt G_{js} \wt G_{tj} \,\dd \wt y \right) G_{is} G_{ti} \right] \nonumber \\
&\qquad\qquad + \frac{2 \gamma^4}{(\lambda \gamma v_a - \tau)^4} \E \left[ F''(\frX) \frac{1}{N^2} \sum_{j, s, t} \left( \im \int_{E_1 -2}^{E_2 -2} (z + \gamma^2 \wt m - \tau) \wt G_{js} \wt G_{tj} \,\dd \wt y \right) G_{is} G_{ti} \right]\nonumber \\
&\qquad\qquad - \frac{4 \gamma^6}{(\lambda \gamma v_a - \tau)^4} \E \left[ F''(\frX) \frac{1}{N^3} \sum_{j, s, t, r} \left( \im \int_{E_1 -2}^{E_2 -2} \wt G_{js} \wt G_{tj} \,\dd \wt y \right) G_{is} G_{tr} G_{ri} \right] \nonumber \\
&\qquad\qquad + \frac{4 \gamma^6}{(\lambda \gamma v_a - \tau)^4} \E \left[ F''(\frX) \frac{1}{N^3} \sum_{j, s, t, r} \left( \im \int_{E_1 -2}^{E_2 -2} \wt G_{jr} \wt G_{rs} \wt G_{tj} \,\dd \wt y \right) G_{is} G_{ti} \right] \nonumber \\
&\qquad\qquad - \frac{2 \gamma^6}{(\lambda \gamma v_a - \tau)^4} \E \left[ F'''(\frX) \frX_{22} \frac{1}{N^2} \sum_{j, s, t} \left( \im \int_{E_1 -2}^{E_2 -2} \wt G_{js} \wt G_{tj} \,\dd \wt y \right) G_{is} G_{ti} \right] \nonumber \\
&\qquad \qquad + \frac{2 \gamma^6}{(\lambda \gamma v_a - \tau)^4} \E \left[ F''(\frX) \frac{1}{N^3} \sum_{j, p, s, t} \left( \im \int_{E_1 -2}^{E_2 -2} \wt G_{st} \wt G_{jp} \wt G_{pj} \,\dd \wt y \right) G_{is} G_{ti} \right] + \caO(\Psi^5)\,. 
\end{align}

\subsection{Expansion of the third term in \eqref{eq:1+2+3+4+5+6}}

We first notice that the third term in \eqref{eq:1+2+3+4+5+6} is of $\caO(\Psi^4)$. Using~\eqref{frX-frX^a^2} we get
\begin{align}
&\frac{F'''(\frX^{(a)})}{2(\lambda \gamma v_a - \tau)^2} (\frX - \frX^{(a)})^2 \sum_{s, t}^{(a)} G_{is}^{(a)} h_{sa} h_{at} G_{ti}^{(a)}\nonumber \\
&\qquad= \frac{F'''(\frX^{(a)})}{2(\lambda \gamma v_a - \tau)^2} \left( \im \int_{E_1 -2}^{E_2 -2} \sum_j \frac{\wt G_{ja} \wt G_{aj}}{\wt G_{aa}} \,\dd \wt y \right) \left( \im \int_{E_1 -2}^{E_2 -2} \sum_k \frac{\widehat G_{ka} \widehat G_{ak}}{\widehat G_{aa}} \,\dd \widehat y \right) \sum_{s, t}^{(a)} G_{is}^{(a)} h_{sa} h_{at} G_{ti}^{(a)} 
\end{align}
Taking the partial expectation $\E_a$, we find
\begin{align}
&\E_a \left[ \frac{F'''(\frX^{(a)})}{2(\lambda \gamma v_a - \tau)^2} (\frX - \frX^{(a)})^2 \sum_{s, t}^{(a)} G_{is}^{(a)} h_{sa} h_{at} G_{ti}^{(a)} \right] \nonumber\\
&\qquad= \frac{F'''(\frX^{(a)})}{2(\lambda \gamma v_a - \tau)^4} \frac{\gamma^6}{N^3} \sum_{j, k, p, q}^{(a)} \left( \im \int_{E_1 -2}^{E_2 -2} \wt G^{(a)}_{jp} \wt G^{(a)}_{pj} \dd \wt y \right) \left( \im \int_{E_1 -2}^{E_2 -2} \widehat G^{(a)}_{kq} \widehat G^{(a)}_{qk} \,\dd \widehat y \right) \sum_s^{(a)} G_{is}^{(a)} G_{si}^{(a)} \nonumber \\
&\qquad\qquad + \frac{F'''(\frX^{(a)})}{(\lambda \gamma v_a - \tau)^4} \frac{\gamma^6}{N^3} \sum_{j, k, p, q}^{(a)} \left( \im \int_{E_1 -2}^{E_2 -2} \wt G^{(a)}_{jp} \wt G^{(a)}_{qj} \,\dd \wt y \right) \left( \im \int_{E_1 -2}^{E_2 -2} \widehat G^{(a)}_{kp} \widehat G^{(a)}_{qk} \,\dd \widehat y \right) \sum_s^{(a)} G_{is}^{(a)} G_{si}^{(a)} \nonumber \\
&\qquad\qquad+ \frac{2 F'''(\frX^{(a)})}{(\lambda \gamma v_a - \tau)^4} \frac{\gamma^6}{N^3} \sum_{j, k, p, s, t}^{(a)} \left( \im \int_{E_1 -2}^{E_2 -2} \wt G^{(a)}_{jp} \wt G^{(a)}_{pj} \,\dd \wt y \right) \left( \im \int_{E_1 -2}^{E_2 -2} \widehat G^{(a)}_{ks} \widehat G^{(a)}_{tk} \,\dd \widehat y \right) G_{is}^{(a)} G_{ti}^{(a)} \nonumber \\
&\qquad\qquad+ \frac{4 F'''(\frX^{(a)})}{(\lambda \gamma v_a - \tau)^4} \frac{\gamma^6}{N^3} \sum_{j, k, p, s, t}^{(a)} \left( \im \int_{E_1 -2}^{E_2 -2} \wt G^{(a)}_{jp} \wt G^{(a)}_{sj} \,\dd \wt y \right) \left( \im \int_{E_1 -2}^{E_2 -2} \widehat G^{(a)}_{kp} \widehat G^{(a)}_{tk} \,\dd \widehat y \right) G_{is}^{(a)} G_{ti}^{(a)} + \caO(\Psi^5)\,.
\end{align}
Thus, after taking the full expectation, we obtain 
\begin{align} \label{eq:3}
&\E \left[ \frac{F'''(\frX^{(a)})}{2(\lambda \gamma v_a - \tau)^2} (\frX - \frX^{(a)})^2 \sum_{s, t}^{(a)} G_{is}^{(a)} h_{sa} h_{at} G_{ti}^{(a)} \right] \nonumber \\
&\qquad= \frac{\gamma^6}{2(\lambda \gamma v_a - \tau)^4} \E [ F'''(\frX) (\frX_{22})^2 \caX_{22}] \nonumber\\
&\qquad\qquad+ \frac{\gamma^6}{(\lambda \gamma v_a - \tau)^4} \E \left[ F'''(\frX) \frac{1}{N^2} \sum_{j, k, p, q} \left( \im \int_{E_1 -2}^{E_2 -2} \wt G_{jp} \wt G_{qj} \,\dd \wt y \right) \left( \im \int_{E_1 -2}^{E_2 -2} \widehat G_{kp} \widehat G_{qk} \,\dd \widehat y \right) \caX_{22} \right] \nonumber \\
&\qquad\qquad + \frac{2 \gamma^6}{(\lambda \gamma v_a - \tau)^4} \E \left[ F'''(\frX) \frX_{22} \frac{1}{N^2} \sum_{k, s, t} \left( \im \int_{E_1 -2}^{E_2 -2} \wt G_{kp} \wt G_{qk} \,\dd \wt y \right) G_{is} G_{ti} \right] \nonumber \\
&\qquad\qquad + \frac{4 \gamma^6}{(\lambda \gamma v_a - \tau)^4} \E \left[ F'''(\frX) \frac{1}{N^3} \sum_{j, k, p, s, t} \left( \im \int_{E_1 -2}^{E_2 -2} \wt G_{jp} \wt G_{sj} \,\dd \wt y \right) \left( \im \int_{E_1 -2}^{E_2 -2} \widehat G_{kp} \widehat G_{tk} \,\dd \widehat y \right) G_{is} G_{ti} \right] \nonumber\\
&\qquad\qquad+ \caO(\Psi^5)\,. 
\end{align}

\subsection{Expansion of the fourth term in \eqref{eq:1+2+3+4+5+6}}

We again begin by taking the partial expectation $\E_a$. We have
\begin{align} \label{eq:41+42}
&\E_a \left[ \frac{2 F'(\frX^{(a)})}{(\lambda \gamma v_a - \tau)^3} \left( z + \sum_{p, q}^{(a)} h_{ap} G_{pq}^{(a)} h_{qa} - \tau \right) \sum_{s, t}^{(a)} G_{is}^{(a)} h_{sa} h_{at} G_{ti}^{(a)} \right]\nonumber \\
&\qquad\qquad= \frac{\gamma^2}{N} \frac{2 F'(\frX^{(a)})}{(\lambda \gamma v_a - \tau)^3} ( z + \gamma^2 m^{(a)} - \tau ) \sum_s^{(a)} G_{is}^{(a)} G_{si}^{(a)} + \frac{\gamma^4}{N^2} \frac{4 F'(\frX^{(a)})}{(\lambda \gamma v_a - \tau)^3} \sum_{s, t}^{(a)} G_{is}^{(a)} G_{st}^{(a)} G_{ti}^{(a)} + \caO(\Psi^5)\,. 
\end{align}
Following the procedure in \eqref{eq:14111+14112+14113+14114+14115} we expand \eqref{eq:41+42}. We first consider
\begin{align} \label{eq:411+412+413+414+415}
&\frac{\gamma^2}{N} \frac{2 F'(\frX^{(a)})}{(\lambda \gamma v_a - \tau)^3} ( z + \gamma^2 m^{(a)} - \tau ) \sum_s^{(a)} G_{is}^{(a)} G_{si}^{(a)}\nonumber \\
&\qquad= \frac{\gamma^2}{N} \frac{2 F'(\frX)}{(\lambda \gamma v_a - \tau)^3} ( z + \gamma^2 m - \tau ) \sum_s^{(a)} G_{is} G_{si} - \frac{\gamma^2}{N} \frac{2 F'(\frX^{(a)})}{(\lambda \gamma v_a - \tau)^3} ( z + \gamma^2 m^{(a)} - \tau ) \sum_s^{(a)} G_{is}^{(a)} \frac{G_{sa} G_{ai}}{G_{aa}} \nonumber \\
&\qquad\qquad - \frac{\gamma^2}{N} \frac{2 F'(\frX^{(a)})}{(\lambda \gamma v_a - \tau)^3} ( z + \gamma^2 m^{(a)} - \tau ) \sum_s^{(a)} \frac{G_{ia} G_{as}}{G_{aa}} G_{si} - \frac{\gamma^4}{N} \frac{2 F'(\frX^{(a)})}{(\lambda \gamma v_a - \tau)^3} (m - m^{(a)}) \sum_s^{(a)} G_{is} G_{si} \nonumber \\
&\qquad\qquad- \frac{\gamma^2}{N} \frac{2 F''(\frX^{(a)})}{(\lambda \gamma v_a - \tau)^3} (\frX - \frX^{(a)}) ( z + \gamma^2 m - \tau ) \sum_s^{(a)} G_{is} G_{si} + \caO(\Psi^5)\,. 
\end{align}
We stop expanding the first term and observe that
\begin{align} \label{eq:411}
\E \left[ \frac{\gamma^2}{N} \frac{2 F'(\frX)}{(\lambda \gamma v_a - \tau)^3} ( z + \gamma^2 m - \tau ) \sum_s^{(a)} G_{is} G_{si} \right] = \frac{2 \gamma^2}{(\lambda \gamma v_a - \tau)^3} \E[ F'(\frX) \caX_{32} ] + \caO(\Psi^5)\,.
\end{align}
We notice that all other terms in \eqref{eq:411+412+413+414+415} are of $\caO(\Psi^4)$. As in the estimates \eqref{eq:14112} and \eqref{eq:14113}, we find for the second and the third term of \eqref{eq:411+412+413+414+415} that
\begin{align} \label{eq:412}
\E \left[ \frac{\gamma^2}{N} \frac{2 F'(\frX^{(a)})}{(\lambda \gamma v_a - \tau)^3} ( z + \gamma^2 m^{(a)} - \tau ) \sum_s^{(a)} G_{is}^{(a)} \frac{G_{sa} G_{ai}}{G_{aa}} \right] = \frac{2 \gamma^4}{(\lambda \gamma v_a - \tau)^4} \E[ F'(\frX) \caX_{43} ] + \caO(\Psi^5)\,,
\end{align}
respectively,
\begin{align} \label{eq:413}
\E \left[ \frac{\gamma^2}{N} \frac{2 F'(\frX^{(a)})}{(\lambda \gamma v_a - \tau)^3} ( z + \gamma^2 m^{(a)} - \tau ) \sum_s^{(a)} \frac{G_{ia} G_{as}}{G_{aa}} G_{si} \right] = \frac{2 \gamma^4}{(\lambda \gamma v_a - \tau)^4} \E[ F'(\frX) \caX_{43} ] + \caO(\Psi^5)\,.
\end{align}
Since
\begin{align*}
m - m^{(a)} = \frac{1}{N} \sum_r \frac{G_{ra} G_{ar}}{G_{aa}}\,,
\end{align*}
we obtain for the fourth term in \eqref{eq:411+412+413+414+415} that
\begin{align} \label{eq:414}
\E \left[ \frac{\gamma^4}{N} \frac{2 F'(\frX^{(a)})}{(\lambda \gamma v_a - \tau)^3} (m - m^{(a)}) \sum_s^{(a)} G_{is} G_{si} \right] = \frac{2 \gamma^6}{(\lambda \gamma v_a - \tau)^4} \E[ F'(\frX) \caX_{44}' ] + \caO(\Psi^5)\,.
\end{align}
Finally, similar to~\eqref{eq:14115}, we get for the last term in \eqref{eq:411+412+413+414+415} that
\begin{align} \label{eq:415}
\E \left[ \frac{\gamma^2}{N} \frac{2 F''(\frX^{(a)})}{(\lambda \gamma v_a - \tau)^3} (\frX - \frX^{(a)}) ( z + \gamma^2 m - \tau ) \sum_s^{(a)} G_{is} G_{si} \right] = \frac{2 \gamma^4}{(\lambda \gamma v_a - \tau)^4} \E[ F''(\frX) \frX_{22} \caX_{32} ] + \caO(\Psi^5)\,.
\end{align}
Thus, from \eqref{eq:411+412+413+414+415}, \eqref{eq:411}, \eqref{eq:412}, \eqref{eq:413}, \eqref{eq:414} and \eqref{eq:415}, we get
\begin{align} \label{eq:41}
&\E \left[ \frac{\gamma^2}{N} \frac{2 F'(\frX^{(a)})}{(\lambda \gamma v_a - \tau)^3} ( z + \gamma^2 m^{(a)} - \tau ) \sum_s^{(a)} G_{is}^{(a)} G_{si}^{(a)} \right] \nonumber \\
&\qquad= \frac{2 \gamma^2}{(\lambda \gamma v_a - \tau)^3} \E[ F'(\frX) \caX_{32} ] - \frac{4 \gamma^4}{(\lambda \gamma v_a - \tau)^4} \E[ F'(\frX) \caX_{43} ] - \frac{2 \gamma^6}{(\lambda \gamma v_a - \tau)^4} \E[ F'(\frX) \caX_{44}' ] \nonumber\\
&\qquad\qquad- \frac{2 \gamma^4}{(\lambda \gamma v_a - \tau)^4} \E[ F''(\frX) \frX_{22} \caX_{32} ] + \caO(\Psi^5)\,. 
\end{align}
This completes the expansion of the first term in \eqref{eq:41+42}.

The expansion of the second term in \eqref{eq:41+42} was already done in \eqref{eq:14111+14112+14113+14114+14115} and \eqref{eq:1411}, which shows that
\begin{align} \label{eq:42}
&\E \left[ \frac{\gamma^4}{N^2} \frac{4 F'(\frX^{(a)})}{(\lambda \gamma v_a - \tau)^3} \sum_{s, t}^{(a)} G_{is}^{(a)} G_{st}^{(a)} G_{ti}^{(a)} \right]\nonumber \\
&\qquad= \frac{4 \gamma^4}{(\lambda \gamma v_a - \tau)^3} \E[ F'(\frX) \caX_{33} ] - \frac{12 \gamma^6}{(\lambda \gamma v_a - \tau)^4} \E[ F'(\frX) \caX_{44}] - \frac{4 \gamma^6}{(\lambda \gamma v_a - \tau)^4} \E [ F''(\frX) \frX_{22} \caX_{33}] + \caO(\Psi^5) \,. 
\end{align}
We thus find from \eqref{eq:41+42}, \eqref{eq:41} and \eqref{eq:42} that
\begin{align} \label{eq:4}
&\E \left[ \frac{2 F'(\frX^{(a)})}{(\lambda \gamma v_a - \tau)^3} \left( z + \sum_{p, q}^{(a)} h_{ap} G_{pq}^{(a)} h_{qa} - \tau \right) \sum_{s, t}^{(a)} G_{is}^{(a)} h_{sa} h_{at} G_{ti}^{(a)} \right] \nonumber\\
&\quad= \frac{2 \gamma^2}{(\lambda \gamma v_a - \tau)^3} \E[ F'(\frX) \caX_{32} ] + \frac{4 \gamma^4}{(\lambda \gamma v_a - \tau)^3} \E[ F'(\frX) \caX_{33} ] - \frac{4 \gamma^4}{(\lambda \gamma v_a - \tau)^4} \E[ F'(\frX) \caX_{43} ] - \frac{12 \gamma^6}{(\lambda \gamma v_a - \tau)^4} \E[ F'(\frX) \caX_{44}] \nonumber \\
&\quad\qquad - \frac{2 \gamma^6}{(\lambda \gamma v_a - \tau)^4} \E[ F'(\frX) \caX_{44}' ] - \frac{2 \gamma^4}{(\lambda \gamma v_a - \tau)^4} \E[ F''(\frX) \frX_{22} \caX_{32} ] - \frac{4 \gamma^6}{(\lambda \gamma v_a - \tau)^4} \E [ F''(\frX) \frX_{22} \caX_{33}] + \caO(\Psi^5) \,. 
\end{align}

\subsection{Expansion of the fifth term in \eqref{eq:1+2+3+4+5+6}}

Recall that, by definition,
\begin{align}
\frX - \frX^{(a)} = \im \int_{E_1 -2}^{E_2 -2} \sum_j \frac{\wt G_{ja} \wt G_{aj}}{\wt G_{aa}} \,\dd \wt y
\end{align}
and that $\frX - \frX^{(a)} = \caO(\Psi)$. We have
\begin{align}
& \frac{2 F''(\frX^{(a)})}{(\lambda \gamma v_a - \tau)^3} (\frX - \frX^{(a)}) \left( z + \sum_{p, q}^{(a)} h_{ap} G_{pq}^{(a)} h_{qa} - \tau \right) \sum_{s, t}^{(a)} G_{is}^{(a)} h_{sa} h_{at} G_{ti}^{(a)} \\
&= \frac{2 F''(\frX^{(a)})}{(\lambda \gamma v_a - \tau)^4} \left( \im \int_{E_1 -2}^{E_2 -2} \sum_{j, r, m}^{(a)} \wt G^{(a)}_{jr} h_{ra} h_{am} \wt G^{(a)}_{mj} \,\dd \wt y \right) \left( z + \sum_{p, q}^{(a)} h_{ap} G_{pq}^{(a)} h_{qa} - \tau \right) \sum_{s, t}^{(a)} G_{is}^{(a)} h_{sa} h_{at} G_{ti}^{(a)} + \caO(\Psi^5)\,. \nonumber
\end{align}
As in \eqref{eq:122}, we take the partial expectation $\E_a$ to obtain 
\begin{align}
& \E_a \left[ \frac{2 F''(\frX^{(a)})}{(\lambda \gamma v_a - \tau)^3} (\frX - \frX^{(a)}) \left( z + \sum_{p, q}^{(a)} h_{ap} G_{pq}^{(a)} h_{qa} - \tau \right) \sum_{s, t}^{(a)} G_{is}^{(a)} h_{sa} h_{at} G_{ti}^{(a)} \right] \nonumber \\
&\qquad= \frac{2 F''(\frX^{(a)})}{(\lambda \gamma v_a - \tau)^4} \frac{\gamma^4}{N^2} \sum_{j, r, s}^{(a)} \left( \im \int_{E_1 -2}^{E_2 -2} \wt G^{(a)}_{jr} \wt G^{(a)}_{rj} \,\dd \wt y \right) ( z + \gamma^2 m^{(a)} - \tau ) G_{is}^{(a)} G_{si}^{(a)} \nonumber \\
&\qquad\qquad + \frac{4 F''(\frX^{(a)})}{(\lambda \gamma v_a - \tau)^4} \frac{\gamma^6}{N^3} \sum_{j, r, s, t}^{(a)} \left( \im \int_{E_1 -2}^{E_2 -2} \wt G^{(a)}_{jr} \wt G^{(a)}_{rj} \,\dd \wt y \right) G_{is}^{(a)} G_{st}^{(a)} G_{ti}^{(a)}\nonumber \\
&\qquad\qquad + \frac{4 F''(\frX^{(a)})}{(\lambda \gamma v_a - \tau)^4} \frac{\gamma^4}{N^2} \sum_{j, s, t}^{(a)} \left( \im \int_{E_1 -2}^{E_2 -2} \wt G^{(a)}_{js} \wt G^{(a)}_{tj} \,\dd \wt y \right) ( z + \gamma^2 m^{(a)} - \tau ) G_{is}^{(a)} G_{ti}^{(a)} \nonumber \\
&\qquad\qquad + \frac{4 F''(\frX^{(a)})}{(\lambda \gamma v_a - \tau)^4} \frac{\gamma^6}{N^3} \sum_{j, p, q, s}^{(a)} \left( \im \int_{E_1 -2}^{E_2 -2} \wt G^{(a)}_{jp} \wt G^{(a)}_{qj} \,\dd \wt y \right) G_{is}^{(a)} G_{pq}^{(a)} G_{si}^{(a)} \nonumber \\
&\qquad\qquad + \frac{16 F''(\frX^{(a)})}{(\lambda \gamma v_a - \tau)^4} \frac{\gamma^6}{N^3} \sum_{j, p, s, t}^{(a)} \left( \im \int_{E_1 -2}^{E_2 -2} \wt G^{(a)}_{jp} \wt G^{(a)}_{sj} \,\dd \wt y \right) G_{is}^{(a)} G_{pt}^{(a)} G_{ti}^{(a)} + \caO(\Psi^5)\,.
\end{align}
We thus have for the fifth term in \eqref{eq:1+2+3+4+5+6} that
\begin{align} \label{eq:5}
& \E \left[ \frac{2 F''(\frX^{(a)})}{(\lambda \gamma v_a - \tau)^3} (\frX - \frX^{(a)}) ( z + \sum_{p, q}^{(a)} h_{ap} G_{pq}^{(a)} h_{qa} - \tau ) \sum_{s, t}^{(a)} G_{is}^{(a)} h_{sa} h_{at} G_{ti}^{(a)} \right] \nonumber \\
&\qquad= \frac{2 \gamma^4}{(\lambda \gamma v_a - \tau)^4} \E [ F''(\frX) \frX_{22} \caX_{32} ]  + \frac{4 \gamma^6}{(\lambda \gamma v_a - \tau)^4} \E [ F''(\frX) \frX_{22} \caX_{33} ] \nonumber \\
&\qquad\qquad + \frac{4 \gamma^4}{(\lambda \gamma v_a - \tau)^4} \E \left[ F''(\frX) \frac{1}{N^2} \sum_{j, s, t} \left( \im \int_{E_1 -2}^{E_2 -2} \wt G_{js} \wt G_{tj} \,\dd \wt y \right) ( z + \gamma^2 m - \tau ) G_{is} G_{ti} \right] \nonumber\\
&\qquad\qquad + \frac{4 \gamma^6}{(\lambda \gamma v_a - \tau)^4} \E \left[ F''(\frX) \frac{1}{N^2} \sum_{j, p, q} \left( \im \int_{E_1 -2}^{E_2 -2} \wt G_{jp} \wt G_{qj} \,\dd \wt y \right) G_{pq} \caX_{22} \right] \nonumber \\
&\qquad\qquad + \frac{16 \gamma^6}{(\lambda \gamma v_a - \tau)^4} \E \left[ F''(\frX) \frac{1}{N^3} \sum_{j, p, s, t} \left( \im \int_{E_1 -2}^{E_2 -2} \wt G_{jp} \wt G_{sj} \,\dd \wt y \right) G_{is} G_{pt} G_{ti} \right] + \caO(\Psi^5)\,.
\end{align}

\subsection{Expansion of the sixth term in \eqref{eq:1+2+3+4+5+6}}

We take the partial expectation $\E_a$ to get
\begin{align}
&\E_a \left[ \frac{3 F'(\frX^{(a)})}{(\lambda \gamma v_a - \tau)^4} ( z + \sum_{p, q}^{(a)} h_{ap} G_{pq}^{(a)} h_{qa} - \tau )^2 \sum_{s, t}^{(a)} G_{is}^{(a)} h_{sa} h_{at} G_{ti}^{(a)} \right]\nonumber \\
&\qquad= \frac{3 F'(\frX^{(a)})}{(\lambda \gamma v_a - \tau)^4} \frac{\gamma^2}{N} ( z + \gamma^2 \wt m^{(a)} - \tau )^2 \sum_s^{(a)} G_{is}^{(a)} G_{si}^{(a)} + \frac{12 F'(\frX^{(a)})}{(\lambda \gamma v_a - \tau)^4} \frac{\gamma^4}{N^2} ( z + \gamma^2 \wt m^{(a)} - \tau ) \sum_{s, t}^{(a)} G_{is}^{(a)} G_{st}^{(a)} G_{ti}^{(a)} \nonumber \\
&\qquad\qquad+ \frac{6 F'(\frX^{(a)})}{(\lambda \gamma v_a - \tau)^4} \frac{\gamma^6}{N^3} \sum_{p, q, s}^{(a)} G_{pq}^{(a)} G_{qp}^{(a)} G_{is}^{(a)} G_{si}^{(a)} + \frac{24 F'(\frX^{(a)})}{(\lambda \gamma v_a - \tau)^4} \frac{\gamma^6}{N^3} \sum_{p, s, t}^{(a)} G_{ps}^{(a)} G_{tp}^{(a)} G_{is}^{(a)} G_{ti}^{(a)} + \caO(\Psi^5)\,. 
\end{align}
We thus find for the sixth term in \eqref{eq:1+2+3+4+5+6} that
\begin{align} \label{eq:6}
&\E \left[ \frac{3 F'(\frX^{(a)})}{(\lambda \gamma v_a - \tau)^4} ( z + \sum_{p, q}^{(a)} h_{ap} G_{pq}^{(a)} h_{qa} - \tau )^2 \sum_{s, t}^{(a)} G_{is}^{(a)} h_{sa} h_{at} G_{ti}^{(a)} \right] \nonumber \\
&\qquad= \frac{3 \gamma^2}{(\lambda \gamma v_a - \tau)^4} \E [ F'(\frX) \caX_{42} ] + \frac{12 \gamma^4}{(\lambda \gamma v_a - \tau)^4} \E [ F'(\frX) \caX_{43} ] + \frac{6 \gamma^6}{(\lambda \gamma v_a - \tau)^4} \E [ F'(\frX) \caX_{44}' ] \nonumber\\
&\qquad\qquad+ \frac{24 \gamma^6}{(\lambda \gamma v_a - \tau)^4} \E [ F'(\frX) \caX_{44} ] + \caO(\Psi^5)\,. 
\end{align}

\subsection{Expansion of \eqref{eq:1+2+3+4+5+6}}

We now collect the results \eqref{eq:1}, \eqref{eq:2}, \eqref{eq:3}, \eqref{eq:4}, \eqref{eq:5} and \eqref{eq:6}. We then find from \eqref{eq:1+2+3+4+5+6} that, for $a \neq i$,
\begin{align} \label{eq:first}
&\E [ F'(\frX) G_{ia} G_{ai} ] \nonumber \\
&\qquad= \frac{\gamma^2}{(\lambda \gamma v_a - \tau)^2} \E \left[ F'(\frX) \caX_{22} \right] + \frac{2 \gamma^2}{(\lambda \gamma v_a - \tau)^3} \E[ F'(\frX) \caX_{32} ] + \frac{2 \gamma^4}{(\lambda \gamma v_a - \tau)^3} \E[ F'(\frX) \caX_{33} ] \nonumber \\
&\qquad\qquad + \frac{2 \gamma^4}{(\lambda \gamma v_a - \tau)^3} \E \left[ F''(\frX) \frac{1}{N^2} \sum_{j, s, t} \left( \im \int_{E_1 -2}^{E_2 -2} \wt G_{js} \wt G_{tj} \,\dd \wt y \right) G_{is} G_{ti} \right] \nonumber \\
&\qquad\qquad + \frac{3 \gamma^2}{(\lambda \gamma v_a - \tau)^4} \E\, [ F'(\frX) \caX_{42} ] + \frac{6 \gamma^4}{(\lambda \gamma v_a - \tau)^4} \E\, [ F'(\frX) \caX_{43} ]\nonumber\\
&\qquad\qquad + \frac{3 \gamma^6}{(\lambda \gamma v_a - \tau)^4} \E \,[ F'(\frX) \caX_{44}' ] + \frac{12 \gamma^6}{(\lambda \gamma v_a - \tau)^4} \E\, [ F'(\frX) \caX_{44} ] \nonumber \\
&\qquad\qquad+ \frac{8 \gamma^6}{(\lambda \gamma v_a - \tau)^4} \E \left[ F''(\frX) \frac{1}{N^3} \sum_{j, r, t, s} \left( \im \int_{E_1 -2}^{E_2 -2} \wt G_{jr} \wt G_{tj} \,\dd \wt y \right) G_{ir} G_{ts} G_{si} \right]\nonumber \\
&\qquad\qquad + \frac{2 \gamma^4}{(\lambda \gamma v_a - \tau)^4} \E \left[ F''(\frX) \frac{1}{N^2} \sum_{j, s, t} \left( \im \int_{E_1 -2}^{E_2 -2} (z + \gamma^2 \wt m - \tau) \wt G_{js} \wt G_{tj} \,\dd \wt y \right) G_{is} G_{ti} \right] \nonumber \\
&\qquad\qquad + \frac{4 \gamma^6}{(\lambda \gamma v_a - \tau)^4} \E \left[ F''(\frX) \frac{1}{N^3} \sum_{j, s, t, r} \left( \im \int_{E_1 -2}^{E_2 -2} \wt G_{jr} \wt G_{rs} \wt G_{tj} \,\dd \wt y \right) G_{is} G_{ti} \right] \nonumber \\
& \qquad\qquad + \frac{2 \gamma^6}{(\lambda \gamma v_a - \tau)^4} \E \left[ F''(\frX) \frac{1}{N^3} \sum_{j, p, s, t} \left( \im \int_{E_1 -2}^{E_2 -2} \wt G_{st} \wt G_{jp} \wt G_{pj} \,\dd \wt y \right) G_{is} G_{ti} \right] \nonumber \\
&\qquad\qquad + \frac{4 \gamma^6}{(\lambda \gamma v_a - \tau)^4} \E \left[ F'''(\frX) \frac{1}{N^3} \sum_{j, k, p, s, t} \left( \im \int_{E_1 -2}^{E_2 -2} \wt G_{jp} \wt G_{sj} \,\dd \wt y \right) \left( \im \int_{E_1 -2}^{E_2 -2} \widehat G_{kp} \widehat G_{tk} \,\dd \widehat y \right) G_{is} G_{ti} \right] \nonumber \\
&\qquad\qquad + \frac{4 \gamma^4}{(\lambda \gamma v_a - \tau)^4} \E \left[ F''(\frX) \frac{1}{N^2} \sum_{j, s, t} \left( \im \int_{E_1 -2}^{E_2 -2} \wt G_{js} \wt G_{tj} \,\dd \wt y \right) ( z + \gamma^2 m - \tau ) G_{is} G_{ti} \right] \nonumber \\
&\qquad\qquad + \frac{4 \gamma^6}{(\lambda \gamma v_a - \tau)^4} \E \left[ F''(\frX) \frac{1}{N^2} \sum_{j, p, q} \left( \im \int_{E_1 -2}^{E_2 -2} \wt G_{jp} \wt G_{qj} \,\dd \wt y \right) G_{pq} \caX_{22} \right]  + \caO(\Psi^5)\,. 
\end{align}
We remark that~\eqref{eq:first} is indeed of the form~\eqref{le aim of the expansion} claimed at the beginning of this appendix.

\section{}\label{appendix IV}

In a second step of the proof of Lemma~\ref{lem:G_ii estimate}, we expand in this appendix the second to last term on the right side of~\eqref{EF G_ii expand}. As before, we always work on the event $\Xi$ and abbreviate $\caO_{\Xi}\equiv \caO$. We begin with the  expansion
\begin{align} \label{eq:a+b+c}
&\frac{1}{N} \sum_a F'(\frX) G_{ia}G_{bb}G_{ai} \nonumber \\
&\qquad= \frac{1}{N} \sum_a \frac{1}{\lambda \gamma v_b - \tau} F'(\frX) G_{ia} G_{ai} + \frac{1}{N} \sum_a^{(b)} \frac{1}{(\lambda \gamma v_b - \tau)^2} F'(\frX) \left( z + \sum_{p, q}^{(b)} h_{bp} G_{pq}^{(b)} h_{qb} - \tau \right) G_{ia} G_{ai} \nonumber\\
&\qquad\qquad + \frac{1}{N} \sum_a^{(b)} \frac{1}{(\lambda \gamma v_b - \tau)^3} F'(\frX) \left( z + \sum_{p, q}^{(b)} h_{bp} G_{pq}^{(b)} h_{qb} - \tau \right)^2 G_{ia} G_{ai} + \caO(\Psi^5)\,. 
\end{align}

Next, we decouple the indices $a$ and $b$ in~\eqref{eq:a+b+c} so that the index $b$ appears in the deterministic part only. Since this is already fulfilled in the first term on the right side of~\eqref{eq:a+b+c}, we keep the first term as it is, i.e.,
\begin{align} \label{eq:a}
\E \left[ \frac{1}{N} \sum_a \frac{1}{\lambda \gamma v_b - \tau} F'(\frX) G_{ia} G_{ai} \right] = \frac{1}{\lambda \gamma v_b - \tau} \E \,[ F'(\frX) \caX_{22}]\,.
\end{align}

Next, we expand the other two terms on the right side of~\eqref{eq:a+b+c} as in the Appendix~\ref{appendix III}.

\subsection{Expansion of the second term in \eqref{eq:a+b+c}}

Since
\begin{align}
F'(\frX) = F'(\frX^{(b)}) + F''(\frX^{(b)}) ( \frX - \frX^{(b)}) + \caO(\Psi^2)\,,
\end{align}
we have 
\begin{align} \label{eq:b1+b2+b3+b4}
& \frac{1}{N} \sum_a^{(b)} \frac{1}{(\lambda \gamma v_b - \tau)^2} F'(\frX) \left( z + \sum_{p, q}^{(b)} h_{bp} G_{pq}^{(b)} h_{qb} - \tau \right) G_{ia} G_{ai} \nonumber \\
&\qquad= \frac{1}{N} \sum_a^{(b)} \frac{1}{(\lambda \gamma v_b - \tau)^2} F'(\frX^{(b)}) \left( z + \sum_{p, q}^{(b)} h_{bp} G_{pq}^{(b)} h_{qb} - \tau \right) G_{ia}^{(b)} G_{ai}^{(b)} \nonumber \\
&\qquad\qquad + \frac{1}{N} \sum_a^{(b)} \frac{1}{(\lambda \gamma v_b - \tau)^2} F'(\frX^{(b)}) \left( z + \sum_{p, q}^{(b)} h_{bp} G_{pq}^{(b)} h_{qb} - \tau \right) G_{ia}^{(b)} \frac{G_{ab} G_{bi}}{G_{bb}}\nonumber \\
&\qquad\qquad + \frac{1}{N} \sum_a^{(b)} \frac{1}{(\lambda \gamma v_b - \tau)^2} F'(\frX^{(b)}) \left( z + \sum_{p, q}^{(b)} h_{bp} G_{pq}^{(b)} h_{qb} - \tau \right) \frac{G_{ib} G_{ba}}{G_{bb}} G_{ai} \nonumber \\
&\qquad\qquad + \frac{1}{N} \sum_a^{(b)} \frac{1}{(\lambda \gamma v_b - \tau)^2}  F''(\frX^{(b)}) ( \frX - \frX^{(b)}) \left( z + \sum_{p, q}^{(b)} h_{bp} G_{pq}^{(b)} h_{qb} - \tau \right) G_{ia} G_{ai} + \caO(\Psi^5)\,. 
\end{align} 
Taking the partial expectation $\E_b$, we find for the first term in \eqref{eq:b1+b2+b3+b4} that
\begin{align} \label{eq:b11+b12+b13+b14+b15}
&\E_b \left[ \frac{1}{N} \sum_a^{(b)} \frac{1}{(\lambda \gamma v_b - \tau)^2} F'(\frX^{(b)}) \left( z + \sum_{p, q}^{(b)} h_{bp} G_{pq}^{(b)} h_{qb} - \tau \right) G_{ia}^{(b)} G_{ai}^{(b)} \right] \nonumber \\
&\qquad= \frac{1}{N} \sum_a^{(b)} \frac{1}{(\lambda \gamma v_b - \tau)^2} F'(\frX) ( z + \gamma^2 m - \tau ) G_{ia} G_{ai} - \frac{1}{N} \sum_a^{(b)} \frac{1}{(\lambda \gamma v_b - \tau)^2} F'(\frX^{(b)}) ( z + \gamma^2 m^{(b)} - \tau ) G_{ia}^{(b)} \frac{G_{ab} G_{bi}}{G_{bb}} \nonumber \\
&\quad\qquad - \frac{1}{N} \sum_a^{(b)} \frac{1}{(\lambda \gamma v_b - \tau)^2} F'(\frX^{(b)}) ( z + \gamma^2 m^{(b)} - \tau ) \frac{G_{ib} G_{ba}}{G_{bb}} G_{ai} - \frac{1}{N} \sum_a^{(b)} \frac{1}{(\lambda \gamma v_b - \tau)^2} F'(\frX^{(b)}) (m - m^{(b)}) G_{ia} G_{ai} G_{ai} \nonumber \\
&\quad\qquad - \frac{1}{N} \sum_a^{(b)} \frac{1}{(\lambda \gamma v_b - \tau)^2}  F''(\frX^{(b)}) ( \frX - \frX^{(b)}) ( z + \gamma^2 m^{(b)} - \tau ) G_{ia} G_{ai} + \caO(\Psi^5)\,.
\end{align}
We stop expanding the first term and observe that
\begin{align} \label{eq:b11}
\E \left[ \frac{1}{N} \sum_a^{(b)} \frac{1}{(\lambda \gamma v_b - \tau)^2} F'(\frX) ( z + \gamma^2 m - \tau ) G_{ia} G_{ai} \right] = \frac{1}{(\lambda \gamma v_b - \tau)^2} \E [ F'(\frX) \caX_{32}] + \caO(\Psi^5)\,.
\end{align}
All other terms on the right side of~\eqref{eq:b11+b12+b13+b14+b15} are of $\caO(\Psi^4)$. Thus, continuing, we have
\begin{align} \label{eq:b12}
\E \left[ \frac{1}{N} \sum_a^{(b)} \frac{1}{(\lambda \gamma v_b - \tau)^2} F'(\frX^{(b)}) ( z + \gamma^2 m^{(b)} - \tau ) G_{ia}^{(b)} \frac{G_{ab} G_{bi}}{G_{bb}} \right] = \frac{\gamma^2}{(\lambda \gamma v_b - \tau)^3} \E [ F'(\frX) \caX_{43}] + \caO(\Psi^5)\,,
\end{align}
which we know from \eqref{eq:412}. Similarly, we get
\begin{align} \label{eq:b13}
\E \left[ \frac{1}{N} \sum_a^{(b)} \frac{1}{(\lambda \gamma v_b - \tau)^2} F'(\frX^{(b)}) ( z + \gamma^2 m^{(b)} - \tau ) \frac{G_{ib} G_{ba}}{G_{bb}} G_{ai} \right] = \frac{\gamma^2}{(\lambda \gamma v_b - \tau)^3} \E [ F'(\frX) \caX_{43}] + \caO(\Psi^5)\,,
\end{align}
\begin{align} \label{eq:b14}
\E \left[ \frac{1}{N} \sum_a^{(b)} \frac{1}{(\lambda \gamma v_b - \tau)^2} F'(\frX^{(b)}) (m - m^{(b)}) G_{ia} G_{ai} G_{ai} \right] = \frac{\gamma^2}{(\lambda \gamma v_b - \tau)^3} \E [ F'(\frX) \caX_{44}'] + \caO(\Psi^5)\,,
\end{align}
and, from \eqref{eq:415},
\begin{align} \label{eq:b15}
&\E \left[ \frac{1}{N} \sum_a^{(b)} \frac{1}{(\lambda \gamma v_b - \tau)^2}  F''(\frX^{(b)}) ( \frX - \frX^{(b)}) ( z + \gamma^2 m^{(b)} - \tau ) G_{ia} G_{ai} \right]= \frac{\gamma^2}{(\lambda \gamma v_b - \tau)^3} \E [ F''(\frX) \frX_{22} \caX_{32}] + \caO(\Psi^5)\,.
\end{align}
Thus, we find from \eqref{eq:b11+b12+b13+b14+b15}, \eqref{eq:b11}, \eqref{eq:b12}, \eqref{eq:b13}, \eqref{eq:b14} and \eqref{eq:b15} that
\begin{align} \label{eq:b1}
&\E \left[ \frac{1}{N} \sum_a^{(b)} \frac{1}{(\lambda \gamma v_b - \tau)^2} F'(\frX^{(b)}) \left( z + \sum_{p, q}^{(b)} h_{bp} G_{pq}^{(b)} h_{qb} - \tau \right) G_{ia}^{(b)} G_{ai}^{(b)} \right] \nonumber \\
&\qquad= \frac{1}{(\lambda \gamma v_b - \tau)^2} \E [ F'(\frX) \caX_{32}] - \frac{2 \gamma^2}{(\lambda \gamma v_b - \tau)^3} \E [ F'(\frX) \caX_{43}] - \frac{\gamma^4}{(\lambda \gamma v_b - \tau)^3} \E [ F'(\frX) \caX_{44}']\nonumber \\
&\qquad\qquad - \frac{\gamma^2}{(\lambda \gamma v_b - \tau)^3} \E [ F''(\frX) \frX_{22} \caX_{32}] + \caO(\Psi^5)\,. 
\end{align}

In order to expand the second term in \eqref{eq:b1+b2+b3+b4}, we first notice that
\begin{align*}
&\frac{1}{N} \sum_a^{(b)} \frac{1}{(\lambda \gamma v_b - \tau)^2} F'(\frX^{(b)}) \left( z + \sum_{p, q}^{(b)} h_{bp} G_{pq}^{(b)} h_{qb} - \tau \right) G_{ia}^{(b)} \frac{G_{ab} G_{bi}}{G_{bb}} \nonumber\\
&\qquad= \frac{1}{N} \frac{1}{(\lambda \gamma v_b - \tau)^3} F'(\frX^{(b)}) \left( z + \sum_{p, q}^{(b)} h_{bp} G_{pq}^{(b)} h_{qb} - \tau \right) \sum_{a, p, q}^{(b)} G_{ia}^{(b)} G_{ap}^{(b)} h_{pb} h_{bq} G_{qi}^{(b)} + \caO(\Psi^5)\,.
\end{align*}
Then, using~\eqref{eq:1412}, we find that
\begin{align} \label{eq:b2}
&\E \left[ \frac{1}{N} \sum_a^{(b)} \frac{1}{(\lambda \gamma v_b - \tau)^2} F'(\frX^{(b)}) \left( z + \sum_{p, q}^{(b)} h_{bp} G_{pq}^{(b)} h_{qb} - \tau \right) G_{ia}^{(b)} \frac{G_{ab} G_{bi}}{G_{bb}} \right]\nonumber \\
&\qquad= \frac{\gamma^2}{(\lambda \gamma v_b - \tau)^3} \E [ F'(\frX) \caX_{43}] + \frac{2 \gamma^4}{(\lambda \gamma v_b - \tau)^3} \E [ F'(\frX) \caX_{44}] + \caO(\Psi^5)\,.
\end{align}
Similarly, we also have
\begin{align} \label{eq:b3}
&\E \left[ \frac{1}{N} \sum_a^{(b)} \frac{1}{(\lambda \gamma v_b - \tau)^2} F'(\frX^{(b)}) \left( z + \sum_{p, q}^{(b)} h_{bp} G_{pq}^{(b)} h_{qb} - \tau \right) \frac{G_{ib} G_{ba}}{G_{bb}} G_{ai} \right]\nonumber \\
&\qquad= \frac{\gamma^2}{(\lambda \gamma v_b - \tau)^3} \E [ F'(\frX) \caX_{43}] + \frac{2 \gamma^4}{(\lambda \gamma v_b - \tau)^3} \E [ F'(\frX) \caX_{44}] + \caO(\Psi^5)\,.
\end{align}

Finally, for the last term in \eqref{eq:b1+b2+b3+b4}, we observe that
\begin{align*}
&\frac{1}{N} \sum_a^{(b)} \frac{F''(\frX^{(b)})}{(\lambda \gamma v_b - \tau)^2} ( \frX - \frX^{(b)}) \left( z + \sum_{p, q}^{(b)} h_{bp} G_{pq}^{(b)} h_{qb} - \tau \right) G_{ia} G_{ai}\nonumber \\
&\qquad= \frac{1}{N} \sum_a^{(b)} \frac{F''(\frX^{(b)})}{(\lambda \gamma v_b - \tau)^3} \left( \im \int_{E_1 -2}^{E_2 -2} \sum_{j, r, m}^{(b)} \wt G^{(b)}_{jr} h_{ra} h_{am} \wt G^{(b)}_{mj} \,\dd \wt y \right) \left( z + \sum_{p, q}^{(b)} h_{bp} G_{pq}^{(b)} h_{qb} - \tau \right) G_{ia}^{(b)} G_{ai}^{(b)} + \caO(\Psi^5)\,. 
\end{align*}
Thus taking the partial expectation we find 
\begin{align} \label{eq:b4}
&\E \left[ \frac{1}{N} \sum_a^{(b)} \frac{F''(\frX^{(b)})}{(\lambda \gamma v_b - \tau)^2} ( \frX - \frX^{(b)}) \left( z + \sum_{p, q}^{(b)} h_{bp} G_{pq}^{(b)} h_{qb} - \tau \right) G_{ia} G_{ai} \right] \\
&\qquad= \frac{\gamma^2}{(\lambda \gamma v_b - \tau)^3} \E [ F''(\frX) \frX_{22} \caX_{32}] + \frac{2 \gamma^4}{(\lambda \gamma v_b - \tau)^3} \E \left[ F''(\frX) \frac{1}{N^2} \sum_{j, p, q} \left( \im \int_{E_1 -2}^{E_2 -2} \wt G_{jp} \wt G_{qj} \,\dd \wt y \right) G_{pq} \caX_{22} \right] + \caO(\Psi^5)\,. \nonumber
\end{align}

Combining \eqref{eq:b1+b2+b3+b4}, \eqref{eq:b1}, \eqref{eq:b2}, \eqref{eq:b3} and \eqref{eq:b4}, we obtain
\begin{align} \label{eq:b}
&\E \left[ \frac{1}{N} \sum_a^{(b)} \frac{1}{(\lambda \gamma v_b - \tau)^2} F'(\frX) \left( z + \sum_{p, q}^{(b)} h_{bp} G_{pq}^{(b)} h_{qb} - \tau \right) G_{ia} G_{ai} \right] \nonumber \\
&\qquad= \frac{1}{(\lambda \gamma v_b - \tau)^2} \E [ F'(\frX) \caX_{32}] - \frac{\gamma^2}{(\lambda \gamma v_b - \tau)^3} \E [ F'(\frX) \caX_{44}'] + \frac{4 \gamma^4}{(\lambda \gamma v_b - \tau)^3} \E [ F'(\frX) \caX_{44}]\nonumber \\
&\qquad\qquad + \frac{2 \gamma^4}{(\lambda \gamma v_b - \tau)^3} \E \left[ F''(\frX) \frac{1}{N^2} \sum_{j, p, q} \left( \im \int_{E_1 -2}^{E_2 -2} \wt G_{jp} \wt G_{qj} \,\dd \wt y \right) G_{pq} \caX_{22} \right] + \caO(\Psi^5)\,. 
\end{align}
This completes the expansion of the second term in \eqref{eq:a+b+c}.

\subsection{Expansion of the third term in \eqref{eq:a+b+c}}

We next focus on the third term in~\eqref{eq:a+b+c} which is of $\caO(\Psi^4)$. We have
\begin{align}
&\E_b \left[ \frac{1}{N} \sum_a^{(b)} \frac{1}{(\lambda \gamma v_b - \tau)^3} F'(\frX) \left( z + \sum_{p, q}^{(b)} h_{bp} G_{pq}^{(b)} h_{qb} - \tau \right)^2 G_{ia} G_{ai} \right] \nonumber \\
&\qquad= \E_b \left[ \frac{1}{N} \sum_a^{(b)} \frac{1}{(\lambda \gamma v_b - \tau)^3} F'(\frX^{(b)}) \left( z + \sum_{p, q}^{(b)} h_{bp} G_{pq}^{(b)} h_{qb} - \tau \right)^2 G_{ia}^{(b)} G_{ai}^{(b)} \right] + \caO(\Psi^5)\nonumber \\
&\qquad= \frac{F'(\frX^{(b)})}{(\lambda \gamma v_b - \tau)^3} \frac{1}{N} \sum_a^{(b)} ( z + \gamma^2 m^{(b)} - \tau )^2 G_{ia}^{(b)} G_{ai}^{(b)} + \frac{F'(\frX^{(b)})}{(\lambda \gamma v_b - \tau)^3} \frac{\gamma^4}{N} \sum_{a, p, r}^{(b)} G_{pr}^{(b)} G_{rp}^{(b)} G_{ia}^{(b)} G_{ai}^{(b)} + \caO(\Psi^5)\,.
\end{align}
Thus, we get
\begin{align} \label{eq:c}
&\E \left[ \frac{1}{N} \sum_a^{(b)} \frac{1}{(\lambda \gamma v_b - \tau)^3} F'(\frX) \left( z + \sum_{p, q}^{(b)} h_{bp} G_{pq}^{(b)} h_{qb} - \tau \right)^2 G_{ia} G_{ai} \right] \nonumber\\
&\qquad= \frac{1}{(\lambda \gamma v_b - \tau)^3} \E [F'(\frX) \caX_{42}] + \frac{2 \gamma^4}{(\lambda \gamma v_b - \tau)^3} \E [F'(\frX) \caX_{44}'] + \caO(\Psi^5)\,.
\end{align}

\subsection{Expansion of $\E [F'(\frX) G_{ia}G_{bb}G_{ai}]$}

From \eqref{eq:a+b+c}, \eqref{eq:a}, \eqref{eq:b} and \eqref{eq:c}, we obtain
\begin{align} \label{eq:second}
&\E \left[ \frac{1}{N} \sum_a F'(\frX) G_{ia}G_{bb}G_{ai} \right] \nonumber \\
&\quad= \frac{1}{\lambda \gamma v_b - \tau} \E [ F'(\frX) \caX_{22}] + \frac{1}{(\lambda \gamma v_b - \tau)^2} \E [ F'(\frX) \caX_{32}] + \frac{1}{(\lambda \gamma v_b - \tau)^3} \E [ F'(\frX) \caX_{42}]  + \frac{\gamma^4}{(\lambda \gamma v_b - \tau)^3} \E [ F'(\frX) \caX_{44}']\nonumber \\
&\quad\quad  + \frac{4 \gamma^4}{(\lambda \gamma v_b - \tau)^3} \E [ F'(\frX) \caX_{44}] + \frac{2 \gamma^4}{(\lambda \gamma v_b - \tau)^3} \E \left[ F''(\frX) \frac{1}{N^2} \sum_{j, p, q} \left( \im \int_{E_1 -2}^{E_2 -2} \wt G_{jp} \wt G_{qj} \,\dd \wt y \right) G_{pq} \caX_{22} \right] + \caO(\Psi^5)\,. 
\end{align}

\begin{remark}
Choosing $F' \equiv 1$ we infer from \eqref{EF G_ii expand}, \eqref{eq:first} and \eqref{eq:second} that
\begin{align}
\E \,\dd G_{ii} &= \left( -\gamma^2 \partial_t(\lambda\gamma) \sum_j \frac{v_j}{(\lambda \gamma v_a - \tau)^2} + \dot z + 2 \dot\gamma\gamma \sum_j \frac{1}{\lambda \gamma v_j - \tau} \right) \E\, [\caX_{22} ]\,\dd t \nonumber \\
&\qquad\quad + \left( -2 \gamma^2 \partial_t(\lambda\gamma) \sum_j \frac{v_j}{(\lambda \gamma v_a - \tau)^3} + 2\dot\gamma\gamma^{-1} \right) \E \,[\caX_{32} + \gamma^2 \caX_{33}] \,\dd t\nonumber \\
&\qquad\quad -3 \gamma^2 \partial_t(\lambda\gamma) \left( \sum_j \frac{v_j}{(\lambda \gamma v_a - \tau)^4} \right) \E\, \big[ \caX_{42} + 2 \gamma^2 \caX_{43} + 4 \gamma^4 \caX_{44} + \gamma^4 \caX_{44}' \big] \,\dd t \nonumber \\
&\qquad\quad + 2\dot\gamma\gamma \left( \sum_j \frac{1}{(\lambda \gamma v_a - \tau)^3} \right) \E\, \big[ \caX_{42} + 4 \gamma^4 \caX_{44} + \gamma^4 \caX_{44}' \big] \,\dd t + \E\,[G_{ii}^2] \,\dd t + \caO(N^{1/2} \Psi^3) \,\dd t \,,
\end{align}
where we used the sum rule
\begin{align}
\sum_j \frac{1}{(\lambda \gamma v_j - \tau)^2} = \frac{1}{\gamma^2}\,.
\end{align}
Note that we have added the term $G_{ii}$ at the end to account for the case $a=i$.
\end{remark}

\section{}\label{appendix V}

In a third step of the proof of Lemma~\ref{lem:G_ii estimate}, we further simplify in this appendix the right side of~\eqref{EF G_ii expand}, using~\eqref{eq:first} and~\eqref{eq:second}. We always work on the event $\Xi$ and abbreviate $\caO_{\Xi}\equiv \caO$.

Let
\begin{align} \label{X_2}
X_2 \deq \frac{1}{N} \sum_i F'(\frX) \caX_{22}
\end{align}
and
\begin{align} \label{X_3}
X_3 \deq \frac{1}{N} \sum_i \left( F'(\frX) \caX_{32} + \gamma^2 F'(\frX) \caX_{33} + F''(\frX) \frac{\gamma^2}{N^2} \sum_{j, s, t} \left( \im \int_{E_1 -2}^{E_2 -2} \wt G_{js} \wt G_{tj} \,\dd \wt y \right) G_{is} G_{ti} \right)\,.
\end{align}
It is obvious that $X_2 = \caO(\Psi^2)$ and $X_3 = \caO(\Psi^3)$. Furthermore, for $a \neq i$, we let $X_4$ be a random variable not containing $a$ as a fixed index that satisfies
\begin{align} \label{X_4}
\frac{1}{N} \sum_i \E [F'(\frX) G_{ia} G_{ai}] &= \frac{\gamma^2}{(\lambda \gamma v_a - \tau)^2} \E X_2 + \frac{2\gamma^2}{(\lambda \gamma v_a - \tau)^3} \E X_3 + \frac{\gamma^2}{(\lambda \gamma v_a - \tau)^4} \E X_4 + \caO(\Psi^5) 
\end{align}
and $X_4 = \caO(\Psi^4)$. We can easily check the existence of such an $X_4$ from \eqref{eq:first}.

Using the notations in \eqref{A_n}, we have from \eqref{X_4} that, after summing over the index $a$,
\begin{align}
N \E\,[ X_2 ]= N \gamma^2 A_2\E\,[ X_2 ]+ 2N \gamma^2 A_3 \E\,[ X_3] + N \gamma^2 A_4 \E\,[ X_4] + \frac{1}{N} \sum_i \E\,[ F'(\frX) G_{ii}^2] + \caO(\Psi^2)\,,
\end{align}
which also implies
\begin{align} \label{3-4 relation}
\E \,[X_4] = -\frac{2 A_3}{A_4} \E\,[ X_3] - \frac{1}{\gamma^2 A_4 N^2} \sum_i \E\,[ F'(\frX) G_{ii}^2] + \caO(\Psi^5)\,,
\end{align}
 where we used that $A_2 = \gamma^{-2}$.

\subsection{Simplification of \eqref{eq:first}}

We next consider the first term~\eqref{EF G_ii expand}, where we have an additional factor $v_a$ in the summand. Since $v_a$ is considered fix, the addition of such a factor does not change the expansion results we have obtained so far. We hence get
\begin{align} \label{eq:first_2}
&\frac{1}{N} \sum_{i, a} v_a \E \,[F'(\frX) G_{ia} G_{ai}] \nonumber \\
&\qquad= \sum_a \frac{\gamma^2 v_a}{(\lambda \gamma v_a - \tau)^2} \E\, X_2 + \sum_a \frac{2\gamma^2 v_a}{(\lambda \gamma v_a - \tau)^3} \E \,X_3 + \sum_a \frac{\gamma^2 v_a}{(\lambda \gamma v_a - \tau)^4} \E\, X_4 + \frac{1}{N} \sum_i \E\,[ v_i F'(\frX) G_{ii}^2] \nonumber \\
&\qquad= N \gamma^2 A_2' \E\, X_2 + 2N \gamma^2 A_3' \E\, X_3 + N \gamma^2 A_4' \E\, X_4 + \frac{1}{N} \sum_i \E\,[ v_i F'(\frX) G_{ii}^2] + \caO(\Psi^2)\nonumber \\
&\qquad= N \gamma^2 A_2' \E\, X_2 + 2N \gamma^2 \left( A_3' - \frac{A_3 A_4'}{A_4} \right) \E\, X_3 - \frac{A_4'}{A_4} \frac{1}{N} \sum_i \E\,[ F'(\frX) G_{ii}^2] + \frac{1}{N} \sum_i \E\,[ v_i F'(\frX) G_{ii}^2] + \caO(\Psi^2)\,. 
\end{align}
The last two terms on the right side of~\eqref{eq:first_2} can be further expanded using
\begin{align*}
F'(\frX) G_{ii}^2 = \frac{1}{(\lambda \gamma v_i - \tau)^2} F'(\frX) + \frac{2}{(\lambda \gamma v_i - \tau)^3} \left( z + \sum_{p, q}^{(i)} h_{ip} G_{pq}^{(i)} h_{qi} - \tau \right) F'(\frX) + \caO(\Psi^2)\,.
\end{align*}
We further observe that
\begin{align*}
 \E_i\, \left[ \left( z + \sum_{p, q}^{(i)} h_{ip} G_{pq}^{(i)} h_{qi} - \tau \right) F'(\frX) \right] &= \E_i \,\left[ \left( z + \sum_{p, q}^{(i)} h_{ip} G_{pq}^{(i)} h_{qi} - \tau \right) F'(\frX^{(i)}) \right] + \caO(\Psi^2) \nonumber \\
&= (z + \gamma^2 m^{(i)} - \tau) F'(\frX^{(i)}) + \caO(\Psi^2)\nonumber\\& = (z + \gamma^2 m - \tau) F'(\frX) + \caO(\Psi^2)\,,
\end{align*}
and we hence obtain
\begin{align*}
\E\, [ F'(\frX) G_{ii}^2 ] = \frac{1}{(\lambda \gamma v_i - \tau)^2} \E\, [F'(\frX)] + \frac{2}{(\lambda \gamma v_i - \tau)^3} \E\,[ (z + \gamma^2 m - \tau) F'(\frX)] + \caO(\Psi^2)\,.
\end{align*}
We thus have 
\begin{align} \label{G_ii^2}
\frac{1}{N} \sum_i \E\,[ F'(\frX) G_{ii}^2] = A_2 \E\, [F'(\frX)] + 2A_3 \E\,[ (z + \gamma^2 m - \tau) F'(\frX)] + \caO(\Psi^2)
\end{align}
and, similarly,
\begin{align}\label{G_ii^2 2}
\frac{1}{N} \sum_i \E\,[ v_i F'(\frX) G_{ii}^2] = A_2' \E \,[F'(\frX)] + 2A_3' \E\,[ (z + \gamma^2 m - \tau) F'(\frX)] + \caO(\Psi^2)\,.
\end{align}
Putting~\eqref{G_ii^2} and~\eqref{G_ii^2 2} back into~\eqref{eq:first_2}, we find 
\begin{align} \label{eq:first_final}
\frac{1}{N} \sum_{i, a} v_a \E\, [F'(\frX) G_{ia} G_{ai}] &= N \gamma^2 A_2' \E\,[ X_2] + 2N \gamma^2 \left( A_3' - \frac{A_3 A_4'}{A_4} \right) \E\,[ X_3] + \left( A_2'- \frac{A_2 A_4'}{A_4} \right) \E\,[F'(\frX)] \nonumber\\
&\qquad + \left( A_3'- \frac{A_3 A_4'}{A_4} \right) \E[(z + \gamma^2 m - \tau) F'(\frX)] + \caO(\Psi^2)\,.
\end{align}

We now go back to the other terms in \eqref{EF G_ii expand}. We have from \eqref{eq:second} that
\begin{align} \label{eq:second_2}
&\frac{1}{N^2} \sum_{i, a, b} \big(\E\, [F'(\frX) G_{ia}G_{ab}G_{bi}] + \E\, [F'(\frX) G_{ia}G_{bb}G_{ai}] \big) \nonumber \\
&\qquad + \frac{1}{N^2} \sum_{i, a, b, j} \E\, \left[ F''(\frX) \left( \im \int_{E_1 -2}^{E_2 -2} G_{ja}' G_{bj}' \,\dd \wt y \right) G_{ia}G_{bi} \right]\nonumber \\
&= N A_1 \E\,[ X_2 ]+ N A_2 \E\, [X_3] + A_3 \sum_i \left( \E\,[ F'(\frX) \caX_{42} ] + \gamma^4 \E\, [F'(\frX) \caX_{44}'] + 4 \gamma^4 \E\,[ F'(\frX) \caX_{44} ] \right) \nonumber \\
&\qquad + 2 \gamma^4 A_3 \E\, \left[ F''(\frX) \frac{1}{N^3} \sum_{i, j, p, q, s} \left( \im \int_{E_1 -2}^{E_2 -2} \wt G_{jp} \wt G_{qj} \,\dd \wt y \right) G_{pq} G_{is} G_{si} \right] + \caO(\Psi^2)\,. 
\end{align}
We note that the right side of~\eqref{eq:second_2} contains several terms of $\caO(\Psi)$ that are not directly related to $X_4$. In order to compare those terms with the terms in \eqref{eq:first_final}, we expand the terms in $X_3$ to find more ``optical theorems''.

\subsection{Optical theorem from $\caX_{32}$}

For $F'(\frX) \caX_{32}$, we consider
\begin{align} \label{eq:321+322}
F'(\frX) \caX_{32} = \frac{F'(\frX)}{N} (z+ \gamma^2 m - \tau) \sum_a G_{ia} G_{ai} = \frac{F'(\frX)}{N} (z+ \gamma^2 m - \tau) \left( G_{ii}^2 + \sum_a^{(i)} G_{ia} G_{ai} \right)
\end{align}
and expand each term on the right side. We can easily see that the first term of \eqref{eq:321+322} becomes
\begin{align} \label{eq:321}
\frac{F'(\frX)}{N} (z+ \gamma^2 m - \tau) G_{ii}^2 = \frac{1}{(\lambda \gamma v_i - \tau)^2} \frac{F'(\frX)}{N} (z+ \gamma^2 m - \tau) + \caO(\Psi^2)\,.
\end{align}

We next expand the second term of \eqref{eq:321+322} to find that, for $a \neq i$,
\begin{align} \label{eq:3221+3222}
& F'(\frX) (z+ \gamma^2 m - \tau) G_{ia} G_{ai} \nonumber \\
& \qquad= \frac{1}{(\lambda \gamma v_a - \tau)^2} F'(\frX) (z+ \gamma^2 m - \tau) \sum_{p, q}^{(a)} G_{ip}^{(a)} h_{pa} h_{aq} G_{qi}^{(a)}\nonumber \\
&\qquad\qquad + \frac{2}{(\lambda \gamma v_a - \tau)^3} F'(\frX) (z+ \gamma^2 m - \tau) \left( z + \sum_{p, q}^{(a)} h_{ap} G_{pq}^{(a)} h_{qa} - \tau \right) \sum_{p, q}^{(a)} G_{ip}^{(a)} h_{pa} h_{aq} G_{qi}^{(a)} + \caO(\Psi^5)\,. 
\end{align}
Expanding further the first term in \eqref{eq:3221+3222}, we get
\begin{align} \label{eq:32211+32212+32213}
&\frac{1}{(\lambda \gamma v_a - \tau)^2} F'(\frX) (z+ \gamma^2 m - \tau) \sum_{p, q}^{(a)} G_{ip}^{(a)} h_{pa} h_{aq} G_{qi}^{(a)} \nonumber \\
&\qquad= \frac{1}{(\lambda \gamma v_a - \tau)^2} F'(\frX^{(a)}) (z+ \gamma^2 m^{(a)} - \tau) \sum_{p, q}^{(a)} G_{ip}^{(a)} h_{pa} h_{aq} G_{qi}^{(a)} \nonumber\\
&\qquad\qquad + \frac{\gamma^2}{(\lambda \gamma v_i - \tau)^2} F'(\frX^{(a)}) (m - m^{(a)}) \sum_{p, q}^{(a)} G_{ip}^{(a)} h_{pa} h_{aq} G_{qi}^{(a)} \nonumber \\
&\qquad\qquad + \frac{1}{(\lambda \gamma v_i - \tau)^2} F''(\frX^{(a)}) (\frX - \frX^{(a)}) (z+ \gamma^2 m - \tau) \sum_{p, q}^{(a)} G_{ip}^{(a)} h_{pa} h_{aq} G_{qi}^{(a)} + \caO(\Psi^5)\,.
\end{align}
We then take the partial expectation $\E_a$ for each term in \eqref{eq:32211+32212+32213}. The first term yields 
\begin{align}
&\E_a \left[ \frac{1}{(\lambda \gamma v_a - \tau)^2} F'(\frX^{(a)}) (z+ \gamma^2 m^{(a)} - \tau) \sum_{p, q}^{(a)} G_{ip}^{(a)} h_{pa} h_{aq} G_{qi}^{(a)} \right]\nonumber \\
&\qquad= \frac{\gamma^2}{N} \frac{1}{(\lambda \gamma v_a - \tau)^2} F'(\frX^{(a)}) (z+ \gamma^2 m^{(a)} - \tau) \sum_p^{(a)} G_{ip}^{(a)} G_{pi}^{(a)}\,.
\end{align}
Thus, taking the full expectation we obtain with \eqref{eq:41} that
\begin{align} \label{eq:32211}
&\E\, \left[ \frac{1}{(\lambda \gamma v_a - \tau)^2} F'(\frX^{(a)}) (z+ \gamma^2 m^{(a)} - \tau) \sum_{p, q}^{(a)} G_{ip}^{(a)} h_{pa} h_{aq} G_{qi}^{(a)} \right] \nonumber \\
&\qquad= \frac{\gamma^2}{(\lambda \gamma v_a - \tau)^2} \E\,[ F'(\frX) \caX_{32} ] - \frac{2 \gamma^4}{(\lambda \gamma v_a - \tau)^3} \E\,[ F'(\frX) \caX_{43} ] - \frac{\gamma^6}{(\lambda \gamma v_a - \tau)^3} \E\,[ F'(\frX) \caX_{44}' ]\nonumber \\
&\qquad\qquad - \frac{\gamma^4}{(\lambda \gamma v_a - \tau)^3} \E\,[ F''(\frX) \frX_{22} \caX_{32} ] + \caO(\Psi^5)\,. 
\end{align}
The second term in \eqref{eq:32211+32212+32213} can be expanded similarly, using the relation
$$
m - m^{(a)} = \frac{1}{N} \sum_r \frac{G_{ra} G_{ar}}{G_{aa}}.
$$
We then have
\begin{align} \label{eq:32212}
&\E \,\left[ \frac{\gamma^2}{(\lambda \gamma v_i - \tau)^2} F'(\frX^{(a)}) (m - m^{(a)}) \sum_{p, q}^{(a)} G_{ip}^{(a)} h_{pa} h_{aq} G_{qi}^{(a)} \right]\nonumber \\
&\qquad= \frac{\gamma^6}{(\lambda \gamma v_a - \tau)^3} \E\,[ F'(\frX) \caX_{44}' ] + \frac{2 \gamma^6}{(\lambda \gamma v_a - \tau)^3} \E\,[ F'(\frX) \caX_{44} ] + \caO(\Psi^5)\,. 
\end{align}
Finally, we have for the third term in \eqref{eq:32211+32212+32213}, from the relation
$$
\frX - \frX^{(a)} = \im \int_{E_1 -2}^{E_2 -2} \sum_j \frac{\wt G_{ja} \wt G_{aj}}{\wt G_{aa}} \,\dd \wt y \,,
$$
that
\begin{align} \label{eq:32213}
&\E\left[ \frac{1}{(\lambda \gamma v_i - \tau)^2} F''(\frX^{(a)}) (\frX - \frX^{(a)}) (z+ \gamma^2 m - \tau) \sum_{p, q}^{(a)} G_{ip}^{(a)} h_{pa} h_{aq} G_{qi}^{(a)} \right] \nonumber \\
&\qquad= \frac{\gamma^4}{(\lambda \gamma v_a - \tau)^3} \E\,[ F'(\frX) \frX_{22} \caX_{32} ]\nonumber \\
&\qquad\qquad + \frac{2 \gamma^4}{(\lambda \gamma v_a - \tau)^3} \E \left[ F''(\frX) \frac{1}{N^2} \sum_{r, p, q} \left( \im \int_{E_1 -2}^{E_2 -2} \wt G_{rp} \wt G_{qr} \,\dd \wt y \right) ( z + \gamma^2 m - \tau ) G_{ip} G_{qi} \right] + \caO(\Psi^5)\,. 
\end{align}

Thus, from \eqref{eq:32211+32212+32213}, \eqref{eq:32211}, \eqref{eq:32212} and \eqref{eq:32213}, we get
\begin{align} \label{eq:3221}
&\E \left[ \frac{1}{(\lambda \gamma v_a - \tau)^2} F'(\frX) (z+ \gamma^2 m - \tau) \sum_{p, q}^{(a)} G_{ip}^{(a)} h_{pa} h_{aq} G_{qi}^{(a)} \right] \nonumber \\
&\qquad= \frac{\gamma^2}{(\lambda \gamma v_a - \tau)^2} \E\,[ F'(\frX) \caX_{32} ] - \frac{2 \gamma^4}{(\lambda \gamma v_a - \tau)^3} \E\,[ F'(\frX) \caX_{43} ] + \frac{2 \gamma^6}{(\lambda \gamma v_a - \tau)^3} \E\,[ F'(\frX) \caX_{44} ] \nonumber\\
&\qquad\qquad + \frac{2 \gamma^4}{(\lambda \gamma v_a - \tau)^3} \E \left[ F''(\frX) \frac{1}{N^2} \sum_{r, p, q} \left( \im \int_{E_1 -2}^{E_2 -2} \wt G_{rp} \wt G_{qr} \,\dd \wt y \right) ( z + \gamma^2 m - \tau ) G_{ip} G_{qi} \right] + \caO(\Psi^5)\,, 
\end{align}
which completes the expansion of the first term in \eqref{eq:3221+3222}. The second term in \eqref{eq:3221+3222} is of $\caO(\Psi^2)$ and we observe that
\begin{align} \label{eq:3222}
&\E \left[ \frac{2}{(\lambda \gamma v_a - \tau)^3} F'(\frX) (z+ \gamma^2 m - \tau) \left( z + \sum_{p, q}^{(a)} h_{ap} G_{pq}^{(a)} h_{qa} - \tau \right) \sum_{p, q}^{(a)} G_{ip}^{(a)} h_{pa} h_{aq} G_{qi}^{(a)} \right] \nonumber \\
&\qquad= \frac{2 \gamma^2}{(\lambda \gamma v_a - \tau)^3} \E[ F'(\frX) \caX_{42} ] + \frac{4 \gamma^4}{(\lambda \gamma v_a - \tau)^3} \E[ F'(\frX) \caX_{43} ] + \caO(\Psi^5)\,.
\end{align}
Combining \eqref{eq:321+322}, \eqref{eq:321}, \eqref{eq:3221+3222}, \eqref{eq:3221} and \eqref{eq:3222}, we conclude that
\begin{align}
\E\,[ F'(\frX) \caX_{32} ] &= \E\,[ F'(\frX) \caX_{32} ] + 2 A_3 \big( \gamma^2 \E\,[ F'(\frX) \caX_{42} ] + \gamma^4 \E\,[ F'(\frX) \caX_{43} ] + \gamma^6 \E\,[ F'(\frX) \caX_{44} ] \big) \nonumber \\
&\qquad + 2 A_3 \gamma^4 \E\, \left[ F''(\frX) \frac{1}{N^2} \sum_{r, p, q} \left( \im \int_{E_1 -2}^{E_2 -2} \wt G_{rp} \wt G_{qr} \,\dd \wt y \right) ( z + \gamma^2 m - \tau ) G_{ip} G_{qi} \right]\nonumber \\
&\qquad + \frac{1}{(\lambda \gamma v_i - \tau)^2} \frac{1}{N} \E\, \left[ F'(\frX) (z+ \gamma^2 m - \tau) \right] + \caO(\Psi^5)\,, 
\end{align}
which implies the first ``optical theorem'' of this appendix,
\begin{align} \label{3-4 relation 2}
&\E\,[ F'(\frX) \caX_{42} ] + \gamma^2 \E\,[ F'(\frX) \caX_{43} ] + \gamma^4 \E\,[ F'(\frX) \caX_{44} ] \big) \nonumber \\
&\qquad + \gamma^2 \E \left[ F''(\frX) \frac{1}{N^2} \sum_{r, p, q} \left( \im \int_{E_1 -2}^{E_2 -2} \wt G_{rp} \wt G_{qr} \,\dd \wt y \right) ( z + \gamma^2 m - \tau ) G_{ip} G_{qi} \right]\nonumber \\
&= - \frac{1}{2 \gamma^2 A_3} \frac{1}{(\lambda \gamma v_i - \tau)^2} \frac{1}{N} \E\, \left[ F'(\frX) (z+ \gamma^2 m - \tau) \right] + \caO(\Psi^5)\,. 
\end{align}

\subsection{Optical theorem from $\caX_{33}$}

We perform a similar expansion for $F'(\frX) \caX_{33}$. We first notice that
\begin{align} \label{eq:331+332}
&F'(\frX) \frac{1}{N} \sum_b G_{ia} G_{ab} G_{bi} \nonumber \\
&\qquad= \frac{1}{(\lambda \gamma v_a - \tau)^2} \frac{F'(\frX)}{N} \sum_{b, s, t}^{(a)} G_{is}^{(a)} h_{sa} h_{at} G_{tb}^{(a)} G_{bi}\nonumber \\
&\qquad\qquad + \frac{2}{(\lambda \gamma v_a - \tau)^3} \frac{F'(\frX)}{N} \left( z + \sum_{p, q}^{(a)} h_{ap} G_{pq}^{(a)} h_{qa} - \tau \right) \sum_{b, s, t}^{(a)} G_{is}^{(a)} h_{sa} h_{at} G_{tb}^{(a)} G_{bi} + \caO(\Psi^5)\,. 
\end{align}
The first term in \eqref{eq:331+332} can be written as
\begin{align} \label{eq:3311+3312+3313}
&\frac{1}{(\lambda \gamma v_a - \tau)^2} \frac{F'(\frX)}{N} \sum_{b, s, t}^{(a)} G_{is}^{(a)} h_{sa} h_{at} G_{tb}^{(a)} G_{bi}\nonumber \\
&\qquad= \frac{1}{(\lambda \gamma v_a - \tau)^2} \frac{F'(\frX^{(a)})}{N} \sum_{b, s, t}^{(a)} G_{is}^{(a)} h_{sa} h_{at} G_{tb}^{(a)} G_{bi}^{(a)} + \frac{1}{(\lambda \gamma v_a - \tau)^2} \frac{F'(\frX^{(a)})}{N} \sum_{b, s, t}^{(a)} G_{is}^{(a)} h_{sa} h_{at} G_{tb}^{(a)} \frac{G_{ba} G_{ai}}{G_{aa}}. \nonumber \\
&\qquad\qquad + \frac{1}{(\lambda \gamma v_a - \tau)^2} \frac{F''(\frX^{(a)})}{N} (\frX - \frX^{(a)}) \sum_{b, s, t}^{(a)} G_{is}^{(a)} h_{sa} h_{at} G_{tb}^{(a)} G_{bi} + \caO(\Psi^5)\,.
\end{align}
The expectation of the first term in \eqref{eq:3311+3312+3313} has already been computed in~\eqref{eq:1411}, which gives us
\begin{align} \label{eq:3311}
&\E \left[ \frac{1}{(\lambda \gamma v_a - \tau)^2} \frac{F'(\frX^{(a)})}{N} \sum_{b, s, t}^{(a)} G_{is}^{(a)} h_{sa} h_{at} G_{tb}^{(a)} G_{bi}^{(a)} \right]\nonumber \\
&\qquad= \frac{\gamma^2}{(\lambda \gamma v_a - \tau)^2} \E\,[ F'(\frX) \caX_{33} ] - \frac{3 \gamma^4}{(\lambda \gamma v_a - \tau)^3} \E\,[ F'(\frX) \caX_{44}] - \frac{\gamma^4}{(\lambda \gamma v_a - \tau)^3} \E \,[ F''(\frX) \frX_{22} \caX_{33}] + \caO(\Psi^5) \,. 
\end{align}
The expansion of the second term in \eqref{eq:3311+3312+3313} is similar to the one in \eqref{eq:142} and we get
\begin{align} \label{eq:3312}
&\E\, \left[ \frac{1}{(\lambda \gamma v_a - \tau)^2} \frac{F'(\frX^{(a)})}{N} \sum_{b, s, t}^{(a)} G_{is}^{(a)} h_{sa} h_{at} G_{tb}^{(a)} \frac{G_{ba} G_{ai}}{G_{aa}} \right]\nonumber \\
&\qquad= \frac{\gamma^4}{(\lambda \gamma v_a - \tau)^3} \E\,[ F'(\frX) \caX_{44}'] + \frac{2 \gamma^4}{(\lambda \gamma v_a - \tau)^3} \E\,[ F'(\frX) \caX_{44}] + \caO(\Psi^5) \,.
\end{align}
The third term in \eqref{eq:3311+3312+3313} yields, using \eqref{eq:122},
\begin{align} \label{eq:3313}
&\E \left[ \frac{1}{(\lambda \gamma v_a - \tau)^2} \frac{F''(\frX^{(a)})}{N} (\frX - \frX^{(a)}) \sum_{b, s, t}^{(a)} G_{is}^{(a)} h_{sa} h_{at} G_{tb}^{(a)} G_{bi} \right] \\
&\qquad= \frac{\gamma^4}{(\lambda \gamma v_a - \tau)^3} \E\,[ F''(\frX) \frX_{22} \caX_{33} ] + \frac{2 \gamma^4}{(\lambda \gamma v_a - \tau)^3} \E \left[ F''(\frX) \frac{1}{N^3} \sum_{j, r, t, s} \left( \im \int_{E_1 -2}^{E_2 -2} \wt G_{jr} \wt G_{tj} \,\dd \wt y \right) G_{ir} G_{ts} G_{si} \right] + \caO(\Psi^5)\,.\nonumber
\end{align}
Thus, we obtain from \eqref{eq:3311+3312+3313}, \eqref{eq:3311}, \eqref{eq:3312} and \eqref{eq:3313} that
\begin{align} \label{eq:331}
&\E \left[ \frac{1}{(\lambda \gamma v_a - \tau)^2} \frac{F'(\frX)}{N} \sum_{b, s, t}^{(a)} G_{is}^{(a)} h_{sa} h_{at} G_{tb}^{(a)} G_{bi} \right]\nonumber \\
&\qquad= \frac{\gamma^2}{(\lambda \gamma v_a - \tau)^2} \E\,[ F'(\frX) \caX_{33} ] - \frac{\gamma^4}{(\lambda \gamma v_a - \tau)^3} \E\,[ F'(\frX) \caX_{44}] + \frac{\gamma^4}{(\lambda \gamma v_a - \tau)^3} \E\,[ F'(\frX) \caX_{44}']  \nonumber \\
&\qquad\qquad + \frac{2 \gamma^4}{(\lambda \gamma v_a - \tau)^3} \E\, \left[ F''(\frX) \frac{1}{N^3} \sum_{j, r, t, s} \left( \im \int_{E_1 -2}^{E_2 -2} \wt G_{jr} \wt G_{tj} \,\dd \wt y \right) G_{ir} G_{ts} G_{si} \right] + \caO(\Psi^5)\,. 
\end{align}

The second term in \eqref{eq:331+332} can be expanded as in \eqref{eq:1412} and we find that
\begin{align} \label{eq:332}
&\E \left[ \frac{2}{(\lambda \gamma v_a - \tau)^3} \frac{F'(\frX)}{N} \left( z + \sum_{p, q}^{(a)} h_{ap} G_{pq}^{(a)} h_{qa} - \tau \right) \sum_{b, s, t}^{(a)} G_{is}^{(a)} h_{sa} h_{at} G_{tb}^{(a)} G_{bi} \right] \nonumber \\
&\qquad= \frac{2 \gamma^2}{(\lambda \gamma v_a - \tau)^3} \E\,[ F'(\frX) \caX_{43}] + \frac{4 \gamma^4}{(\lambda \gamma v_a - \tau)^3} \E\,[ F'(\frX) \caX_{44}] + \caO(\Psi^5)\,.
\end{align}
We thus now have from \eqref{eq:331+332}, \eqref{eq:331} and \eqref{eq:332} that
\begin{align}
\E [ F'(\frX) \caX_{33} ] &= \E\,[ F'(\frX) \caX_{33} ] + \gamma^2 A_3 \big( 2 \E[ F'(\frX) \caX_{43}] + \gamma^4 \E\,[ F'(\frX) \caX_{44}'] + 3 \gamma^4 \E\,[ F'(\frX) \caX_{44}] ) \nonumber \\
&\qquad + 2 \gamma^4 A_3 \E\, \left[ F''(\frX) \frac{1}{N^3} \sum_{j, r, t, s} \left( \im \int_{E_1 -2}^{E_2 -2} \wt G_{jr} \wt G_{tj} \,\dd \wt y \right) G_{ir} G_{ts} G_{si} \right] + \caO(\Psi^5)\,, \nonumber
\end{align}
which implies the second ``optical theorem'' of this appendix,
\begin{align} \label{3-4 relation 3}
&2 \E\,[ F'(\frX) \caX_{43}] + \gamma^4 \E\,[ F'(\frX) \caX_{44}'] + 3 \gamma^4 \E\,[ F'(\frX) \caX_{44}] + 2 \gamma^2 \E\, \left[ F''(\frX) \frac{1}{N^3} \sum_{j, r, t, s} \left( \im \int_{E_1 -2}^{E_2 -2} \wt G_{jr} \wt G_{tj} \,\dd \wt y \right) G_{ir} G_{ts} G_{si} \right] \nonumber \\
&\qquad\qquad= \caO(\Psi^5)\,.
\end{align}

\subsection{Optical theorem from the other term in $X_3$}

We next expand the term
\begin{align} \label{eq:o1+o2+o3}
&\frac{1}{N} \sum_{b, j} F''(\frX) \left( \im \int_{E_1 -2}^{E_2 -2} G_{ja}' G_{bj}' \,\dd \wt y \right) G_{ia}G_{bi} \nonumber \\
&\quad= \frac{1}{(\lambda \gamma v_j - \tau)^2} \frac{F''(\frX)}{N} \sum_{b, j, s, t}^{(a)} \left( \im \int_{E_1 -2}^{E_2 -2} \wt G^{(a)}_{js} h_{sa} \wt G_{bj} \,\dd \wt y \right) G_{it}^{(a)} h_{ta} G_{bi}\nonumber \\
&\quad \quad+ \frac{1}{(\lambda \gamma v_j - \tau)^3} \frac{F''(\frX)}{N} \sum_{b, j, s, t}^{(a)} \left( \im \int_{E_1 -2}^{E_2 -2} \left( z + \sum_{p, q}^{(a)} h_{ap} \wt G^{(a)}_{pq} h_{qa} - \tau \right) \wt G^{(a)}_{js} h_{sa} \wt G_{bj} \,\dd \wt y \right) G_{it}^{(a)} h_{ta} G_{bi} \nonumber \\
&\quad \quad+ \frac{1}{(\lambda \gamma v_j - \tau)^3} \frac{F''(\frX)}{N} \sum_{b, j, s, t}^{(a)} \left( \im \int_{E_1 -2}^{E_2 -2} \wt G^{(a)}_{js} h_{sa} \wt G_{bj} \,\dd \wt y \right) \left( z + \sum_{p, q}^{(a)} h_{ap} G^{(a)}_{pq} h_{qa} - \tau \right) G_{it}^{(a)} h_{ta} G_{bi} + \caO(\Psi^5)\,. 
\end{align}
We further expand the first term in \eqref{eq:o1+o2+o3} to find
\begin{align} \label{eq:o11+o12+o13+o14}
&\frac{1}{(\lambda \gamma v_j - \tau)^2} \frac{F''(\frX)}{N} \sum_{b, j, s, t}^{(a)} \left( \im \int_{E_1 -2}^{E_2 -2} \wt G^{(a)}_{js} h_{sa} \wt G_{bj} \,\dd \wt y \right) G_{it}^{(a)} h_{ta} G_{bi} \nonumber \\
&\qquad= \frac{1}{(\lambda \gamma v_j - \tau)^2} \frac{F''(\frX^{(a)})}{N} \sum_{b, j, s, t}^{(a)} \left( \im \int_{E_1 -2}^{E_2 -2} \wt G^{(a)}_{js} h_{sa} \wt G^{(a)}_{bj} \,\dd \wt y \right) G_{it}^{(a)} h_{ta} G_{bi}^{(a)} \nonumber \\
&\qquad\qquad + \frac{1}{(\lambda \gamma v_j - \tau)^2} \frac{F''(\frX^{(a)})}{N} \sum_{b, j, s, t}^{(a)} \left( \im \int_{E_1 -2}^{E_2 -2} \wt G^{(a)}_{js} h_{sa} \wt G^{(a)}_{bj} \,\dd \wt y \right) G_{it}^{(a)} h_{ta} \frac{G_{ba} G_{ai}}{G_{aa}} \nonumber\\
&\qquad\qquad + \frac{1}{(\lambda \gamma v_j - \tau)^2} \frac{F''(\frX^{(a)})}{N} \sum_{b, j, s, t}^{(a)} \left( \im \int_{E_1 -2}^{E_2 -2} \wt G^{(a)}_{js} h_{sa} \frac{\wt G_{ba} G_{aj}}{G_{aa}} \,\dd \wt y \right) G_{it}^{(a)} h_{ta} G_{bi}^{(a)} \nonumber \\
&\qquad\qquad + \frac{1}{(\lambda \gamma v_j - \tau)^2} \frac{F'''(\frX^{(a)})}{N} (\frX - \frX^{(a)}) \sum_{b, j, s, t}^{(a)} \left( \im \int_{E_1 -2}^{E_2 -2} \wt G^{(a)}_{js} h_{sa} \wt G^{(a)}_{bj} \,\dd \wt y \right) G_{it}^{(a)} h_{ta} G_{bi}^{(a)} + \caO(\Psi^5)\,. 
\end{align}

Taking the partial expectation $\E_a$, we obtain
\begin{align}
&\E_a \left[ \frac{1}{(\lambda \gamma v_j - \tau)^2} \frac{F''(\frX^{(a)})}{N} \sum_{b, j, s, t}^{(a)} \left( \im \int_{E_1 -2}^{E_2 -2} \wt G^{(a)}_{js} h_{sa} \wt G^{(a)}_{bj} \,\dd \wt y \right) G_{it}^{(a)} h_{ta} G_{bi}^{(a)} \right] \nonumber \\
&\qquad= \frac{\gamma^2}{(\lambda \gamma v_j - \tau)^2} \frac{F''(\frX^{(a)})}{N^2} \sum_{b, j, s}^{(a)} \left( \im \int_{E_1 -2}^{E_2 -2} \wt G^{(a)}_{js} \wt G^{(a)}_{bj} \,\dd \wt y \right) G_{is}^{(a)} G_{bi}^{(a)}\nonumber \\
&\qquad= \frac{\gamma^2}{(\lambda \gamma v_j - \tau)^2} \frac{F''(\frX)}{N^2} \sum_{b, j, s} \left( \im \int_{E_1 -2}^{E_2 -2} \wt G_{js} \wt G_{bj} \,\dd \wt y \right) G_{is} G_{bi} \nonumber \\
&\qquad\qquad - \frac{\gamma^2}{(\lambda \gamma v_j - \tau)^2} \frac{F''(\frX^{(a)})}{N^2} \sum_{b, j, s}^{(a)} \left( \im \int_{E_1 -2}^{E_2 -2} \wt G^{(a)}_{js} \wt G^{(a)}_{bj} \,\dd \wt y \right) G_{is}^{(a)} \frac{G_{ba} G_{ai}}{G_{aa}} \nonumber\\
&\qquad\qquad - \frac{\gamma^2}{(\lambda \gamma v_j - \tau)^2} \frac{F''(\frX^{(a)})}{N^2} \sum_{b, j, s}^{(a)} \left( \im \int_{E_1 -2}^{E_2 -2} \wt G^{(a)}_{js} \wt G^{(a)}_{bj} \,\dd \wt y \right) \frac{G_{ia} G_{as}}{G_{aa}} G_{bi} \nonumber \\
&\qquad\qquad - \frac{\gamma^2}{(\lambda \gamma v_j - \tau)^2} \frac{F''(\frX^{(a)})}{N^2} \sum_{b, j, s}^{(a)} \left( \im \int_{E_1 -2}^{E_2 -2} \wt G^{(a)}_{js} \frac{\wt G_{ba} \wt G_{aj}}{\wt G_{aa}} \,\dd \wt y \right) G_{is} G_{bi} \nonumber \\
&\qquad\qquad - \frac{\gamma^2}{(\lambda \gamma v_j - \tau)^2} \frac{F''(\frX^{(a)})}{N^2} \sum_{b, j, s}^{(a)} \left( \im \int_{E_1 -2}^{E_2 -2} \frac{\wt G_{ja} \wt G_{as}}{\wt G_{aa}} \wt G_{bj} \,\dd \wt y \right) G_{is} G_{bi} \nonumber \\
&\qquad\qquad- \frac{\gamma^2}{(\lambda \gamma v_j - \tau)^2} \frac{F'''(\frX^{(a)})}{N^2} (\frX - \frX^{(a)}) \sum_{b, j, s} \left( \im \int_{E_1 -2}^{E_2 -2} \wt G_{js} \wt G_{bj} \,\dd \wt y \right) G_{is} G_{bi} + \caO(\Psi^5)\,.
\end{align}
Taking the full expectation, we find for the first term in \eqref{eq:o11+o12+o13+o14} that
\begin{align} \label{eq:o11}
&\E \left[ \frac{1}{(\lambda \gamma v_j - \tau)^2} \frac{F''(\frX^{(a)})}{N} \sum_{b, j, s, t}^{(a)} \left( \im \int_{E_1 -2}^{E_2 -2} \wt G^{(a)}_{js} h_{sa} \wt G^{(a)}_{bj} \,\dd \wt y \right) G_{it}^{(a)} h_{ta} G_{bi}^{(a)} \right] \nonumber \\
&\qquad= \frac{\gamma^2}{(\lambda \gamma v_j - \tau)^2} \E \left[ \frac{F''(\frX)}{N^2} \sum_{b, j, s} \left( \im \int_{E_1 -2}^{E_2 -2} \wt G_{js} \wt G_{bj} \,\dd \wt y \right) G_{is} G_{bi} \right] \nonumber \\
&\qquad\qquad - \frac{2 \gamma^4}{(\lambda \gamma v_j - \tau)^3} \E \left[ \frac{F''(\frX)}{N^3} \sum_{b, j, s, p} \left( \im \int_{E_1 -2}^{E_2 -2} \wt G_{js} \wt G_{bj} \,\dd \wt y \right) G_{is} G_{bp} G_{pi} \right] \nonumber\\
&\qquad\qquad - \frac{2 \gamma^4}{(\lambda \gamma v_j - \tau)^3} \E \left[ \frac{F''(\frX)}{N^3} \sum_{b, j, s, p} \left( \im \int_{E_1 -2}^{E_2 -2} \wt G_{js} \wt G_{bp} \wt G_{pj} \,\dd \wt y \right) G_{is} G_{bi}  \right] \nonumber \\
&\qquad\qquad - \frac{\gamma^4}{(\lambda \gamma v_j - \tau)^3} \E \left[ \frac{F'''(\frX)}{N^2} \frX_{22} \sum_{b, j, s} \left( \im \int_{E_1 -2}^{E_2 -2} \wt G_{js} \wt G_{bj} \,\dd \wt y \right) G_{is} G_{bi} \right] + \caO(\Psi^5)\,. 
\end{align}

We expand the other terms in \eqref{eq:o11+o12+o13+o14}. The second term becomes
\begin{align} \label{eq:o12}
&\E \left[ \frac{1}{(\lambda \gamma v_j - \tau)^2} \frac{F''(\frX^{(a)})}{N} \sum_{b, j, s, t}^{(a)} \left( \im \int_{E_1 -2}^{E_2 -2} \wt G^{(a)}_{js} h_{sa} \wt G^{(a)}_{bj} \,\dd \wt y \right) G_{it}^{(a)} h_{ta} \frac{G_{ba} G_{ai}}{G_{aa}} \right] \nonumber \\
&\qquad= \frac{2 \gamma^4}{(\lambda \gamma v_j - \tau)^3} \E \left[ \frac{F''(\frX)}{N^3} \sum_{b, j, s, p} \left( \im \int_{E_1 -2}^{E_2 -2} \wt G_{js} \wt G_{bj} \,\dd \wt y \right) G_{is} G_{bp} G_{pi} \right]\nonumber \\
&\qquad\qquad + \frac{\gamma^4}{(\lambda \gamma v_j - \tau)^3} \E \left[ \frac{F''(\frX)}{N^3} \sum_{b, j, s, t} \left( \im \int_{E_1 -2}^{E_2 -2} \wt G_{js} \wt G_{bj} \,\dd \wt y \right) G_{it} G_{bs} G_{ti} \right] + \caO(\Psi^5)\,.
\end{align}
Similarly,
\begin{align} \label{eq:o13}
&\E \left[ \frac{1}{(\lambda \gamma v_j - \tau)^2} \frac{F''(\frX^{(a)})}{N} \sum_{b, j, s, t}^{(a)} \left( \im \int_{E_1 -2}^{E_2 -2} \wt G^{(a)}_{js} h_{sa} \frac{\wt G_{ba} G_{aj}}{G_{aa}} \,\dd \wt y \right) G_{it}^{(a)} h_{ta} G_{bi}^{(a)} \right] \nonumber \\
&\qquad= \frac{2 \gamma^4}{(\lambda \gamma v_j - \tau)^3} \E \left[ \frac{F''(\frX)}{N^3} \sum_{b, j, s, p} \left( \im \int_{E_1 -2}^{E_2 -2} \wt G_{js} \wt G_{bp} \wt G_{pj} \,\dd \wt y \right) G_{is} G_{bi} \right]\nonumber \\
&\qquad\qquad + \frac{\gamma^4}{(\lambda \gamma v_j - \tau)^3} \E \left[ \frac{F''(\frX)}{N^3} \sum_{b, j, s, t} \left( \im \int_{E_1 -2}^{E_2 -2} \wt G_{js} \wt G_{bt} \wt G_{sj} \,\dd \wt y \right) G_{it} G_{bi} \right] + \caO(\Psi^5)\,. 
\end{align}
Finally, the fourth term in \eqref{eq:o11+o12+o13+o14} yields
\begin{align} \label{eq:o14}
&\E \left[ \frac{1}{(\lambda \gamma v_j - \tau)^2} \frac{F'''(\frX^{(a)})}{N} (\frX - \frX^{(a)}) \sum_{b, j, s, t}^{(a)} \left( \im \int_{E_1 -2}^{E_2 -2} \wt G^{(a)}_{js} h_{sa} \wt G^{(a)}_{bj} \,\dd \wt y \right) G_{it}^{(a)} h_{ta} G_{bi}^{(a)} \right] \nonumber \\
&\qquad= \frac{\gamma^4}{(\lambda \gamma v_j - \tau)^3} \E \left[ \frac{F'''(\frX)}{N^2} \frX_{22} \sum_{b, j, s} \left( \im \int_{E_1 -2}^{E_2 -2} \wt G_{js} \wt G_{bj} \,\dd \wt y \right) G_{is} G_{bi} \right]\nonumber \\
&\qquad\quad + \frac{2 \gamma^4}{(\lambda \gamma v_j - \tau)^3} \E \left[ \frac{F'''(\frX)}{N^2} \sum_{b, j, k, s, t} \left( \im \int_{E_1 -2}^{E_2 -2} \wt G_{js} \wt G_{bj} \,\dd \wt y \right) \left( \im \int_{E_1 -2}^{E_2 -2} \widehat G_{ks} \widehat G_{tk} \,\dd \widehat y \right) G_{is} G_{bi} \right] + \caO(\Psi^5)\,. 
\end{align}
Thus, from \eqref{eq:o11+o12+o13+o14}, \eqref{eq:o11}, \eqref{eq:o12}, \eqref{eq:o13} and \eqref{eq:o14}, we obtain 
\begin{align} \label{eq:o1}
&\E \left[ \frac{1}{(\lambda \gamma v_j - \tau)^2} \frac{F''(\frX)}{N} \sum_{b, j, s, t}^{(a)} \left( \im \int_{E_1 -2}^{E_2 -2} \wt G^{(a)}_{js} h_{sa} \wt G_{bj} \,\dd \wt y \right) G_{it}^{(a)} h_{ta} G_{bi} \right] \nonumber \\
&\qquad= \frac{\gamma^2}{(\lambda \gamma v_j - \tau)^2} \E \left[ \frac{F''(\frX)}{N^2} \sum_{b, j, s} \left( \im \int_{E_1 -2}^{E_2 -2} \wt G_{js} \wt G_{bj} \,\dd \wt y \right) G_{is} G_{bi} \right] \nonumber \\
&\quad\qquad + \frac{\gamma^4}{(\lambda \gamma v_j - \tau)^3} \E \left[ \frac{F''(\frX)}{N^3} \sum_{b, j, s, t} \left( \im \int_{E_1 -2}^{E_2 -2} \wt G_{js} \wt G_{bj} \,\dd \wt y \right) G_{it} G_{bs} G_{ti} \right]\nonumber \\
&\quad\qquad+ \frac{\gamma^4}{(\lambda \gamma v_j - \tau)^3} \E \left[ \frac{F''(\frX)}{N^3} \sum_{b, j, s, t} \left( \im \int_{E_1 -2}^{E_2 -2} \wt G_{js} \wt G_{bt} \wt G_{sj} \,\dd \wt y \right) G_{it} G_{bi} \right] \nonumber \\
&\quad\qquad + \frac{2 \gamma^4}{(\lambda \gamma v_j - \tau)^3} \E \left[ \frac{F'''(\frX)}{N^2} \sum_{b, j, k, s, t} \left( \im \int_{E_1 -2}^{E_2 -2} \wt G_{js} \wt G_{bj} \,\dd \wt y \right) \left( \im \int_{E_1 -2}^{E_2 -2} \widehat G_{ks} \widehat G_{tk} \,\dd \widehat y \right) G_{it} G_{bi} \right]  + \caO(\Psi^5)\,.
\end{align}
The other terms in \eqref{eq:o1+o2+o3} are of $\caO(\Psi^4)$ and we observe that
\begin{align} \label{eq:o2}
&\E \left[ \frac{1}{(\lambda \gamma v_j - \tau)^3} \frac{F''(\frX)}{N} \sum_{b, j, s, t}^{(a)} \left( \im \int_{E_1 -2}^{E_2 -2} \left( z + \sum_{p, q}^{(a)} h_{ap} \wt G^{(a)}_{pq} h_{qa} - \tau \right) \wt G^{(a)}_{js} h_{sa} \wt G_{bj} \,\dd \wt y \right) G_{it}^{(a)} h_{ta} G_{bi} \right] \nonumber \\
&\qquad= \frac{\gamma^2}{(\lambda \gamma v_j - \tau)^3} \E \left[ \frac{F''(\frX)}{N^2} \sum_{b, j, s} \left( \im \int_{E_1 -2}^{E_2 -2} ( z + \gamma^2 \wt m - \tau ) \wt G_{js} \wt G_{bj} \,\dd \wt y \right) G_{is} G_{bi} \right]\nonumber \\
&\qquad\qquad + \frac{2 \gamma^4}{(\lambda \gamma v_j - \tau)^3} \E \left[ \frac{F''(\frX)}{N^3} \sum_{b, j, s, t} \left( \im \int_{E_1 -2}^{E_2 -2} \wt G_{st} \wt G_{js} \wt G_{bj} \,\dd \wt y \right) G_{it} G_{bi} \right] + \caO(\Psi^5)
\end{align}
and, similarly,
\begin{align} \label{eq:o3}
&\E \left[ \frac{1}{(\lambda \gamma v_j - \tau)^3} \frac{F''(\frX)}{N} \sum_{b, j, s, t}^{(a)} \left( \im \int_{E_1 -2}^{E_2 -2} \wt G^{(a)}_{js} h_{sa} \wt G_{bj} \,\dd \wt y \right) \left( z + \sum_{p, q}^{(a)} h_{ap} G^{(a)}_{pq} h_{qa} - \tau \right) G_{it}^{(a)} h_{ta} G_{bi} \right] \nonumber \\
&\qquad= \frac{\gamma^2}{(\lambda \gamma v_j - \tau)^3} \E \left[ \frac{F''(\frX)}{N^2} \sum_{b, j, s} \left( \im \int_{E_1 -2}^{E_2 -2} \wt G_{js} \wt G_{bj} \,\dd \wt y \right) ( z + \gamma^2 m - \tau ) G_{is} G_{bi} \right]\nonumber \\
&\qquad\qquad + \frac{2 \gamma^4}{(\lambda \gamma v_j - \tau)^3} \E \left[ \frac{F''(\frX)}{N^3} \sum_{b, j, s, t} \left( \im \int_{E_1 -2}^{E_2 -2} \wt G_{js} \wt G_{bj} \,\dd \wt y \right) G_{st} G_{it} G_{bi} \right] + \caO(\Psi^5)\,. 
\end{align}
Thus, from \eqref{eq:o1+o2+o3}, \eqref{eq:o1}, \eqref{eq:o2} and \eqref{eq:o3}, we conclude that
\begin{align}
&\E \left[ \frac{1}{N} \sum_{b, j} F''(\frX) \left( \im \int_{E_1 -2}^{E_2 -2} G_{ja}' G_{bj}' \,\dd \wt y \right) G_{ia}G_{bi} \right] \nonumber \\
&\qquad= \frac{\gamma^2}{(\lambda \gamma v_j - \tau)^2} \E \left[ \frac{F''(\frX)}{N^2} \sum_{b, j, s} \left( \im \int_{E_1 -2}^{E_2 -2} \wt G_{js} \wt G_{bj} \,\dd \wt y \right) G_{is} G_{bi} \right] \nonumber \\
&\qquad\qquad + \frac{\gamma^4}{(\lambda \gamma v_j - \tau)^3} \E \left[ \frac{F''(\frX)}{N^3} \sum_{b, j, s, t} \left( \im \int_{E_1 -2}^{E_2 -2} \wt G_{js} \wt G_{bj} \,\dd \wt y \right) G_{it} G_{bs} G_{ti} \right]\nonumber \\
&\qquad\qquad + \frac{\gamma^4}{(\lambda \gamma v_j - \tau)^3} \E \left[ \frac{F''(\frX)}{N^3} \sum_{b, j, s, t} \left( \im \int_{E_1 -2}^{E_2 -2} \wt G_{js} \wt G_{bt} \wt G_{sj} \,\dd \wt y \right) G_{it} G_{bi} \right] \nonumber \\
&\qquad\qquad + \frac{2 \gamma^4}{(\lambda \gamma v_j - \tau)^3} \E \left[ \frac{F'''(\frX)}{N^2} \sum_{b, j, k, s, t} \left( \im \int_{E_1 -2}^{E_2 -2} \wt G_{js} \wt G_{bj} \,\dd \wt y \right) \left( \im \int_{E_1 -2}^{E_2 -2} \widehat G_{ks} \widehat G_{tk} \,\dd \widehat y \right) G_{it} G_{bi} \right] \nonumber \\
&\qquad\qquad + \frac{\gamma^2}{(\lambda \gamma v_j - \tau)^3} \E \left[ \frac{F''(\frX)}{N^2} \sum_{b, j, s} \left( \im \int_{E_1 -2}^{E_2 -2} ( z + \gamma^2 \wt m - \tau ) \wt G_{js} \wt G_{bj} \,\dd \wt y \right) G_{is} G_{bi} \right] \nonumber \\
&\qquad\qquad + \frac{\gamma^2}{(\lambda \gamma v_j - \tau)^3} \E \left[ \frac{F''(\frX)}{N^2} \sum_{b, j, s} \left( \im \int_{E_1 -2}^{E_2 -2} \wt G_{js} \wt G_{bj} \,\dd \wt y \right) ( z + \gamma^2 m - \tau ) G_{is} G_{bi} \right] \nonumber \\
&\qquad\qquad + \frac{2 \gamma^4}{(\lambda \gamma v_j - \tau)^3} \E \left[ \frac{F''(\frX)}{N^3} \sum_{b, j, s, t} \left( \im \int_{E_1 -2}^{E_2 -2} \wt G_{st} \wt G_{js} \wt G_{bj} \,\dd \wt y \right) G_{it} G_{bi} \right] \nonumber \\
&\qquad\qquad + \frac{2 \gamma^4}{(\lambda \gamma v_j - \tau)^3} \E \left[ \frac{F''(\frX)}{N^3} \sum_{b, j, s, t} \left( \im \int_{E_1 -2}^{E_2 -2} \wt G_{js} \wt G_{bj} \,\dd \wt y \right) G_{st} G_{it} G_{bi} \right] + \caO(\Psi^5)\,, 
\end{align}
which yields, after summing over the index $j$, the third ``optical theorem'' of this appendix,
\begin{align} \label{3-4 relation 4}
&\gamma^2 \E \left[ \frac{F''(\frX)}{N^3} \sum_{b, j, s, t} \left( \im \int_{E_1 -2}^{E_2 -2} \wt G_{js} \wt G_{bj} \,\dd \wt y \right) G_{it} G_{bs} G_{ti} \right]\nonumber \\
&\qquad + \gamma^2 \E \left[ \frac{F''(\frX)}{N^3} \sum_{b, j, s, t} \left( \im \int_{E_1 -2}^{E_2 -2} \wt G_{js} \wt G_{bt} \wt G_{sj} \,\dd \wt y \right) G_{it} G_{bi} \right] \nonumber \\
&\qquad + 2 \gamma^2 \E \left[ \frac{F'''(\frX)}{N^2} \sum_{b, j, k, s, t} \left( \im \int_{E_1 -2}^{E_2 -2} \wt G_{js} \wt G_{bj} \,\dd \wt y \right) \left( \im \int_{E_1 -2}^{E_2 -2} \widehat G_{ks} \widehat G_{tk} \,\dd \widehat y \right) G_{it} G_{bi} \right] \nonumber \\
&\qquad + \E \left[ \frac{F''(\frX)}{N^2} \sum_{b, j, s} \left( \im \int_{E_1 -2}^{E_2 -2} ( z + \gamma^2 \wt m - \tau ) \wt G_{js} \wt G_{bj} \,\dd \wt y \right) G_{is} G_{bi} \right] \nonumber \\
&\qquad + \E \left[ \frac{F''(\frX)}{N^2} \sum_{b, j, s} \left( \im \int_{E_1 -2}^{E_2 -2} \wt G_{js} \wt G_{bj} \,\dd \wt y \right) ( z + \gamma^2 m - \tau ) G_{is} G_{bi} \right] \nonumber \\
&\qquad + 2 \gamma^2 \E \left[ \frac{F''(\frX)}{N^3} \sum_{b, j, s, t} \left( \im \int_{E_1 -2}^{E_2 -2} \wt G_{st} \wt G_{js} \wt G_{bj} \,\dd \wt y \right) G_{it} G_{bi} \right] \nonumber \\
&\qquad + 2 \gamma^2 \E \left[ \frac{F''(\frX)}{N^3} \sum_{b, j, s, t} \left( \im \int_{E_1 -2}^{E_2 -2} \wt G_{js} \wt G_{bj} \,\dd \wt y \right) G_{st} G_{it} G_{bi} \right] \nonumber \\
&= \caO(\Psi^5)\,. 
\end{align}

\subsection{Simplification of \eqref{eq:second} and Proof of Lemma \ref{lem:G_ii estimate}}

Recall that $X_4$ is the sum (over the index $i$) of the terms of order $\caO(\Psi^4)$ given on the right side of~\eqref{eq:first}. Subtracting twice~\eqref{3-4 relation 2} and~$(2 \gamma^2)$-times \eqref{3-4 relation 3} and~\eqref{3-4 relation 4} from~\eqref{3-4 relation}, we obtain
\begin{align}\label{le second_final before}
& \sum_i \left( \E\,[ F'(\frX) \caX_{42} ] + \gamma^4 \E\, [F'(\frX) \caX_{44}'] + 4 \gamma^4 \E\,[ F'(\frX) \caX_{44} ] \right) \nonumber \\
&\qquad\qquad + 2 \gamma^4 \E\, \left[ F''(\frX) \frac{1}{N^3} \sum_{i, j, p, q, s} \left( \im \int_{E_1 -2}^{E_2 -2} \wt G_{jp} \wt G_{qj} \,\dd \wt y \right) G_{pq} G_{is} G_{si} \right]\nonumber \\
&\qquad= N \E\, X_4 + \frac{1}{\gamma^2 A_3} \frac{1}{N} \sum_i \frac{1}{(\lambda \gamma v_i - \tau)^2} \E\, \left[ F'(\frX) (z+ \gamma^2 m - \tau) \right] \nonumber \\
&\qquad= -\frac{2 A_3}{A_4} N \E\, X_3 - \frac{1}{\gamma^2 A_4} \frac{1}{N} \sum_i \E\,[ F'(\frX) G_{ii}^2] + \frac{1}{\gamma^4 A_3} \E \left[ F'(\frX) (z+ \gamma^2 m - \tau) \right] \nonumber \\
&\qquad= -\frac{2 A_3}{A_4} N \E\, X_3 - \frac{A_2}{\gamma^2 A_4} \E [F'(\frX)] + \left( \frac{1}{\gamma^4 A_3} - \frac{2 A_3}{\gamma^2 A_4} \right) \E\,[ (z + \gamma^2 m - \tau) F'(\frX)] \,,
\end{align}
where we also used \eqref{G_ii^2}. Plugging~\eqref{le second_final before} into \eqref{eq:second_2}, we conclude that
\begin{align} \label{eq:second_final}
&\frac{1}{N^2} \sum_{i, a, b} \big(\E \,[F'(\frX) G_{ia}G_{ab}G_{bi}] + \E\, [F'(\frX) G_{ia}G_{bb}G_{ai}] \big)  + \frac{1}{N^2} \sum_{i, a, b, j} \E \left[ F''(\frX) \left( \im \int_{E_1 -2}^{E_2 -2} G_{ja}' G_{bj}' \,\dd \wt y \right) G_{ia}G_{bi} \right] \nonumber \\
&\qquad= N A_1 \E\,[ X_2 ]+ \left( \gamma^{-2} - \frac{2 A_3^2}{A_4} \right) N \E\,[ X_3] - \frac{A_2 A_3}{\gamma^2 A_4} \E\, [F'(\frX)] \nonumber\\ 
&\qquad\qquad\qquad+ \left( \gamma^{-4} - \frac{2 A_3^2}{\gamma^2 A_4} \right) \E\,[ (z + \gamma^2 m - \tau) F'(\frX)] + \caO(\Psi^2)\,.
\end{align}
Finally, combining~\eqref{eq:first_final} and~\eqref{eq:second_final}, we obtain the estimate~\eqref{G_ii estimate}. This completes the proof of Lemma \ref{lem:G_ii estimate}.

\section{} \label{appendix II}

In this last appendix, we prove Lemma~\ref{lem:Q'' estimate}. Recall from~\eqref{definition of the Q} that we denote $R(w_{ab})=F'(\frX)G_{ia}G_{bi}$ and that we assumed in Lemma~\ref{lem:Q'' estimate} that $i\not=a\not=b\not=j$.
\begin{proof}[Proof of Lemma \ref{lem:Q'' estimate}]
In~\eqref{le definition of le X} we defined
\begin{align*}
\frX \deq N \int_{E_1}^{E_2} \im m( \widehat{L}_++x -2 + \ii \eta) \,\dd x = N \int_{E_1 -2}^{E_2 -2} \im m( \widehat{L}_++\wt y  + \ii \eta) \,\dd \wt y\,.
\end{align*}
Also recall that we have abbreviated $\wt G \equiv G(\widehat{L}_+ + \wt y + \ii \eta)$. From definition above we see that
\begin{align*}
\frac{\partial \frX}{\partial w_{ab}} = 2 \im \sum_j \int_{E_1 -2}^{E_2 -2} \wt G_{ja} \wt G_{bj} \,\dd \wt y = \caO_{\Xi}(\Psi)\,.
\end{align*}
Similarly, we can also show that $\frac{\partial^2 \frX}{\partial w_{ab}^2} = \caO_{\Xi}(\Psi)$.

 We next consider
\begin{align} \label{Q'' expand}
\frac{\partial^2 R (w_{ab})}{\partial w_{ab}^2} &= F'''(\frX) \left(\frac{\partial \frX}{\partial w_{ab}} \right)^2 G_{ia} G_{bi} + F''(\frX) \frac{\partial^2 \frX}{\partial w_{ab}^2} G_{ia} G_{bi} + F''(\frX) \frac{\partial \frX}{\partial w_{ab}} \frac{\partial (G_{ia} G_{bi})}{\partial w_{ab}}
\nonumber\\&\qquad + F'(\frX) \frac{\partial^2 (G_{ia} G_{bi})}{\partial w_{ab}^2}\,.
\end{align}
Clearly the first two terms are of $\caO_{\Xi}(\Psi^3)$. Since
\begin{align*}
\frac{\partial (G_{ia} G_{bi})}{\partial w_{ab}} = 2 G_{ia} G_{ba} G_{bi} + G_{ib} G_{aa} G_{bi} + G_{ia} G_{bb} G_{ai} = \caO_{\Xi}(\Psi^2)\,,
\end{align*}
the third term in \eqref{Q'' expand} is also $\caO_{\Xi}(\Psi^2)$. Finally, in $(\partial^2 / \partial w_{ab}^2) (G_{ia} G_{bi})$ every term contains at least three off-diagonal terms except $G_{ia} G_{bb} G_{aa} G_{bi}$ and $G_{ib} G_{aa} G_{bb} G_{ai}$. We first observe that
\begin{align*}
\E \,[ F'(\frX) G_{ia} G_{bb} G_{aa} G_{bi} ] = \frac{1}{\lambda \gamma v_a - \tau} \frac{1}{\lambda \gamma v_b - \tau} \E\, [ F'(\frX) G_{ia} G_{bi} ]+ \caO_\Xi(\Psi^3)\,.
\end{align*}
Using the resolvent formula~\eqref{twosided}, we find that
\begin{align*}
G_{ia} = -G_{aa} G_{ii}^{(a)} \left(h_{ia} - \sum_{p, q}^{(a)} h_{ip} G_{pq}^{(a)} h_{qa}\right) = \frac{1}{\lambda \gamma v_a - \tau} \frac{1}{\lambda \gamma v_i - \tau} \left(h_{ia} - \sum_{p, q}^{(a)} h_{ip} G_{pq}^{(a)} h_{qa}\right) + \caO_{\Xi}(\Psi^2)\,.
\end{align*}
Let
\begin{align*}
\frX^{(a)} \deq N \int_{E_1}^{E_2} \im m^{(a)}(\widehat{L}_++x   -2 + \ii \eta) \,\dd x = N \int_{E_1 -2}^{E_2 -2} \im m^{(a)}(\widehat{L}_++y   + \ii \eta) \,\dd y \,,
\end{align*}
and note that $\frX - \frX^{(a)} = \caO_{\Xi}(\Psi)$. Thus, we obtain
\begin{align}\label{le equation in last appendix}
\E_a\, [ F'(\frX) G_{ia} G_{bi} ] &= \E_a\, [ F'(\frX^{(a)}) G_{ia} G_{bi}^{(a)} ] + \caO_{\Xi}(\Psi^3)\nonumber \\
&= \frac{1}{\lambda \gamma v_a - \tau} \frac{1}{\lambda \gamma v_i - \tau} \E_a\left[ F'(\frX^{(a)}) G_{bi}^{(a)} (h_{ia} - \sum_{p, q}^{(a)} h_{ip} G_{pq}^{(a)} h_{qa})  \right] + \caO_{\Xi}(\Psi^3)\,. 
\end{align}
 Since $i\not=a$, the first term on the right side of~\eqref{le equation in last appendix} vanishes. Therefore we have
\begin{align*}
\E \,[ F'(\frX) G_{ia} G_{bi} ] = \E\, \E_a \,[ F'(\frX) G_{ia} G_{bi} ] = \caO_{\Xi}(\Psi^3)\,.
\end{align*}
This completes the proof of the Lemma~\ref{lem:Q'' estimate} and also concludes this last appendix.
\end{proof}

\end{appendix}

\end{document}